\documentclass[a4paper,french]{article}

\usepackage[utf8]{inputenc}%caractères accentués
\usepackage[T1]{fontenc}
\usepackage{babel}
\usepackage[babel]{csquotes}

\usepackage{amsmath}%substack, displaystyle
\usepackage{amsfonts}%mathbb, mathfrak, mathcal
\usepackage{amsthm}%newtheorem, proof
\usepackage{bbm}%fonction indicatrice
\usepackage{amssymb}%sphericalangle

\usepackage{float}%choisir l'emplacement des images
\usepackage{tikz}%dessins
\usetikzlibrary{shapes.misc}
\tikzset{cross/.style={cross out, draw, 
         minimum size=2*(#1-\pgflinewidth), 
         inner sep=0pt, outer sep=0pt}}

\usepackage{url}%tilde dans les url

\usepackage{hyperref}

\newcommand{\nsnote}[1]{}
\newcommand{\wknote}[1]{}
\newcommand{\comm}[1]{}

\newtheorem{theorem}{Théorème}[section]
\newtheorem{proposition}[theorem]{Proposition}
\newtheorem{lemma}[theorem]{Lemme}
\newtheorem{corollary}[theorem]{Corollaire}

\newtheorem{thm}{Théorème}
\newtheorem{prop}{Proposition}
\newtheorem*{cor}{Corollaire}

\theoremstyle{definition}

\newtheorem*{remark}{Remarque}

\DeclareMathOperator{\Tr}{Tr}

\DeclareMathOperator{\GL}{GL}
\DeclareMathOperator{\SL}{SL}
\DeclareMathOperator{\PSL}{PSL}
\DeclareMathOperator{\PGL}{PGL}
\DeclareMathOperator{\End}{End}

\DeclareMathOperator{\Supp}{Supp}
\DeclareMathOperator{\diag}{diag}

\DeclareMathOperator{\Stab}{Stab}
\DeclareMathOperator{\Aut}{Aut}
\DeclareMathOperator{\stab}{\mathfrak{stab}}

\DeclareMathOperator{\Ad}{Ad}
\DeclareMathOperator{\ad}{ad}

\DeclareMathOperator{\Hom}{Hom}

\DeclareMathOperator{\Id}{Id}

\DeclareMathOperator{\card}{card}

\newcommand{\ssl}{\mathfrak{sl}}
\newcommand{\gl}{\mathfrak{gl}}

\newcommand{\bracket}[1]{\langle#1\rangle}
\newcommand{\dd}{\mathrm{d}}

\newcommand{\ssum}{\Sigma}

\newcommand{\pgcd}{\mathrm{pgcd}}

\newcommand{\R}{\mathbb{R}}

\newcommand{\C}{\mathbb{C}}
\newcommand{\Z}{\mathbb{Z}}
\newcommand{\ZZ}{\widehat{\mathbb{Z}}}
\newcommand{\N}{\mathbb{N}}
\newcommand{\Q}{\mathbb{Q}}
\newcommand{\E}{\mathbb{E}}
\newcommand{\PP}{\mathbb{P}}

\newcommand{\T}{\mathbb{T}}

\newcommand{\cI}{\mathcal{I}}

\newcommand{\cF}{\mathcal{F}}

\newcommand{\cP}{\mathcal{P}}
\newcommand{\cG}{\mathcal{G}}

\newcommand{\g}{\mathfrak{g}}
\newcommand{\h}{\mathfrak{h}}

\newcommand{\sr}{\mathsf{r}}

\newcommand{\eps}{\varepsilon}

\newcommand{\abs}[1]{\lvert#1\rvert}    % valeur absolue
\newcommand{\norm}[1]{\lVert#1\rVert}   % norme
\newcommand{\normHS}[1]{\lVert#1\rVert_{\mathrm{HS}}}   % norme Hilbert-Schmidt
\newcommand{\normop}[1]{\lVert#1\rVert_{\mathrm{op}}}   % norme d'opérateur
\newcommand{\floor}[1]{\left\lfloor#1\right\rfloor}   % partie entière, taille décidée par latex

\newcommand\dash{\nobreakdash-\hspace{0pt}} % tiret qui ne saute pas de ligne

\title{Trou spectral dans les groupes simples}
\author{Weikun He et Nicolas de Saxcé}

\begin{document}

\maketitle

\begin{abstract}
Nous montrons la propriété du trou spectral pour la famille des graphes de Cayley obtenus par réduction modulo $q$ d'un sous-groupe de $\SL_d(\Z)$ dont l'adhérence de Zariski est un $\Q$-groupe simple.
\end{abstract}

\tableofcontents

\section*{Introduction}

Dans $\SL_d(\Z)$, on considère un sous-groupe $\Gamma$, dont on note $\Omega$ l'adhérence dans le groupe profini $\SL_d(\ZZ)$ des matrices de déterminant 1 à coefficients dans l'anneau profini $\ZZ=\underleftarrow{\lim}\ \Z/q\Z$.
Le groupe $\Omega$ est un groupe topologique compact, que l'on munit naturellement de sa probabilité de Haar, notée $m_\Omega$.
On s'intéresse à l'action de $\Gamma$ sur l'espace $L^2(\Omega)$ définie par
\begin{equation}\label{tg}
\forall g \in \Gamma,\; \forall f \in L^2(\Omega),\; \forall x \in \Omega, \quad T_gf(x) = f(g^{-1}x).
\end{equation}
Comme $\Gamma$ est dense dans $\Omega$, toute fonction $\Gamma$-invariante dans $L^2(\Omega)$ est constante; on dit que l'action $\Gamma \curvearrowright \Omega$ est \emph{ergodique}.
De manière équivalente, dans l'espace
\[
L_0^2(\Omega) = \bigl\{\, \xi \in L^2(\Omega) \mid \int_\Omega \xi \dd m_\Omega = 0\,\bigr\},
\]
le groupe $\Gamma$ n'a pas de vecteur invariant non nul.
Nous dirons que l'action $\Gamma \curvearrowright \Omega$ a un \emph{trou spectral} si $\Gamma$ n'a pas de vecteur \emph{presque} invariant dans $L_0^2(\Omega)$, i.e. s'il existe une constante $\eps > 0$ et une partie finie $S \subset \Gamma$ tels que 
\[
\forall v \in V_\rho,\; \exists g \in S,\quad \norm{\rho(g)v - v} \geq \eps\norm{v}.
\]
Cette propriété implique bien sûr l'ergodicité de l'action, mais est en fait strictement plus forte.
Le but principal de cet article est de démontrer le théorème suivant.

\begin{thm}[Propriété du trou spectral]
\label{trouspectral}
Soit $\Gamma$ un sous-groupe de $\SL_d(\Z)$ et $\Omega$ son adhérence dans $\SL_d(\ZZ)$.
Supposons que l'adhérence de Zariski de $\Gamma$ soit un groupe algébrique simple, alors l'action de $\Gamma$ sur $\Omega$ admet un trou spectral.
\end{thm}

De façon plus concrète, on peut comprendre cet énoncé en termes des graphes de Cayley des projections de $\Gamma$ modulo $q$, où $q$ est un entier naturel non nul.
Pour cela, rappelons qu'étant donné un entier $k\geq 2$, une famille de graphes $(\cG_q)_{q\in\N^*}$ est dite \emph{famille d'expanseurs} s'il existe une constante $c>0$ tels que pour tout $q\in\N^*$, pour toute partie $X \subset \cG_q$ telle que $\abs{X}\leq \frac{\abs{\cG_q}}{2}$, on ait
\[
\abs{\partial X} \geq c\abs{X}
\]
où $\partial X = \{\, y \in \cG_q \setminus X \mid \exists x \in X,\ x\leftrightarrow y\}$ est la frontière de $X$ dans le graphe $\cG_q$.
On renvoie au livre de Lubotzky \cite{Lubotzky} pour une introduction à la théorie des graphes expanseurs.
Ci-dessous, on note $\pi_q \colon \SL_d(\Z) \to \SL_d(\Z/q\Z)$ la réduction modulo $q$.
Le théorème~\ref{trouspectral} est équivalent à l'énoncé suivant.

\begin{thm}[Graphes de Cayley expanseurs]
\label{expanseur}
Soit $S$ une partie symétrique finie de $\SL_d(\Z)$ et $\Gamma$ le sous-groupe engendré par $S$.
On suppose que l'adhérence de Zariski de $\Gamma$ dans $\SL_d$ est un groupe simple.
Alors, la famille de graphes de Cayley $\cG(\pi_q(\Gamma),\pi_q(S))_{q\in\N^*}$ forme une famille d'expanseurs.
\end{thm}

Un corollaire notable de ce théorème est une borne logarithmique sur le diamètre de ces graphes de Cayley.

\begin{cor}[Diamètre des graphes de Cayley]
\label{diametre}
Sous les hypothèses du théorème~\ref{expanseur}, il existe une constante $C\geq 0$ telle que pour tout $q\in\N^*$,
\[
\mathrm{diam}\, \cG(\pi_q(\Gamma),\pi_q(S)) \leq C \log q.
\]
\end{cor}

On peut enfin voir le théorème~\ref{trouspectral} comme une propriété forte d'équidistribution des marches aléatoires sur $\Omega$ engendrées par des éléments de $\Gamma$.
Soit $\mu$ une probabilité dont le support engendre le groupe $\Gamma$.
On s'intéresse à la marche aléatoire $(x_n)_{n\geq 1}$ sur $\Omega$ définie par
\[
\forall n\geq 1,\quad x_n = g_n\dots g_1,
\]
où $(g_n)_{n\geq 1}$ est une suite de variables aléatoires indépendantes identiquement distribuées de loi $\mu$ sur $\Gamma$.
Si $\mu$ est symétrique cette marche aléatoire $(x_n)_{n\geq 1}$ converge en loi vers la mesure de Haar sur $\Omega$:
\[
\forall f\in C(\Omega),\quad \lim_{n\to\infty}\E[f(x_n)] = \int_\Omega f(x)\,\dd m_\Omega(x).
\] 
%Il est clair que la marche est adaptée, car $\Gamma$ est dense dans $\Omega$, par définition.
%L'hypothèse d'apériodicité n'est pas nécessaire car $\mu$ est symétrique.
\comm{
\begin{remark}
On dit qu'une mesure dont le support engendre un sous-groupe dense de $\Omega$ est \emph{adaptée}.
L'hypothèse de symétrie n'est pas vraiment nécessaire pour obtenir la convergence ci-dessus, il suffit en fait de supposer que $\mu$ est \emph{apériodique} sur $\Omega$, i.e. que son support n'est pas inclus dans une classe à gauche d'un sous-groupe fermé distingué dans $\Omega$.
Si $\check{\mu}$ est l'image de $\mu$ par l'application $g\mapsto g^{-1}$, alors $\mu$ est apériodique si et seulement si le support de $\check{\mu}*\mu$ engendre $\Omega$ topologiquement.
\end{remark}
}

Pour $g\in\Omega$, l'opérateur $T_g$ sur l'espace $L^2(\Omega)$ a été définit en \eqref{tg} ci-dessus.
Plus généralement, on définit \emph{l'opérateur de convolution} associé à la probabilité $\mu$ par la formule
\[
T_\mu = \int_G T_g\,\dd\mu(g),
\]
et on note $T_\mu^0$ sa restriction au sous-espace $L^2_0(\Omega)$ des fonctions de moyenne nulle.
Le résultat d'équidistribution mentionné ci-dessus montre que la suite d'opérateurs $((T^0_\mu)^n)_{n\geq 1}$ converge simplement vers $0$.
Dire que l'action $\Gamma\curvearrowright\Omega$ admet un trou spectral revient à dire que cette convergence a lieu à vitesse exponentielle.

\begin{thm}[Rayon spectral des marches aléatoires]
\label{mesures}
Soit $\mu$ une probabilité symétrique sur $\SL_d(\Z)$, $\Gamma$ le sous-groupe engendré par le support de $\mu$, et $\Omega$ l'adhérence de $\Gamma$ dans $\SL_d(\ZZ)$.
Si l'adhérence de Zariski de $\Gamma$ dans $\SL_d$ est simple, alors 
\[
\normop{T^0_\mu} < 1.
\]
\end{thm}

\comm{
\begin{remark}
Si le support de $\check{\mu}*\mu$ n'engendre pas un sous-groupe dense dans $\Omega$, le rayon spectral de $T^0_\mu$ n'est pas strictement inférieur à $1$.
Et cela peut arriver.
Pour s'en convaincre, notons $p\geq 2$ un nombre premier quelconque, $g_0=\begin{pmatrix} 1 & 1\\ 0 & 1\end{pmatrix}$, et prenons une mesure $\mu$ supportée par 
\[ g_0U_p = \{g\in\SL_2(\Z)\ |\ g\equiv g_0  \mod p\}.\]
Comme $g_0^{-1}\Supp\mu$ contient le premier sous-groupe de congruence $U_p$, le sous-groupe engendré par $\mu$ est Zariski dense.
Mais les puissances de convolution $\mu^n$ ne s'équidistribuent pas vers la mesure de Haar sur $\Omega$, car à chaque étape $n$, la mesure $\mu^n$ est supportée par $g_0^nU_p$.
En fait, $\Gamma/\Gamma_p\simeq\Z/p\Z$ et la marche aléatoire induite par $\mu$ sur ce quotient est cyclique.
\end{remark}
}

Il n'est pas difficile de montrer que les trois théorèmes énoncés ci-dessus sont équivalents.
Le but du présent article est de donner une démonstration de ces théorèmes.
Nous commencerons par montrer le théorème~\ref{mesures}, avant d'en déduire les théorèmes~\ref{trouspectral} et \ref{expanseur}.

\bigskip

Pour la démonstration, nous suivrons la stratégie proposée par Bourgain et Varjú \cite{bv} dans le cas où l'adhérence de Zariski de $\Gamma$ est égale à $\SL_d$ tout entier.
Cette approche se fonde d'ailleurs sur la méthode développée par Bourgain et Gamburd dans la série d'articles \cite{bg_su2,bg_sud,bg0,bg1,bg2} pour montrer la propriété du trou spectral dans certains groupes compacts simples.
La méthode de Bourgain et Gamburd consiste à ramener la propriété du trou spectral à une propriété d'expansion combinatoire dans les groupes.
Dans leur premier article \cite{bg0} sur le sujet, dédié à l'étude des familles de graphes de Cayley de $\SL_2(\Z/p\Z)$ obtenus par projection modulo un nombre premier $p$ d'un système de générateurs fixé dans $\SL_2(\Z)$, cette propriété combinatoire provenait des travaux de Helfgott \cite{helfgott}.
Pour nous, le résultat prend la forme de la proposition ci-dessous.

\begin{prop}
\label{global}
Soit $\mu$ une probabilité symétrique à support fini sur $\SL_d(\Z)$.
On suppose que l'adhérence de Zariski du sous-groupe $\Gamma$ engendré par le support de $\mu$ est simple, connexe, et simplement connexe.
Alors, pour tout $\tau > 0$, il existe deux constantes $\eps > 0$ et $C \geq 1$ telles que pour tout $q \in \N^*$ suffisamment grand, pour tout $n \geq C \log q$, et toute partie symétrique $A \subset \Omega$, si $\mu^{*n}(A) \geq q^{-\eps}$, alors l'ensemble produit $\Pi_CA=\{a_1\dots a_C\ ;\ a_1,\dots,a_C\in A\}$ vérifie
\[N(\Pi_C A,q) \geq q^{-\tau} N(\Omega,q).\]
\end{prop}

Cette proposition constitue le cœur de l'article.
Sa démonstration, donnée dans les parties \ref{sec:sqf}, \ref{sec:torus} et \ref{sec:croissance}, reprend les idées de Bourgain et Varjú \cite{bv}, et utilise les résultats d'équidistribution quantitative des marches aléatoires linéaires sur le tore dûs à Bourgain, Furman, Lindenstrauss et Mozes \cite{bflm}, que nous avons généralisés récemment dans \cite{hs_nonprox}.
En un mot, l'idée est d'appliquer le résultat d'équidistribution sur le tore dans la représentation adjointe de $G$ sur son algèbre de Lie $\g$, puis d'utiliser l'application exponentielle pour en déduire la propriété d'expansion dans $G$.
Une complication intervient: l'application exponentielle modulo $q$ n'est définie que pour les éléments de $\g(\Z)$ qui sont divisibles par le double du radical de $q$, produit des diviseurs premiers de $q$.
C'est à cause de cela qu'il faut aussi utiliser les travaux de Salehi Golsefidy et Varjú \cite{sgv} sur l'expansion modulo les entiers sans facteurs carrés pour conclure.

\bigskip

L'article est divisé en quatre parties.
Dans la première, nous rappelons à grands traits l'argument de Bourgain et Gamburd pour déduire les théorèmes~\ref{trouspectral}, \ref{expanseur} et \ref{mesures} de la proposition~\ref{global}.
La deuxième partie est consacrée aux résultats de Salehi Golsefidy et Varjú \cite{sgv} et à certaines de leurs conséquences, nécessaires pour la suite de la démonstration.
Dans la troisième partie, nous énonçons le théorème d'équidistribution quantitative des marches aléatoires linéaires sur le tore, et l'appliquons à l'étude du phénomène d'expansion dans le groupe $\Omega$.
La quatrième et dernière partie termine la démonstration de la proposition~\ref{global}; il s'agit de combiner astucieusement les résultats des deux parties précédentes pour construire à partir de produits d'éléments de $A$ un élément qui satisfasse de bonnes propriétés de congruence.

Pour rendre l'article plus accessible, nous avons ajouté trois appendices.
Le premier résume certains résultats importants dûs à Matthews, Vaserstein et Weisfeiler \cite{MVW} et à Nori \cite{Nori} concernant l'approximation par des points entiers dans les groupes algébriques simples.
Le second donne les propriétés élémentaires de l'application exponentielle modulo un entier $q\in\N^*$.
Dans le dernier appendice, nous démontrons une borne inférieure sur la dimension d'une représentation unitaire irréductible du groupe pro-fini $\Omega$, essentielle pour déduire la propriété du trou spectral de la proposition~\ref{global}.

\bigskip

\noindent\textbf{\large{Résumé des notations}}

\smallskip

\begin{itemize}
\item $\Z$ anneau des entiers relatifs, $\Q$ corps des nombres rationnels
\item $\Z_p$, $\Q_p$ complétions respectives de $\Z$ et $\Q$ pour la valeur absolue $p$-adique
\item $\ZZ=\varprojlim\Z/q\Z\simeq \prod_p \Z_p$ complétion profinie de $\Z$
\item $\mu$ probabilité symétrique (à support fini) sur $\SL_d(\Z)$
\item $\Gamma$ sous-groupe de $\SL_d(\Z)$ engendré par le support de $\mu$
\item $G$ adhérence de Zariski de $\Gamma$ dans $\SL_d$
\item $\Omega$ adhérence de $\Gamma$ dans le groupe profini $\SL_d(\ZZ)$
\item $\gl_d(R)$ anneau des matrices carrées $d\times d$ à coefficients dans l'anneau $R=\Z,\, \Z_p, \ZZ$, ...etc.
\item $q\in\N^*$, $\pi_q:\gl_d(\ZZ)\to\gl_d(\Z/q\Z)$ projection modulo $q\in\N^*$
\item $\Omega_q =\{ g\in \Omega\ |\ \pi_q(g)=1\}$
\item $A\subset\gl_d(\ZZ)$, $N(A,q) = \card\pi_q(A)$
\item $A,B \subset \gl_d(\ZZ)$,
\[A + B = \{\, a + b \mid a \in A,\, b\in B \,\}\]
\[AB = \{\, ab \mid a \in A,\, b\in B \,\}\]
\item $A\subset\gl_d(\ZZ)$, $k\in\N^*$,
\[
\ssum_k A = \underbrace{A+ \dotsb+ A}_{\text{$k$ fois}} = \{\, a_1 + \dotsb + a_k \mid \forall i,\, a_i\in A\,\}
\]
\[
\Pi_k A = \underbrace{A\dotsm A}_{\text{$k$ fois}} = \{\, a_1\dots a_k \mid \forall i,\, a_i\in A\,\}.
\]
\item $\delta_2(p)=\left\{\begin{array}{ll} 1 & \mbox{si}\ p=2\\ 0 & \mbox{sinon}\end{array}\right.$, symbole de Kronecker.
\item $x\in\gl_d(\ZZ)$, $p$ premier, $v_p(x)$ valuation $p$-adique de $x$ dans $\gl_d(\ZZ)$
\item $q=\prod_p p^{m_p}$, de radical $r=\prod_{p|q} p$.
\item $I$ ensemble de nombres premiers, $q_I = \prod_{p \in I} p^{m_p}$ et $r_I = \prod_{p\in I} p$.
\item $\delta\in]0,1[$,  $q_\delta = \prod_p p^{\floor{\delta m_p}}$.
\item $\Hom(H,S^1)$ ensemble des caractères unitaires d'un groupe $H$.
\end{itemize}

\section{La stratégie de Bourgain-Gamburd}

Dans cette partie, nous rappelons la stratégie mise au point par Bourgain et Gamburd pour démontrer la propriété du trou spectral.
Cette méthode est robuste et s'applique aussi bien aux groupes profinis \cite{bg0, bg1, bg2, sgv, sg1, sg2} qu'aux groupes de Lie compacts \cite{bg_su2,bg_sud,bs}.
Pour plus de simplicité, nous restreindrons notre attention au cadre présenté dans l'introduction, et qui est celui du théorème~\ref{mesures}, que nous voulons démontrer.
Ainsi, dans toute la suite, $\mu$ désigne une probabilité symétrique sur $\SL_d(\Z)$, $\Gamma$ le sous-groupe engendré par le support de $\mu$, et $\Omega$ l'adhérence de $\Gamma$ dans $\SL_d(\ZZ)$.
Nous supposerons en outre pour commencer que le support de $\mu$ est fini, et ne lèverons cette hypothèse qu'au paragraphe~\ref{pts}.

\subsection{Aplanissement et trou spectral}

La première étape de la démonstration du théorème~\ref{mesures} consiste à montrer un lemme d'aplanissement pour les puissances de convolution de la mesure $\mu$.
Soit $m_\Omega$ la probabilité de Haar sur $\Omega$ et $L^2(\Omega)$ l'espace $L^2$ associé sur $\Omega$.
La norme usuelle sur $L^2(\Omega)$ est notée $\norm{\,\cdot\,}_2$.
Pour chaque $q \in \N^*$, on considère le sous-groupe de congruence
\[
\Omega_q= \{g\in\Omega\ |\ g\equiv 1\mod q\}
\]
et on note $P_q = \frac{\mathbbm{1}_{\Omega_q}}{m_\Omega(\Omega_q)}$ la suite d'unités approchées obtenue à partir de la famille $(\Omega_q)_{q \in \N^*}$.

\begin{lemma}[Lemme d'aplanissement]
\label{lm:aplaniss}
Il existe une constante $\delta > 0$ dépendant de $\tau$ et $C$ telles que pour tout $n\geq C\log q$, si 
\[
\norm{\mu^n*P_q}_2 \geq q^\tau,
\]
alors
\[
\norm{\mu^{2n}*P_q}_2 \leq q^{-\delta}\norm{\mu^n*P_q}_2.
\]
\end{lemma}
\begin{proof}
%Rappelons une version (implicitement contenu dans~\cite{bg0}) du lemme de Balog-Szemerédi-Gowers.
%Soit $H$ un groupe fini. Notons par $\norm{\cdot}_2$ la norme $\ell^2$ sur $H$.
%Soit $K \geq 1$ un parametre.
Rappelons qu'étant donné un paramètre $K\geq 1$, une partie $A$ d'un groupe fini $H$ est un \emph{sous-groupe $K$\dash{}approximatif} si $A$ est symétrique, contient l'élément neutre et s'il existe un ensemble $X \subset H$ fini de cardinal $\leq K$ tel que $AA \subset AX$.
En particulier, si $A$ est un sous-groupe $K$\dash{approximatif} alors
\begin{equation*}
%\label{eq:Kapprox}
\forall k \geq 1, \quad \abs{\Pi_k A} \leq K^{k - 1} \abs{A}.
\end{equation*}
Grâce à une version non commutative du lemme de Balog-Szemerédi-Gowers \cite[Corollary~2.46]{taovu}, les sous-groupes approximatifs d'un groupe fini sont reliés aux convolutions de mesures via le lemme suivant, qui apparaît déjà implicitement dans l'article de Bourgain et Gamburd \cite{bg0}.

\begin{lemma}
\label{lm:BSG}
Soit $K\geq 2$ un paramètre.
Soient $\nu$ et $\nu'$ des mesures de probabilité sur un groupe fini $H$. 
On suppose que 
\[\norm{\nu * \nu'}_2 \geq K^{-1}\norm{\nu}_2^{\frac{1}{2}}\norm{\nu'}_2^{\frac{1}{2}}.\]
Alors il existe un sous-groupe $K^{O(1)}$\dash{}approximatif $A \subset H$ tel que
\begin{itemize}

\item $K^{-O(1)} \norm{\nu}_2^{-2} \leq \abs{A} \leq K^{O(1)} \norm{\nu}_2^{-2}$,

\item il existe $h \in H$ tel que $\nu(hA) \geq K^{-O(1)}$,
\item il existe $h' \in H$ tel que $\nu'(Ah') \geq K^{-O(1)}$.
\end{itemize} 
\end{lemma}

%
%\begin{proof}[Démonstration du lemme \ref{lm:aplaniss}]
Pour démontrer le lemme~\ref{lm:aplaniss}, on raisonne par l'absurde en supposant que
\[
\norm{\mu^{*2n}*P_q}_2 > q^{-\delta}\norm{\mu^{*n} * P_q}_2,
\]
pour $\delta > 0$ arbitrairement petit.
D'après le lemme~\ref{lm:BSG} ci-dessus, appliqué dans le groupe fini $\Omega/\Omega_q$ à la mesure image de $\mu^{*n}$ par $\pi_q$,
il existe une partie symétrique $A \subset \Omega$ telle que $\pi_q(A)$ soit un sous-groupe $q^{O(\delta)}$\dash{approximatif} et que
\[
\mu^{*2n}(AA) \geq q^{-O(\delta)} \label{eq:marcheA}
\]
et
\[
N(A,q) \leq q^{-2\tau+O(\delta)} N(\Omega,q).
\]
Rappelons que $N(A,q)$ désigne le nombre de recouvrement de $A$ par des classes de $\Omega_q$; autrement dit, $N(A,q)$ est le cardinal de la projection de $A$ dans $\Omega/\Omega_q$.
Soient $\eps>0$ et $C\geq 1$ les constantes données par la proposition~\ref{global}, de sorte que si $\delta>0$ est choisi suffisamment petit par rapport à $\eps$, alors
\[
N(\Pi_{2C} A,q) \geq N(\Omega,q)^{1 - \tau}.
\]
Mais par ailleurs, comme $A$ est un sous-groupe $q^{O(\delta)}$-approximatif,
\[
N(\Pi_{2C} A, q) \leq q^{O(C\delta)} N(A,q) \leq q^{-2\tau+O(C\delta)}N(\Omega,q).
\]
Si $\delta$ est suffisamment petit par rapport à $\tau/C$, cela donne la contradiction recherchée.
\end{proof}

Une application itérée du lemme d'aplanissement permet alors d'obtenir la proposition suivante.

\begin{proposition}\label{presque}
Soit $\mu$ une probabilité sur $\SL_d(\Z)$.
On suppose que l'adhérence de Zariski du sous-groupe $\Gamma$ engendré par $\Supp\mu$ est simple, connexe et simplement connexe.
Alors, pour tout $\tau > 0$ il existe $C \geq 1$ tel que pour tout $n \geq C \log q$,
\begin{equation}\label{flat}
 \norm{\mu^{*n}*P_q}_2 \leq q^{\tau}.
\end{equation}
\end{proposition}

%
%\begin{proposition}[Bourgain-Gamburd]
%\label{pr:bg}
%Soit $\mu$ une mesure de probabilité sur $\Omega$ dont le support engendre une sous-groupe dénombrable $\Gamma$ dense dans $\Omega$.
%On suppose que
%\begin{enumerate}
%\item $\Omega$ est quasi-aléatoire par rapport à la famille $(\Omega_q)_{q \in \N^*}$;
%\item \label{it:MAprod}
%pour tout $\tau > 0$, il existe deux constantes $\eps > 0$ et $C \geq 1$ telles que pour tout $q \in \N^*$ suffisamment grand, pour tout $n \geq C \log N(\Omega,q)$, et toute partie symétrique $A \subset \Omega$, si $\mu^{*n}(A) \geq N(\Omega,q)^{-\eps}$, alors
%\[N(\Pi_C A,q) \geq N(\Omega,q)^{1 - \tau}.\]
%\end{enumerate}
%Alors l'action de $\Gamma$ sur $\Omega$ admet un trou spectral.
%\end{proposition}
%

\subsection{Multiplicité des représentations et trou spectral}

À partir de la proposition~\ref{presque}, la démonstration du théorème~\ref{mesures} se fait par un argument d'analyse de Fourier sur le groupe compact $\Omega$.
On notera $\hat\Omega$ le dual unitaire de $\Omega$, i.e. l'ensemble des représentations unitaires irréductibles de $\Omega$, à équivalence près.
Comme $\Omega$ est un groupe compact, $\hat\Omega$ est constitué de représentations de dimension finie.

L'ingrédient essentiel de la démonstration est la proposition~\ref{qr} ci-dessous, qui tire son origine des observations de Frobenius \comm{\cite{frobenius}} sur les représentations du groupe $\SL_2(\Z/p\Z)$, lorsque $p$ est un nombre premier arbitraire.
Lorsque le groupe $\Gamma$ est dense dans $\SL_d$ au sens de la topologie de Zariski, sa démonstration est plus facile, et apparaît déjà dans \cite[Proof of Theorem~1]{bv}.
Le résultat général est sans doute bien connu des spécialistes, mais les différents cas à étudier dans la démonstration sont éparpillés dans la littérature sur le sujet \cite{ls,sg1,breuillard_survey}.
Nous en donnerons donc une démonstration complète dans l'appendice~\ref{sec:qa}.

\begin{proposition}[Propriété quasi-aléatoire]
\label{qr}
Soit $\Gamma$ un sous-groupe de $\SL_d(\Z)$ et $\Omega$ l'adhérence de $\Gamma$ dans $\SL_d(\ZZ)$.
On suppose que l'adhérence de Zariski de $\Gamma$ dans $\SL_d$ est semi-simple, connexe et simplement connexe.
Alors, il existe $\kappa>0$ tel que pour tout  $(\rho,V_\rho) \in \hat\Omega$, il existe $q \in \N^*$ tel que
\[
\Omega_q \subset \ker \rho
\quad\mbox{et}\quad
\dim V_\rho \geq \kappa q^\kappa.
\]
\end{proposition}

Pour le reste, la démonstration est basée sur la formule de Parseval dans le groupe compact $\Omega$.
Si $(\rho,V_\rho)$ est une représentation unitaire de $\Omega$ et $\mu$ une mesure borélienne finie sur $\Omega$, on définit un élément $\rho(\mu) \in \End(V_\rho)$ par la formule 
\[
\rho(\mu)  = \int_\Omega \rho(g)\, \dd \mu(g).
\]
De façon similaire, si $f \in L^1(\Omega)$, on définit
\[
\forall \rho(f)  = \int_\Omega \rho(g) f(g) \dd m_\Omega(g).
\]
%Par cette définition, on a 
%\[\rho(\mu*f) = \rho(\mu) \rho(f)\]
%pour toute mesure borélienne finie $\mu$ et toute fonction $f \in L^1(\Omega,m_\Omega)$. De même pour les convolutions entre deux mesures ou entre deux fonctions.
%Pour un espace de Hilbert de dimension finie $V$, la norme de Hilbert-Schmidt d'un endomorphisme $\varphi \in \End(V)$ est définie par 
%\[
%\normHS{\varphi}^2= \Tr(\varphi^* \varphi)
%\]
%où $\varphi^*$ désigne l'adjoint de $\phi$.
%On remarque que
%\[
%\normop{\varphi} \leq \normHS{\varphi}.
%\]
%Nous pouvons maintenant démontrer le théorème~\ref{mesures} dans le cas particulier où $\mu$ est à support fini et engendre un groupe $\Gamma$ dont l'adhérence de Zariski, connexe et simplement connexe.

\begin{proof}[Démonstration du théorème~\ref{mesures}, cas particulier]
Pour pouvoir appliquer la proposition~\ref{qr}, on se restreint ici au cas particulier où $\mu$ est à support fini et où l'adhérence de Zariski de $\Gamma$ est un groupe algébrique simple, connexe et simplement connexe.
Pour $n\geq 1$, la formule de Parseval \cite[\S5.3, equation~(5.13)]{folland} appliquée à la fonction $\mu^{*n}*P_q$ sur $\Omega$ s'écrit 
\[
\norm{\mu^{*n}*P_q}_2^2  = \sum_{\rho \in\hat{\Omega}} (\dim V_\rho) \normHS{\rho(\mu)^n\rho(P_q)}^2,
\]
où l'on note $\normHS{\varphi}=\Tr(\varphi^*\varphi)$ la norme de Hilbert-Schmidt d'un endomorphisme $\varphi$ d'un espace de Hilbert.
D'après la proposition~\ref{qr}, il existe $\kappa>0$ tel que pour tout $\rho \in \hat\Omega$, il existe $q \in \N^*$ tel que
\[
\Omega_q \subset \ker \rho
\quad\mbox{et}\quad
\dim V_\rho \geq \kappa q^\kappa.
\]
L'inclusion $\Omega_q\subset\ker\rho$ implique en particulier $\rho(P_q) = \Id_{V_\rho}$.

Alors, d'après la proposition~\ref{presque} appliquée avec $\tau=\kappa/5$, il existe une constante $C\geq 1$ telle que, pour $n = C \log q$,
\[
\kappa q^\kappa \, \normHS{\rho(\mu)^n}^2 \leq \norm{\mu^{*n}*P_q}_2^2 \leq q^{2\tau}.
\]
Comme la norme de Hilbert-Schmidt majore la norme d'opérateur sur $\End V_\rho$, cela implique, avec le fait que $\rho(\mu)$ est un opérateur symétrique, pour $q$ suffisamment grand,
\[
\normop{\rho(\mu)} =  \normop{\rho(\mu)^n}^{\frac{1}{n}} \leq \normHS{\rho(\mu)^n}^{\frac{1}{n}} \leq \kappa^{-\frac{1}{n}} e^{\frac{-3\tau}{2C}} \leq e^{-\frac{\tau}{C}}<1.
\]
On en déduit qu'il existe des constantes $c > 0$ et $Q > 1$ telle pour tout $\rho \in \hat\Omega$,
\begin{enumerate}
\item ou bien $\normop{\rho(\mu)} \leq 1 - c$, 
\item ou bien il existe $q \leq Q$ et $\Omega_q \subset \ker \rho$.
\end{enumerate}
La deuxième option n'est possible que pour un nombre fini de $\rho \in \hat\Omega$, pour lesquelles on doit avoir $\normop{\rho(\mu)}<1$ si $\rho$ est non triviale.
\comm{En effet, si $\normop{\rho(\mu)}=1$, il y a une valeur propre égale à $1$, car $\mu$ est symétrique; mais alors, par uniforme convexité, un vecteur propre associé est nécessairement fixé par le support de $\mu$, et donc par $\Omega$ tout entier, ce qui contredit l'irréductibilité de $V_\rho$.}
Comme $\normop{T_\mu^0}=\sup\{ \normop{\rho(\mu)}\ ;\ \rho\in\hat{\Omega},\ \rho\neq 1\} $ le théorème est démontré lorsque l'adhérence de Zariski de $\Gamma$ est simplement connexe.
%
%Compte tenu du théorème de Peter-Weyl, la représentation régulière $\lambda_\Omega$ de $\Omega$ s'écrit comme une somme orthogonale $\lambda_\Omega = \rho_1 \oplus \rho_2 \oplus 1_\Omega$,
%où $1_\Omega$ est la représentation triviale, $\rho_1$ satisfait $\rs(\rho(\mu)) \leq 1 - c$ et $\rho_2$ est une représentation de de dimension finie sans vecteur $\Omega$-invariant.
%
%Vu comme représentation de $\Gamma$, $\rho_1$ n'a pas de vecteur presque invariant, par le lemme~\ref{lm:vectpresqueinv}.
%De plus, $\rho_2$ n'a pas de vecteur presque invariant car elle est de dimension finie et n'a pas de vecteur $\Gamma$-invariant. 
%On en conclut que $\Gamma \curvearrowright \Omega$ admet un trou spectral. 
\end{proof}

\subsection{La propriété du trou spectral}
\label{pts}

Nous expliquons maintenant comment déduire le cas général du théorème~\ref{mesures} du cas particulier démontré au paragraphe précédent.
Comme le reste de cette partie, l'argument présenté n'est pas nouveau, et apparaît déjà par exemple dans Salehi Golsefidy \cite[\S2.3]{sg1}.

Afin de démontrer le théorème~\ref{mesures} sans l'hypothèse de simple connexité, il est commode de passer par l'énoncé équivalent du théorème~\ref{trouspectral}.
Pour cela, nous ferons usage du lemme suivant, tiré du livre de Bekka, de la Harpe et Valette \cite[Proposition~G.4.2]{BdlHV}.
%Rappelons qu'étant donnée une mesure $\mu$ sur un groupe dénombrable sur $\Gamma$, on définit une mesure $\check\mu$ sur $\Gamma$ par $\forall g \in \Gamma$, $\check\mu(g) = \mu(g^{-1})$. 

\begin{lemma}
\label{lm:vectpresqueinv}
Soit $\Gamma$ un groupe dénombrable, $\rho$ une représentation unitaire de $\Gamma$ et $\mu$ une probabilité symétrique sur $\Gamma$. 
\begin{enumerate}
\item Si
\(\normop{\rho(\mu)} < 1,\)
alors $\rho$ n'a pas de vecteur presque invariant.
%\item Réciproquement, si $\rho$ n'a pas de vecteur presque invariant et que le support de $\mu$ engendre $\Gamma$, alors $1$ n'appartient pas au spectre de $\rho(\mu)$.
\item Si $\rho$ n'a pas de vecteur presque invariant et si le support de $\mu$ engendre $\Gamma$, alors 
%\item Si $\rho$ n'a pas de vecteur presque invariant et que le support de $\check\mu * \mu$ engendre $\Gamma$, alors 
\(\normop{\rho(\mu)} < 1\).
\end{enumerate}
%Si $\Gamma \curvearrowright \Omega$ admet un trou spectral et que $\Supp(\mu)$ engendre $\Gamma$, alors $1$ n'appartient pas au spectre de $T^0(\mu)$.
\end{lemma}

Par ailleurs, nous aurons besoin d'un lemme pour passer du groupe $\Gamma$ à un sous-groupe d'indice fini, et réciproquement.
L'énoncé ci-dessous est tiré de l'article \cite[Lemma~3.3]{AbertElek}.

\begin{lemma}
\label{lm:indicefini}
Soit $\Gamma$ un sous-groupe de $\SL_d(\Z)$ et $\Omega$ son adhérence dans $\SL_d(\ZZ)$.
Soit $\Gamma'$ un sous-groupe d'indice fini de $\Gamma$ et $\Omega'$ l'adhérence de $\Gamma'$ dans $\Omega$.
L'action $\Gamma' \curvearrowright \Omega'$ admet un trou spectral, si et seulement si $\Gamma \curvearrowright\Omega$ admet un trou spectral.
\end{lemma}
\comm{
\begin{proof}
D'après notre définition du trou spectral, l'action $\Gamma\curvearrowright\Omega$ admet un trou spectral si, et seulement si, il existe un sous-groupe $\Gamma_1=\bracket{S}$ de type fini dans $\Gamma$ tel que $\Gamma_1\curvearrowright\Omega$ admette un trou spectral.
D'après Lubotzky~\cite[Theorem~4.3.2]{Lubotzky}, cela est équivalent à dire qu'il existe une partie finie symétrique $S\subset\Gamma$ telle que la famille de graphes de Cayley $(\pi_q(S),\Omega/\Omega_q)$ est une famille d'expanseurs.
Vérifions que cela est équivalent à dire qu'il existe une partie symétrique $S$ finie et $\kappa>0$ tel que pour tout $A\subset\Omega$ mesurable vérifiant $m_\Omega(A)\leq 1/2$, il existe $s\in S$ tel que $\frac{m_\Omega(A\setminus sA)}{m_\Omega(A)}\geq\kappa$.
Clairement, si cette dernière propriété est satisfaite, alors, pour tout $A\subset\Omega/\Omega_q$ tel que $\abs{A}\leq 1/2[\Omega:\Omega_q]$, la partie $A'=\pi_q^{-1}(A)$ est mesurable et vérifie $m_\Omega(A')\leq 1/2$, donc il existe $s\in S$ tel que $\frac{m_\Omega(A'\setminus sA')}{m_\Omega(A')}\geq\kappa$, i.e. $\frac{\abs{A\setminus sA}}{\abs{A}}\geq\kappa$. Donc on a bien une famille d'expanseurs.
Réciproquement, supposons que $(\pi_q(S),\Omega/\Omega_q)$ soit une famille de graphes $\kappa$-expanseurs, et soit $A\subset\Omega$ avec $m_\Omega(A)\leq 1/2$.
Par régularité des mesures de Radon, on peut prendre $A_0\supset A$ ouvert tel que $m_\Omega(A_0\setminus A)\leq\frac{\kappa}{10}m_\Omega(A)$.
Ensuite, écrivant $A_0=\bigcup_{i\in\N}g_i\Omega_{q_i}$, on peut trouver $n$ tel que $m_\Omega(A_0\setminus A_1)\leq\frac{\kappa}{10}m_\Omega(A)$, où $A_1=\bigcup_{i\leq n}g_i\Omega_{q_i}$.
Si $q=\prod_{i\leq n}q_i$, il existe une partie finie $A'\subset\Omega/\Omega_q$ telle que $A_1=\pi^{-1}(A')$.
Par la propriété d'expansion, il existe $s\in S$ tel que $\frac{\abs{A'\setminus sA'}}{\abs{A'}}\geq\kappa\abs{A'}$, i.e. $\frac{m_\Omega(A_1\setminus sA_1)}{m_\Omega(A_1)}\geq \kappa m_\Omega(A_1)$, d'où l'on tire facilement $\frac{m_\Omega(A\setminus sA)}{m_\Omega(A)}\geq \frac{\kappa}{2}m_\Omega(A)$.
Vérifions maintenant que si cette propriété est satisfaite pour $\Gamma'\curvearrowright\Omega'$, alors elle l'est aussi pour $\Gamma\curvearrowright\Omega$.
On raisonne par contraposée, en supposant que pour toute partie finie $S\subset\Gamma$, il existe une suite $(A_n)$ de parties de $\Omega$ telle que $\frac{m(A_n\setminus sA_n)}{m(A_n)}\to 0$ pour tout $s\in S$.
Si $S\subset\Gamma'$ est une partie finie arbitraire, soit $(A_n)$ une telle suite.
Comme $\Omega=\bigcup_{c\in C} c\Omega'$ pour une partie $C$ finie, il existe une sous-suite de $(A_n)$ et $c\in C$ tel que pour tout $n$, posant $B_n=A_n\cap c\Omega'$, on ait $m(B_n)\geq m(A_n)/\abs{C}$.
Alors, pour tout $s\in S$, $\frac{m(B_n\setminus sB_n)}{m(B_n)}\leq \abs{C}\frac{m(A_n\setminus sA_n)}{m(A_n)}\to_{n\to\infty} 0$.
Cela montre que $\Gamma'\curvearrowright\Omega'$ n'admet pas de trou spectral.
La réciproque est vraie aussi, mais nous n'en aurons pas besoin, et renvoyons donc à \cite[Lemma~3.3]{AbertElek} pour la démonstration.
\end{proof}
}

\comm{
\begin{remark}
Le lemme ci-dessus est valable dans un cadre plus général que celui des sous-groupes de $\SL_d(\Z)$, et nous renvoyons à \cite[Theorem~8]{AbertElek} et \cite[Lemma~16]{sg1} pour la démonstration.
\end{remark}
}

Nous pouvons maintenant démontrer le théorème~\ref{trouspectral} en toute généralité.

\begin{proof}[Démonstration du théorème~\ref{trouspectral}]
Notons $G$ l'adhérence de Zariski de $\Gamma$ dans $\SL_d$. Par hypothèse, $G$ est simple.
Si $G$ est connexe et simplement connexe, le cas particulier du théorème~\ref{mesures} démontré au paragraphe précédent montre que $\normop{T_\mu^0}<1$ pour toute probabilité symétrique $\mu$ dont le support est fini et engendre $\Gamma$.
Avec le lemme~\ref{lm:vectpresqueinv}, cela montre que $\Gamma\curvearrowright\Omega$ admet un trou spectral.
\comm{On omet ici un argument: il n'existe pas forcément de mesure à support fini qui engendre $\Gamma$. Mais il existe une mesure à support fini qui engendre un groupe $\Gamma'\subset\Gamma$ ayant même adhérence de Zariski que $\Gamma$, et tel que $\Omega'=\Omega$. (Par approximation forte, $\Omega'$ est toujours d'indice fini dans $\Omega$, donc quitte à ajouter un nombre fini de générateurs dans le support de $\mu$, on peut imposer $\Omega'=\Omega$.) Le trou spectral de l'action de $\Gamma'$ sur $\Omega$ implique le trou spectral pour $\Gamma\curvearrowright\Omega$.}

Si $G$ n'est pas connexe, notons $G^0$ la composante neutre de $G$, $\Gamma^0 = \Gamma \cap G^0$ et $\Omega^0$ l'adhérence de $\Gamma^0$ dans $\Omega$.
Le groupe $\Gamma^0$ est un sous-groupe distingué d'indice fini dans $\Gamma$ dont l'adhérence de Zariski est $G^0$.
Si $G^0$ est simplement connexe, ce qui précède montre que $\Gamma^0\curvearrowright\Omega^0$ admet un trou spectral.
Le lemme~\ref{lm:indicefini} permet d'en déduire que $\Gamma\curvearrowright\Omega$ admet un trou spectral.
Cela démontre le théorème lorsque $G$ est simplement connexe.
\comm{En effet, par définition, cela revient à dire que $G^0$ est simplement connexe.}

Enfin, dans le cas général, soit $\pi \colon \tilde{G} \to G$ le revêtement simplement connexe de $G$.
Rappelons que $\pi$ est une isogénie.
Fixons une $\Q$-représen\-tation fidèle $\tilde{G}\hookrightarrow\SL_{\tilde d}$ dans un groupe linéaire, pour un certain $\tilde{d} \geq 2$, et notons
\[
\tilde{G}(\Z) = \tilde{G}\cap\SL_{\tilde d}(\Z).
\]
D'après Borel \cite[théorème~8.9]{Borel_GA}, les groupes $\pi(\tilde{G}(\Z))$ et $G(\Z)$ sont commensurables.
Posons 
\[
\tilde\Gamma = \tilde{G}(\Z) \cap \pi^{-1}(\Gamma).
\]
L'adhérence de Zariski de $\tilde\Gamma$ dans $\SL_{\tilde d}$ est égale au groupe simplement connexe $\tilde G$, et donc, d'après ce qui précède, l'action de $\tilde\Gamma$ sur son adhérence $\tilde\Omega$ dans $\SL_{\tilde d}(\ZZ)$ admet un trou spectral.
En d'autres termes, si l'on note
\[
\tilde\Gamma_q = \{\, g \in \tilde \Gamma\mid g \equiv 1 \mod q\,\}, \quad q\in \N^*
\]
%Ici, la congruence est determinée dans $\SL_{\tilde d}(\Z)$.
alors $\tilde\Gamma$ a la propriété $(\tau)$ par rapport à la famille $(\tilde\Gamma_q)_{q\in\N^*}$: il existe une partie finie $\tilde S\subset\tilde\Gamma$ et $\eps>0$ tels que
\[
\forall q\in\N^*,\ \forall v\in L^2_0(\tilde\Gamma/\tilde\Gamma_q),\ 
\exists g \in \tilde S,\quad \norm{\rho(g)v - v} \geq \eps\norm{v}.
\]
Comme $\pi$ est un morphisme de $\Q$-groupes, pour $g \in \tilde{G}$, les coefficients de $\pi(g)$ sont donnés par des polynômes rationnels en les coefficients de $g$.
Il existe donc $k \in \N^*$ tel que pour tout $q \in \N^*$, on ait
\[
\pi(\tilde\Gamma_{kq}) \subset \pi(\tilde\Gamma)_q=\{g\in\pi(\tilde\Gamma)\ |\ g\equiv 1\mod q\}.
\]
\comm{L'argument pour justifier cela est celui de Borel \cite[Proposition~7.12]{Borel_GA}: si les coefficients de $\pi(g)$ sont donnés par les polynômes $a_{ij}(g)$, $1\leq i,j\leq d$, alors les polynômes $b_{ij}$ définis par $b_{ij}(g-1) = a_{ij}(g)-\delta_{ij}$, $1\leq i,j\leq d$ sont sans terme constant. Donc si $g-1$ est bien divisible, $\pi(g)-1=(a_{ij}(g)-\delta_{ij})_{ij}$ aussi.}
Par suite, toute représentation de $\pi(\tilde\Gamma)/\pi(\tilde\Gamma)_q$ se relève en une représentation de $\tilde\Gamma/\tilde\Gamma_{kq}$, et $\pi(\tilde\Gamma)$ a la propriété $(\tau)$ par rapport à la famille de ses sous-groupes de congruence $(\pi(\tilde\Gamma)_q)_{q\in \N^*}$: il existe une partie finie $S\subset\pi(\tilde\Gamma)$ et $\eps>0$ tels que
\[
\forall q\in\N^*,\ \forall v\in L^2_0(\pi(\tilde\Gamma)/\pi(\tilde\Gamma)_q),\ 
\exists g \in S,\quad \norm{\rho(g)v - v} \geq \eps\norm{v}.
\]
Cela est équivalent à dire que l'action de $\pi(\tilde\Gamma)$ sur son adhérence dans $\SL_d(\ZZ)$ admet un trou spectral.
Comme $\pi(\tilde\Gamma)$ est d'indice fini dans $\Gamma$, le lemme~\ref{lm:indicefini} permet de conclure que l'action de $\Gamma$ sur $\Omega$ admet un trou spectral.
\end{proof}

\section{Expansion modulo les nombres sans facteur carré}
\label{sec:sqf}

Nous rappelons maintenant la propriété du trou spectral modulo les entiers sans facteur carré, due à Salehi Golsefidy et Varjú \cite{sgv}, et certaines de ses conséquences, dont nous aurons besoin plus tard.

%Dans toute cette partie, $\Gamma$ déginera un sous-groupe de $\SL_d(\Z)$.
%Ce groupe $\Gamma$ sera souvent vu comme un sous-groupe dense du groupe compact $\Omega$, égal à l'adhérence de $\Gamma$ dans $\SL_d(\ZZ)$.
%La topologie sur $\Omega$ est telle que la famille des sous-groupes de congruence
%\[
%\Omega_q = \{g\in\Omega\ |\ g\equiv 1\mod q\},\quad q\in\N^*
%\]
%forme une base de voisinages de l'identité dans $\Omega$.

\subsection{Trou spectral et expansion modulo $r$}
Le théorème suivant est tiré des travaux de Salehi Golsefidy et Varjú sur l'expansion dans les groupes parfaits \cite[Theorem~1]{sgv}.
\begin{theorem}[Trou spectral modulo les entiers sans facteur carré]
\label{thm:sgv}
Soit $\mu$ une probabilité sur $\SL_d(\Z)$, $\Gamma$ le semi-groupe engendré par $\mu$, et $\Omega$ l'adhérence de $\Gamma$ dans $\SL_d(\ZZ)$.
On suppose que l'adhérence de Zariski de $\Gamma$ est parfaite.
Alors l'action du groupe $\Gamma$ sur $\varprojlim \Omega/\Omega_r$, $r$ parcourant l'ensemble des entiers naturels sans facteur carré, admet un trou spectral.
\end{theorem}

Cette propriété d'expansion nous sera utile sous la forme de la proposition suivante, que nous associerons ensuite aux résultats de la partie~\ref{sec:torus} pour obtenir le théorème de trou spectral annoncé dans l'introduction.
Pour pouvoir utiliser les résultats de l'approximation forte rappelés dans l'appendice~\ref{sec:prelim}, nous supposerons toujours dans la suite que l'adhérence de Zariski de $\Gamma$ est simple, connexe et simplement connexe.

\begin{proposition}
\label{sqf}
Soit $\mu$ une probabilité sur $\SL_d(\Z)$ et $\Gamma$ le semi-groupe engendré par $\mu$.
On suppose que l'adhérence de Zariski de $\Gamma$ est simple, connexe et simplement connexe.
\wknote{$C$ ne dépend pas de $\eps_0$.}
Étant donné $\eps_0>0$, il existe $\eps>0$ et $C\geq 0$ tels que l'énoncé suivant soit vérifié pour tout entier sans facteur carré $r$ suffisamment grand et tout $q\geq r$.
Soit $A\subset\Gamma$ tel que $\mu^{*n}(A)\geq q^{-\eps}$, où $n\geq C\log q$.
Il existe un entier $r'|r$ tel que $r' \geq q^{-\eps_0}r$ et
\[
\pi_{r'}(\Pi_C A) =  \Omega/\Omega_{r'}.
\]
\end{proposition}

Nous reprenons essentiellement la démonstration de Bourgain et Varjú \cite[Proposition~3]{bv}.
Commençons par un lemme qui est une conséquence facile de la propriété du trou spectral.

\begin{lemma}
\label{lm:pirA}
Sous les hypothèses de la proposition~\ref{sqf}, il existe une constante $C$ telle que pour tout $n\geq C\log r$ et toute partie $A$ satisfaisant $\mu^{*n}(A) \geq q^{-\eps}$, on a
\[
\abs{\pi_r(A)} \gg q^{-\eps} [\Omega:\Omega_r].
\]
\end{lemma}
\begin{proof}
D'après le théorème~\ref{thm:sgv} et la caratérisation du trou spectral~\ref{lm:vectpresqueinv}, il existe $c>0$ tel que
\[\biggl\lVert{{\pi_r}_*(\mu^{*n}) - \frac{\mathbf{1}_{\Omega/\Omega_r}}{[\Omega : \Omega_r]}}\biggr\rVert_\infty \leq \biggl\lVert{{\pi_r}_*(\mu^{*n}) - \frac{\mathbf{1}_{\Omega/\Omega_r}}{[\Omega : \Omega_r]}}\biggr\rVert_2 \leq e^{-cn}.\]
Le lemme est alors immédiat avec $C = 1/c$.
\end{proof}

Pour la démonstration de la proposition~\ref{sqf}, nous aurons aussi besoin d'un résultat de Gowers \cite{Gowers}. 
Étant donné un groupe fini $H$, nous notons $m(H)$ le degré minimal d'une représentation linéaire complexe non triviale.

\begin{lemma}\label{lm:GowersQR}
Soient $A_1,A_2,A_3$ trois parties d'un groupe fini $H$. Si $\abs{A_1}\abs{A_2}\abs{A_3} > \frac{\abs{H}^3}{m(H)}$, alors $A_1A_2A_3 = H$.
\end{lemma}

Nous admettons ce lemme.
Sa démonstration est une simple application de la formule de Plancherel pour le groupe fini $H$; le lecteur est renvoyé à Gowers~\cite[Theorem~3.3]{Gowers} % \comm{ou Breuillard~\cite[Exercise~8.7]{breuillard_ihp}}
pour plus de détails.

\begin{proof}[Démonstration de la proposition~\ref{sqf}]
Écrivons $r=\prod_{i=1}^m p_i$.
Par le procédé de régularisation expliqué dans Bourgain-Gamburd-Sarnak~\cite[Lemma 5.2]{bgs},
%\wknote{J'ai enlevé la construction car c'est déjà bien expliqué dans Bourgain-Gamburd-Sarnak.}
on peut construire une partie $A'\subset \pi_r(A)$ et des constantes $K_i$, $i=1,\dots,m$ telles que pour chaque $i$, pour tout $g\in A'$,
\[
\abs{\{\, \pi_{p_i}(h) \in \Omega/\Omega_{p_i} \mid h \in A' : \pi_{p_1 \dotsm p_{i-1}}(h) = \pi_{p_1\dotsm p_{i-1}}(g) \,\}} = K_i
\]
et
\[
\abs{A'} = \prod_{i=1}^m K_i \geq (\prod_{i=1}^m 2\log[\Omega:\Omega_{p_i}])^{-1} \abs{\pi_r(A)}.
\]
\comm{
Cela se fait par induction rétrograde.
Tout d'abord, par le principe des tiroirs, on choisit $1\leq K_m\leq[\Omega:\Omega_{p_m}]$ et $A_m\subset\pi_r(A)$ tel que $\abs{A_m}\geq\frac{\abs{\pi_r(A)}}{\log[\Omega:\Omega_{p_m}]}$ en restreignant les $m-1$ premières coordonnées de sorte que pour tout $g=(g_1,\dots,g_m)$ dans $A_m$
\[
K_m \leq \abs{\{x\in \Omega/\Omega_{p_m} \ |\ (g_1,\dots,g_{m-1},x)\in \pi_r(A)\}}
\leq 2 K_m.
\]
Notons que par construction, si $g=(g_1,\dots,g_m)\in A$, alors la condition $(g_1,\dots,g_{m-1},x)\in A$ est équivalente à $(g_1,\dots,g_{m-1},x)\in A_m$, si bien que
\[
K_m \leq \abs{\{x\in \Omega/\Omega_{p_m} \ |\ (g_1,\dots,g_{m-1},x)\in A_m\}}
\leq 2 K_m.
\]
Donc on peut choisir une sous-partie $A'_m\subset A_m$ avec $\abs{A_m'}\geq\frac{\abs{A_m}}{2}$ pour que toutes les fibres de la dernière coordonnée aient le même cardinal:
pour tout $g=(g_1,\dots,g_m)\in A_m'$,
\[
\abs{\{x\in \Omega/\Omega_{p_m} \ |\ (g_1,\dots,g_{m-1},x)\in A_m'\}} = K_m.
\]
De même, on choisit $K_{m-1}$ et $A_{m-1}\subset A_m'$ tel que $\abs{A_{m-1}}\geq (\log[\Omega:\Omega_{p_{m-1}}])^{-1}\abs{A_m'}$ et pour tout $g=(g_1,\dots,g_m)$ dans $A_{m-1}$,
\[
K_{m-1} \leq \abs{\{x\in \Omega/\Omega_{p_{m-1}} \ |\ (g_1,\dots,g_{m-2},x)\in\pi_{p_1\dots p_{m-1}}(A_m') \}}
\leq 2 K_{m-1},
\]
puis $A_{m-1}'\subset A_{m-1}$ tel que pour tout $g=(g_1,\dots,g_m)\in A_{m-1}'$,
\[
\abs{\{x\in \Omega/\Omega_{p_{m-1}} \ |\ (g_1,\dots,g_{m-2},x)\in \pi_{p_1\dots p_{m-1}}(A_{m-1}')\}} = K_{m-1}.
\]
Après $m$ étapes, ce procédé nous permet d'obtenir un ensemble $A'=A_1'$ qui satisfait les conditions requises.
}
Au vu du lemme~\ref{lm:pirA} et du lemme~\ref{chinois2}, si $r$ est assez grand,
\[
\prod_{i=1}^m K_i \geq r^{-\eps}q^{-\eps} {[\Omega:\Omega_r]}
\geq q^{-3\eps} \prod_{i=1}^m {[\Omega:\Omega_{p_i}]}.
\]
La proposition~\ref{qr} donne l'existence de $\kappa > 0$ tel que pour tout nombre premier $p$ suffisamment grand,
$m(\Omega/\Omega_p) \geq p^\kappa$.
Posons
\[
I = \bigl\{\,i \in \{1,\dotsc, m\} \mid K_i > p_i^{-\kappa/3}{[\Omega:\Omega_{p_i}]} \,\bigr\}
\]
et 
\[
r' = \prod_{i \in I} p_i.
\]
On a alors
\[
q^{-3\eps} \prod_{i=1}^m [\Omega:\Omega_{p_i}]
\leq \prod_{i=1}^m K_i 
\leq \Bigl(\frac{r}{r'}\Bigr)^{-\kappa/3}\prod_{i=1}^m [\Omega:\Omega_{p_i}]
\]
et donc, pourvu que $\eps<\kappa\eps_0/9$,
\[
\frac{r}{r'} \leq q^{\frac{9\eps}{\kappa}} \leq q^{\eps_0}.
\]

Pour $i\in\{0,\dots,m-1\}$, notons $r'_i = \prod_{j \in \{1,\dotsc,i\}\cap I} p_j$ et montrons par récurrence sur $i$ que
\begin{equation}
\label{eq:pirA3}
\pi_{r'_i}(\Pi_3 A') = \Omega/\Omega_{r'_i}.
\end{equation}
Le résultat est trivial pour $i=0$.
Supposons donc le résultat connu pour $i - 1 \in \{0,\dotsc, m-1\}$.
Si $i \notin I$, il n'y a rien à démontrer, car $r'_i = r'_{i-1}$.
Sinon $r'_i = r'_{i-1} p_i$ et on dispose d'un morphisme injectif
\[(\pi_{r'_{i-1}},\pi_{p_i}) \colon \Omega/\Omega_{r'_i} \to \Omega/\Omega_{r'_{i-1}} \times \Omega/\Omega_{p_i}.\]
Il suffit de montrer que pour tout $h \in \Omega/\Omega_{r'_{i-1}}$,
\begin{equation}
\label{eq:hpiA3}
\pi_{p_i}\bigl(\pi_{r'_{i-1}}^{-1}(\{h\})\cap \Pi_3 A'\bigr) = \Omega/\Omega_{p_i}.
\end{equation}
Or, par hypothèse de récurrence il existe $g_1,g_2,g_3 \in A'$ tel que $h = \pi_{r'_{i-1}}(g_1g_2g_3)$.
Par construction de $A'$, les trois parties
\[B_j = \{\,g \in A' \mid \pi_{r'_{i-1}}(g) = \pi_{r'_{i-1}}(g_j)\,\}, \quad j = 1,2,3,\] 
vérifient $\abs{\pi_{p_i}(B_j)} \geq K_i$, et donc, vu le choix de $I$,
\[
\abs{\pi_{p_i}(B_1)}\abs{\pi_{p_i}(B_2)}\abs{\pi_{p_i}(B_3)} > \frac{[\Omega:\Omega_{p_i}]^3}{m(\Omega/\Omega_{p_i})}.
\]
D'après le lemme~\ref{lm:GowersQR}, cela implique $\pi_{p_i}(B_1 B_2 B_3) = \Omega/\Omega_{p_i}$.
Comme $\pi_{r'_{i-1}}(B_1B_2B_3) = h$, cela termine la démonstration de \eqref{eq:hpiA3} et ainsi celle de \eqref{eq:pirA3}.
%Montrons par récurrence sur $i$ que $A^{\prime 3}\cdot\Omega_{p_1\dots p_i} \supset \Omega_{r_0}$.
%\begin{tiny}
%Le résultat est trivial pour $i=0$, car $\Omega_{1}=\Omega\supset\Omega_{r_0}$.
%Supposons donc le résultat connu pour $i\geq 0$.
%Si $p_{i+1}|r_0$, il n'y a rien à démontrer.
%Sinon, soit $g=(g_1,\dots,g_{i+1})\in\Omega_{r_0}/\Omega_{p_1\dots p_{i+1}}$.
%(Pour être plus exact, il faudrait quotienter par $\Omega_{\mathrm{ppcm}(r_0,p_1\dots p_{i+1})}$.)
%Par hypothèse de récurrence, $(g_1,\dots,g_i)\in A^{\prime 3}\Omega_{r_0}/\Omega_{p_1\dots p_i}$: on peut trouver $a=(a_1,\dots,a_i)$, $b=(b_1,\dots,b_i)$ et $c=(c_1,\dots,c_i)$ dans $A'$ tels que $abc=g \mod \Omega_{p_1\dots p_i}$.
%Mais par construction de $A'$, les ensembles $X_a=\{x\in\Omega/\Omega_{p_{i+1}}\ |\ (a_1,\dots,a_i,x)\in\pi_{p_1\dots p_{i+1}}(A')\}$, $X_b=\dots$ et $X_c=\dots$ vérifient $\abs{X_a}=\abs{X_b}=\abs{X_c}=K_{i+1}$ et donc $\abs{X_a}\abs{X_b}\abs{X_c} > p_{i+1}^{-\tau}[\Omega:\Omega_{p_{i+1}}]$.
%D'après le théorème de Gowers, $X_aX_bX_c=\Omega/\Omega_{p_{i+1}}$, et l'on peut donc trouver $a_{i+1}$, $b_{i+1}$ et $c_{i+1}$ tels que $a'=(a_1,\dots,a_{i+1})$, $b'=(b_1,\dots,b_{i+1})$ et $c'=(c_1,\dots,c_{i+1})$ soient dans $\pi_{p_1\dots p_{i+1}}(A')$ et vérifient $a'b'c'=g$ modulo $\Omega_{p_1\dots p_{i+1}}$.
%Cela montre bien que $A^{\prime 3}\Omega_{p_1\dots p_{i+1}}\supset\Omega_{r_0}$.
%\end{tiny}
\end{proof}

\subsection{Points dans $\Omega_p/\Omega_{p^2}$}

Nous réunissons ici plusieurs lemmes qui nous serviront plus tard à construire des éléments non triviaux dans les quotients $\Omega_p/\Omega_{p^2}$.
Dans tout le paragraphe, $\Gamma$ désigne un sous-groupe de $\SL_d(\Z)$ dont l'adhérence de Zariski $G$ dans $\SL_d$ est simple, connexe, et simplement connexe. 
On note encore $\Omega$ l'adhérence de $\Gamma$ dans le groupe profini $\SL_d(\ZZ)$, et pour tout $q\in\N^*$,
\[
\Omega_q=\{g\in\Omega\ |\ g\equiv 1\mod q\}.
\]
Enfin, la projection modulo $q$ sera notée $\pi_q\colon\Omega\to\Omega/\Omega_q$.

\begin{lemma}
\label{lm:passcindee}
Pour tout nombre premier $p$ suffisamment grand, la suite exacte
\[1 \to \Omega_p/\Omega_{p^2} \to \Omega/\Omega_{p^2} \to \Omega/\Omega_{p} \to 1\]
n'est pas scindée.
\end{lemma}

\begin{proof}
Supposons $p > 2d$ et montrons que pour tout $g \in \Omega$, si $\pi_p(g)$ est d'ordre $p$ dans $\Omega/\Omega_p$ alors $\pi_{p^2}(g)$ est d'ordre $p^2$ dans $\Omega/\Omega_{p^2}$.

En effet, si $\pi_p(g)$ est d'ordre $p$ alors $g$ s'écrit $g = 1 + u$ avec $u^p \equiv 0 \mod p$.
Soit $k \geq 0$ l'entier minimal tel que $u^k \equiv 0 \mod p^2$.
Comme $1 \leq k \leq d$ et $p > 2d$, $u^p \equiv 0 \mod p^2$.
Par la formule du binôme de Newton,
\[g^p \equiv 1 + p u + {p\choose 2} u^2 + \dotsb + {p\choose k-1} u^{k-1} \mod p^2.\]
Mais
\[\Bigl(u + \frac{1}{p}{p\choose 2} u^2 + \dotsb + \frac{1}{p}{p\choose k-1} u^{k-1}\Bigr)^{k-1} \equiv u^{k-1} \not\equiv 0 \mod p.\]
Donc 
\[u + \frac{1}{p}{p\choose 2} u^2 + \dotsb + \frac{1}{p}{p\choose k-1} u^{k-1} \not\equiv 0 \mod p\]
et $g^p \not\equiv 1 \mod p^2$.
Par la formule du binôme à nouveau, $g^{p^2} \equiv 1 \mod p^2$.
Cela montre que $\pi_{p^2}(g)$ est d'ordre $p^2$ dans $\Omega/\Omega_{p^2}$.

Pour $p$ suffisamment grand, on sait par le lemme~\ref{unipotent} que le groupe $\Omega/\Omega_p$ contient des éléments d'ordre $p$.
On en déduit que $\pi_p \colon \Omega/\Omega_{p^2} \to \Omega/\Omega_{p}$ n'admet pas de section dans la catégorie des groupes. 
\end{proof}

Nous montrons maintenant une version quantitative de l'énoncé: si $\psi$ est une section ensembliste de la suite exacte du lemme~\ref{lm:passcindee}, alors on doit avoir $\psi(xy)\neq\psi(x)\psi(y)$ pour la plupart des couples $(x,y)$ d'éléments de $\Omega/\Omega_p$.

\begin{lemma}
\label{nosec}
Il existe $\sigma > 0$ dépendant seulement de $\Gamma$ tel que l'assertion suivante soit vraie.
Soit $r \in \N^*$ un entier sans facteur carré dont les facteurs premiers sont tous assez grands.
Soit $p$ un facteur premier de $r$.
Si $\psi \colon \Omega/\Omega_r \to \Omega$ est une section ensembliste de $\pi_r$, i.e. $\pi_r \circ \psi= \Id_{\Omega/\Omega_r}$, alors l'événement
\[
\psi(xy) \equiv \psi(x)\psi(y) \mod p^2
\]
est de probabilité au plus $p^{-\sigma}$ lorsque $x$ et $y$ sont deux variables aléatoires uniformes sur $\Omega/\Omega_r$.
\end{lemma}

\begin{proof}
On suppose que tous les facteurs premiers de $r$ sont suffisamment grands, si bien que $\Omega/\Omega_r = \prod_{p | r}\Omega/\Omega_p$ d'après le lemme~\ref{chinois}.
Fixons un facteur premier $p$ de $r$.
Écrivons $r = ps$ et identifions $\Omega/\Omega_r$ à $\Omega/\Omega_p \times \Omega/\Omega_s$.
Pour chaque $x_2 \in \Omega/\Omega_s$, l'application $\psi_{x_2} \colon \Omega/\Omega_p \to \Omega/\Omega_{p^2}$ définie par $\psi_{x_2}(x_1) = \pi_{p^2}\circ \psi(x_1,x_2)$ est une section de la projection $\pi_p \colon \Omega/\Omega_{p^2} \to \Omega/\Omega_p$.
Notons $\nu_{x_2}$ la mesure image de la probabilité uniforme sur $\Omega/\Omega_p$ par l'application $\psi_{x_2}$.

Observons que pour tout $x_2\in \Omega/\Omega_s$
\begin{equation}
\label{eq:normnux2}
\norm{\nu_{x_2}}^2 = {[\Omega:\Omega_p]}^{-\frac{1}{2}}.
\end{equation}
Observons aussi que pour tout couple $(x_2,y_2) \in \Omega/\Omega_s \times \Omega/\Omega_s$,
\begin{align*}
&\card\{\, (x_1,y_1) \in \Omega/\Omega_p \times \Omega/\Omega_p \mid \psi(x_1,x_2) \psi(y_1,y_2) \equiv \psi(x_1y_1,x_2y_2) \mod p^2\,\} \\
=&\card\{\, (x_1,y_1) \in \Omega/\Omega_p \times \Omega/\Omega_p \mid \psi_{x_2}(x_1) \psi_{y_2}(y_1) = \psi_{x_2y_2}(x_1y_1)\,\} \\
=& {[\Omega:\Omega_p]}^3 \bracket{\nu_{x_2} * \nu_{y_2}, \nu_{x_2y_2}}
\end{align*}
de sorte que 
\[ \PP[\psi(xy) \equiv \psi(x)\psi(y) \mod p^2 ] = \frac{[\Omega:\Omega_p]}{{[\Omega:\Omega_s]^2}}\sum_{(x_2,y_2) \in \Omega/\Omega_s \times \Omega/\Omega_s} \bracket{\nu_{x_2} * \nu_{y_2}, \nu_{x_2y_2}}.\]

Supposons par l'absurde que cette probabilité soit supérieure à $p^{-\sigma}$.
Alors il existe  $(x_2,y_2) \in \Omega/\Omega_s \times \Omega/\Omega_s$ tel que
\[p^{-\sigma} \leq {[\Omega:\Omega_p]} \bracket{\nu_{x_2} * \nu_{y_2}, \nu_{x_2y_2}}.\]
Par l'inégalité de Cauchy-Schwarz et \eqref{eq:normnux2},
\[
\norm{\nu_{x_2}*\nu_{y_2}}_2 \geq p^{-\sigma}\norm{\nu_{x_2}}_2^{\frac{1}{2}}\norm{\nu_{y_2}}_2^{\frac{1}{2}}.
\]
Avec le lemme~\ref{lm:BSG}, cela donne un sous-groupe $p^{O(\sigma)}$\dash{}approximatif $A\subset \Omega/\Omega_{p^2}$ et un élément $h \in \Omega/\Omega_{p^2}$ tel que $\nu_{x_2}(hA)\geq p^{-O(\sigma)}$ et 
\begin{equation}
\label{eq:cardAnu}
\abs{A} \leq p^{O(\sigma)} \norm{\nu_{x_2}}_2^{-2} = p^{O(\sigma)}{[\Omega:\Omega_p]}.
\end{equation}
Il s'ensuit que
\[
\frac{\abs{\pi_p(A)}}{[\Omega:\Omega_p]} = ({\pi_p}_*\nu_{x_2})(\pi_p(hA)) \geq \nu_{x_2}(hA) \geq p^{-O(\sigma)}.
\]

Pour $p$ assez grand, $\Omega/\Omega_p = G(\Z/p\Z)$ par le théorème~\ref{thm:Nori}.
Donc, par le théorème~\ref{qr,p}, $m(\Omega/\Omega_p) \geq \frac{p-1}{2}$. 
On peut alors choisir $\sigma$ suffisamment petit pour pouvoir appliquer le lemme~\ref{lm:GowersQR} à $\pi_p(A) \subset \Omega/\Omega_p$.
On obtient
\[\pi_p(\Pi_3 A)= \Omega/\Omega_p.\]
Soit $\phi \colon \Omega/\Omega_p \to \Pi_3 A$ une section de $\pi_p$.
D'après le lemme~\ref{lm:passcindee}, il existe $x_1,y_1 \in \Omega/\Omega_p$ tels que $\phi(x_1y_1) \neq \phi(x_1)\phi(y_1)$, et l'élément
\[
z = \phi(x_1y_1) \phi(y_1)^{-1}\phi(x_1)^{-1} \in \Pi_9 A
\]
vérifie alors
\[
z\in \Omega_p/\Omega_{p^2} \setminus \{1\}.
\]
Considérons l'action de $\Pi_3 A$ sur $\Omega_p/\Omega_{p^2}$ par conjugaison.
D'après le lemme~\ref{conjAd}, l'ensemble
\[B = \{\, aza^{-1} \mid a \in \Pi_3 A\,\}\]
est égal à l'orbite de $z$ sous l'action de $\Omega/\Omega_p$ et de plus $z$ n'est pas un point fixe.
Ainsi, le stabilisateur de $z$ est un sous-groupe propre donc d'indice au moins $\frac{p-1}{2}$ d'après le corollaire~\ref{cr:qr,p}.
On en déduit que $\abs{B} \geq \frac{p-1}{2}$.
En remarquant que $B \subset (\Pi_{15} A) \cap \ker\pi_p$, on obtient
\[\abs{\Pi_{18} A} \geq \abs{ B (\Pi_3 A)} \geq \abs{B}\abs{\pi_p(\Pi_3 A)} \geq \frac{p-1}{2} [\Omega:\Omega_p].\]
Si $\sigma$ est choisi suffisamment petit, cela contredit \eqref{eq:cardAnu} et le fait que $A$ est un sous-groupe $p^{O(\sigma)}$-approximatif.
\end{proof}

\begin{lemma}
\label{nosec2}
Pour tout $\eps > 0$, si $r$ est un entier suffisamment grand sans facteur carré, et si $A$ est une partie de $\Omega$ telle que $\pi_r(A)=\Omega/\Omega_r$, alors il existe $r_0|r$ avec $r_0 \leq r^{\eps}$ et $z\in (\Pi_3 A) \cap \Omega_r$ tel que pour tout facteur premier $p$ de $\frac{r}{r_0}$, $z\not\in\Omega_{p^2}$.
\end{lemma}

\begin{proof}
On peut supposer que tous les facteurs premiers de $r$ sont assez grands pour que le lemme~\ref{nosec} soit valable.
Soient alors $x$ et $y$ deux variables aléatoires indépendantes de loi uniforme sur $\Omega/\Omega_r$, et $\psi \colon \Omega/\Omega_r \to A$ une section de la projection naturelle $\Omega\to\Omega/\Omega_r$.
Par le lemme~\ref{nosec},
\begin{align*}
\E[ \sum_{p|r} (\log p) \mathbbm{1}_{\{\psi(xy)=\psi(x)\psi(y)\mod p^2\}}]
& = \sum_{p|r} (\log p) \PP[\psi(xy)=\psi(x)\psi(y) \mod p^2]\\
& \leq \sum_{p|r} (\log p) p^{-\sigma}\\
& \leq \eps \sum_{p|r} \log p = \eps\log r,
\end{align*}
et il doit exister $x$ et $y$ tels que
$\sum_{p|r} (\log p) \mathbbm{1}_{\{\psi(xy)\equiv\psi(x)\psi(y) \mod p^2\}} \leq \eps\log r$.
L'élément $z=\psi(xy)\psi(y)^{-1}\psi(x)^{-1}$ et l'entier $r_0 = \prod_{p|r : z \in \Omega_{p^2}}p$ vérifient toutes les conditions requises.
\end{proof}

\section{Expansion via la représentation adjointe}
\label{sec:torus}

Outre les résultats de Salehi Golsefidy et Varjú \cite{sgv} sur la propriété du trou spectral modulo un entier sans facteur carré, l'ingrédient principal dans la démonstration de la proposition~\ref{global} est un résultat d'équidistribution quantitative des marches aléatoires linéaires sur le tore.
Il sera utilisé pour démontrer la proposition~\ref{irrep} ci-dessous, qui sera ensuite appliquée dans la représentation adjointe de $G$.

\subsection{Représentations linéaires et répartition modulo~$q$}

Nous considérons une représentation irréductible $V$ définie sur $\Q$ de l'adhérence de Zariski de $\Gamma$, et fixons un réseau rationnel $V(\Z)$ dans $V$ stable par $\Gamma$.
\comm{Cela est possible, d'après Borel \cite[Corollaire~7.13]{Borel_GA}, car par hypothèse $\Gamma\subset G\cap\SL_d(\Z)$ est inclus dans un sous-groupe arithmétique de $G$.}
Pour comprendre la répartition de la marche aléatoire induite sur $V$ modulo un entier arbitraire, nous noterons, pour une partie $X \subset V(\Z)$ et $q\in\N^*$,
\[
N(X,q) = \card \pi_q(X),
\]
où $\pi_q:V(\Z)\to V(\Z)/qV(\Z)$ est la projection naturelle.
Nous utiliserons aussi les notations usuelles de combinatoire additive: pour un entier $C \geq 1$ et une partie $X \subset V(\Z)$,
\[
\textstyle{\sum_C X} = X + \dotsb + X = \{\, v_1+\dotsb+v_C\ ;\ \forall i,\, v_i \in X \,\},
\]
et pour $A \subset\Gamma$ et un vecteur $v \in V$, 
\[
A\cdot v = \{ a\cdot v\ ;\ a\in A\}.
\]
Enfin, pour $q\in\N^*$ et $v\in V(\Z)$ on note $\pgcd(q,v)$ le plus grand diviseur commun entre l'entier $q$ et le vecteur $v$.
(Un entier $d$ est dit diviseur de $v \in V(\Z)$ si $\frac{v}{d} \in V(\Z)$.)

\begin{proposition}[Action sur une représentation irréductible]
\label{irrep}
Soit $\mu$ une probabilité sur $\SL_d(\Z)$, $\Gamma$ le sous-groupe engendré par $\mu$, $G$ l'adhérence de Zariski de $\Gamma$, et $V$ une $\Q$-représentation de $G$.
On suppose que:
\begin{enumerate}
\item $\exists \tau>0:\ \int \norm{g}^\tau \mu(\dd g) < +\infty$;
\item $G$ est connexe pour la topologie de Zariski;
\item l'algèbre engendrée par $G(\R)$ dans $\End(V(\R))$ est égale à $\End(V(\R))$.
\end{enumerate}
On note $V(\Z)$ un réseau de $V(\R)$ stable sous l'action de $\Gamma$.
Il existe une constante $C > 0$ telle que pour tout $\eps>0$ suffisamment petit, l'énoncé suivant soit vérifié pour tout entier $q$ suffisamment grand.\\
Soit $A\subset\Gamma$ un ensemble tel que
\[
\mu^{*n}(A) \geq q^{-\eps}
\quad\mbox{pour un certain}\ 
n\geq C\log q.
\]
Pour tout vecteur non nul $v \in V(\Z)$,
\[
N\bigl(\textstyle{\sum_C A\cdot v},q\bigr) \geq\displaystyle q^{-C\eps}\left(\frac{q}{\pgcd(q,v)}\right)^{\dim V}.
\]
\end{proposition}

Pour la démonstration de cette proposition, nous aurons besoin du théorème d'équidistribution quantitative des marches aléatoires linéaires sur le tore, que nous rappelons ci-dessous.
Avec une hypothèse supplémentaire de proximalité, ce résultat est dû à Bourgain, Furman, Lindenstrauss et Mozes \cite{bflm}; les modifications nécessaires à leur démonstration pour lever cette hypothèse sont détaillées dans \cite{hs_nonprox}, d'où est tiré l'énoncé ci-dessous.
Dans la suite, étant donné un entier $d\geq 2$, on note $\T^d=\R^d/\Z^d$ le tore de dimension $d$.
Si $\nu$ est une mesure borélienne finie sur $\T^d$, ses coefficients de Fourier sont indexés par $a\in\Z^d$, et donnés par la formule
\[
\widehat{\nu}(a) = \int_{\T^d} e^{2i\pi\bracket{a,x}}\nu(\dd x).
\]

\begin{theorem}[Équirépartition des marches aléatoires sur $\T^d$]
\label{bflm}
Soit $\mu$ une probabilité sur $\SL_d(\Z)$, $\Gamma$ le sous-groupe engendré par le support de $\mu$, et $\lambda_1$ le premier exposant de Liapounoff de $\mu$.
On suppose que
\begin{enumerate}
\item $\exists \tau>0:\ \int \norm{g}^\tau \mu(\dd g) < +\infty$;
\item $\Gamma$ agit irréductiblement sur $\R^d$;
\item l'adhérence de Zariski de $\Gamma$ est connexe.
\end{enumerate}
Pour tout $\lambda\in]0,\lambda_1[$, il existe une constante $C>0$ telle que, pour tout $t \in {]0,\frac{1}{2}[}$ et tout $x\in\T^d$, s'il existe $a\in\Z^d\setminus\{0\}$ tel que
\[
\abs{\widehat{\mu^{*n}*\delta_x}(a)} \geq t
%= \Abs{\int_{\T^d} e^{2i\pi\bracket{a,gx}}\mu^{*n}(\dd g)} \geq t,
\qquad\mbox{avec}\ n\geq C\log\frac{\norm{a}}{t},
\]
alors il existe $q\in\N$ et $x'\in(\frac{1}{q}\Z^d/\Z^d)$ tels que
\[
q \leq \left(\frac{\norm{a}}{t}\right)^C
\quad\mbox{et}\quad
d(x,x') \leq e^{-\lambda n}.
\]
\end{theorem}

Ce théorème implique une décroissance de Fourier pour les marches aléatoires dont le point de départ est rationnel, et c'est sous cette forme qu'il nous sera utile.
Le corollaire suivant et sa démonstration sont tirés de l'article de Bourgain et Varjú \cite[Lemma~7]{bv}.

\begin{corollary}[Décroissance de Fourier des marches rationnelles]
\label{decay}
Soit $\mu$ une probabilité sur $\SL_d(\Z)$, $\Gamma$ le sous-groupe engendré par $\mu$, et $\lambda_1$ le premier exposant de Liapounoff de $\mu$.
On suppose que
\begin{enumerate}
\item $\exists \tau>0:\ \int \norm{g}^\tau \mu(\dd g) < +\infty$;
\item la sous-algèbre engendrée par $\Gamma$ est égale à $M_d(\R)$;
\item l'adhérence de Zariski de $\Gamma$ est connexe.
\end{enumerate}

Il existe alors des constantes $C,\ \tau>0$ telles que, pour tout $x_0=\frac{p_0}{q_0}\in\Q^d/\Z^d$ avec $\pgcd(p_0,q_0)=1$,
\[
\forall n\geq C\log q_0,\ \forall b\in\Z^d,\quad
\abs{\widehat{\mu^{*n}*\delta_{x_0}}(b)} \leq \left(\frac{q_0}{\pgcd(q_0, b)}\right)^{-\tau}.
\]
\end{corollary}

\comm{
\begin{remark}
L'hypothèse que $\Gamma$ engendre $M_d(\R)$ implique en particulier que $\Gamma$ agit irréductiblement sur $\R^d$.
La réciproque n'est pas vraie, comme le montre l'exemple de $\GL_{d/2}(\Z[i])$, qui agit irréductiblement sur $\R^d\simeq\C^{d/2}$.
Nous verrons plus tard que ce corollaire est faux si l'algèbre engendrée par $\Gamma$ n'est pas égale à $M_d(\R)$.
\end{remark}
}

\begin{proof}
Si $q_0=dq_1$ et $b=db'$ pour un certain entier $d$, alors, avec $x_1=\frac{p_0}{q_1}$,
\[
\widehat{\mu^{*n}*\delta_{x_0}}(b)
%= \E[ e^{2i\pi\bracket{b,g_n\dots g_1\frac{p_0}{q_0}}}]
%= \E[ e^{2i\pi\bracket{b',g_n\dots g_1\frac{p_0}{q_1}}}]
= \widehat{\mu^{*n}*\delta_{x_1}}(b')
\]
de sorte qu'il suffit de démontrer le théorème dans le cas où $\pgcd(q_0, b) = 1$, ce que nous supposerons dorénavant.
\comm{
En effet, dans le cas général, on pourra majorer
\[
\abs{\widehat{\mu^{*n}*\delta_{x_0}}(b)} = 
\abs{\widehat{\mu^{*n}*\delta_{x_1}}(b')} \leq q_1^{-\tau} = \left(\frac{q_0}{q_0\wedge b}\right)^{-\tau}.
\]
}
Notons $E=M_d(\R)$ et $E^*$ l'espace des formes linéaires sur $E$. Dans $E^*$, nous considérons l'ensemble $E^*(\Q)$ des formes linéaires à coefficients rationnels dans la base canonique, et le réseau $E^*(\Z)$ constitué des formes linéaires $\phi \in E^*(\R)$ telles que $\phi(M_d(\Z)) \subset \Z$.
% De manière similaire, on définit $E^*(\Q) = E^*(\Z) \otimes_\Z \Q$. 
Soit $\tilde{x}_0$ un relevé de $x_0$ dans $\R^d$, et $\phi_0$ l'image dans $E^*/E^*(\Z)$ de la forme linéaire $g \in E \mapsto \bracket{b,g \tilde{x}_0} \in \R$ où $\bracket{\cdot,\cdot}$ désigne le produit scalaire usuel sur $\R^d$.

%Notons $L$ (resp. $R$) l'application $E\to\End(E)$ qui associe à un élément $g \in E$ l'endomorphisme $L_g$ (resp. $R_g$) de multiplication à gauche (resp. à droite) par $g$ dans $E$.
Notons $L^*$ (resp. $R^*$) l'application $E\to\End(E^*)$ définie par
\[
\forall \phi \in E^*,\ L^*_g(\phi)(u) = \phi(gu) \qquad   (\mbox{resp.}\ R^*_g(\phi)(u) = \phi(ug))
\]
et posons
\[
\tilde{\mu} = \frac{1}{2}\bigl(L^*(\mu) + R^*(\mu)\bigr).
\]
Pour $g \in \SL_d(\Z)$, $L^*_g$ est inversible et préserve le réseau $E^*(\Z)$ donc agit sur le tore $E^*/E^*(\Z)$. 
Nous allons appliquer le théorème~\ref{bflm} à la marche aléatoire sur le tore $E^*/E^*(\Z)$ associée à la loi de transition $\tilde{\mu}$ et au point de départ $\phi_0$. 
Les coefficients de Fourier sur  $E^*/E^*(\Z)$ sont naturellement indexés par $M_d(\Z)$ :
pour une mesure $\nu$ sur $E^*/E^*(\Z)$, et $a \in M_d(\Z)$, on a
\[
\widehat\nu(a) = \int_{E^*/E^*(\Z)} e^{2\pi i \phi(a)} \dd \nu(\phi).
\]
Pour $n \geq 0$, le coefficient de Fourier en l'identité de la marche aléatoire est
\[
\widehat{\tilde{\mu}_{n}*\delta_{\phi_0}}(I) = \widehat{\mu^{*n}*\delta_{x_0}}(b).
\]
Remarquons que le groupe algébrique engendré par le support de $\tilde{\mu}$, isomorphe à $G \times G$, est bien connexe, et que son action sur $E^{*}(\R)$ est irréductible, car $E$ est une algèbre simple engendrée par $G(\R)$.
L'hypothèse du moment exponentiel est aussi satisfaite pour $\tilde{\mu}$, dès qu'elle l'est pour $\mu$.
Donc il existe une constante $C\geq 0$ telle que si pour $t\in]0,1/2[$ et $n\geq C\log t$,
\[
\abs{\widehat{\tilde{\mu}_{n}*\delta_{\phi_0}}(I)}
\geq t,\]
alors il existe un élément $\phi\in E^*(\Q)/E^*(\Z)$ de dénominateur au plus $t^{-C}$ et tel que
\begin{equation}\label{approx}
 \norm{\phi_0-\phi} < e^{-n\frac{\lambda_1}{2}}.
\end{equation}
Posons $t=q_0^{-\frac{1}{2C}}$.
Supposons $n\geq \frac{4\log q_0}{\lambda_1}$ et donc $e^{-n\frac{\lambda_1}{2}}\leq q_0^{-2}$.
Comme $\phi_0$ admet un coefficient rationnel réduit de dénominateur $q_0$, il n'admet pas d'approximation rationnelle de hauteur au plus $t^{-C}=q_0^{\frac{1}{2}}$ satisfaisant \eqref{approx}.
Ainsi, si $n\geq\frac{4\log q_0}{\lambda_1}$, on trouve bien
\[
\abs{\widehat{\mu^{*n}*\delta_{x_0}}(b)}
= \abs{\widehat{\tilde{\mu}_{n}*\delta_{\phi_0}}(I)}
< t = q_0^{-\frac{1}{2C}}.
\]
Quitte à remplacer $C$ par $\max(C,\frac{4}{\lambda_1})$, cela montre le corollaire, avec $\tau=\frac{1}{2C}$.
\end{proof}

\comm{
\begin{remark}
Si $E<M_d(\R)$, l'argument ci-dessus n'est pas valide, pour la raison suivante.
La forme linéaire $\phi_0$ est rationnelle de hauteur exactement $q_0$, vue comme élément de $M_d(\R)^*$, et n'est donc pas bien approchable par une forme rationnelle de hauteur au plus $q_0$.
Mais la restriction de $\phi_0$ à $E$ peut être de hauteur plus petite, et donc coïncider avec une forme rationnelle de petite hauteur sur $E$.
\nsnote{Cet exemple serait plus facile avec $E=M_1(\C)$, s'il était engendré par $\Gamma=\SL_1(\Z[i])$.} 
Par exemple, si $\Gamma=\SL_{2}(\Z[i])$, alors $E=M_{2}(\C)$.
Soit $p\equiv 1\mod 4$ et $a$ tel que $a^2\equiv-1\mod p$.
Posons $x_0=\frac{1}{p}\begin{pmatrix}1\\ a\\ 0\\ 0\end{pmatrix}$ et $b=\begin{pmatrix}1 & a & 1 & a\end{pmatrix}$.
En restriction aux matrices de $E$, de la forme
$g=\begin{pmatrix} x & y & \dots & \dots\\
-y & x & \dots & \dots\\
z & t & \dots & \dots\\
-t & z & \dots & \dots
\end{pmatrix}$
la forme linéaire s'écrit
\[
\phi_0(g) = \bracket{b,gx_0} = \frac{1+a^2}{p}(x+ay+z+at)= m(x+ay+z+at)
\]
et appartient donc à $E^*(\Z)$.
Par conséquent, sa hauteur (dans $E^*(\Q)/E^*(\Z)$) est nulle.
\end{remark}
}

\comm{
\begin{remark}
Si on essaie d'appliquer directement le théorème~\ref{bflm} à la marche aléatoire associée à $\mu$, on ne parvient qu'à majorer les coefficients de Fourier $\widehat{\mu^{*n}*\delta_{x_0}}(b)$ pour $\norm{b}\leq q^{1/C}$, mais sous la seule hypothèse d'une action irréductible et de l'adhérence de Zariski connexe.
L'astuce ci-dessus donne une borne indépendante de la norme de $b$.\\
On pourrait s'étonner d'avoir une borne pour les petits coefficients de Fourier même si $E=M_{d/2}(\C)$, alors que vu la remarque précédente, la marche aléatoire ne s'équidistribue pas modulo $p$ si $p\equiv 1\mod 4$: l'action de $E(\Z)$ sur $(\Z/p\Z)^d$ n'est pas irréductible.
Mais cela est possible, car toute forme linéaire $b$ qui s'annule sur un sous-$E(\Z)$-module de $(\Z/p\Z)^d$ doit faire intervenir un résidu quadratique de $-1$ modulo $p$, et est donc nécessairement de norme au moins $\sqrt{p}$.
\end{remark}
}

\comm{
\begin{remark}
Cette astuce donne aussi un résultat lorsque l'on part d'un point $x_0$ irrationnel: on a une décroissance de Fourier sauf si $\phi_0=b\otimes x_0$ admet une bonne approximation rationnelle dans $E^*(\Q)/E^*(\Z)$.
\end{remark}
}

%\begin{proposition}[Action sur ld'algèbre de Lie]
%\label{alglie}
%Soit $\mu$ une probabilité sur $\SL_d(\Z)$, et $\Gamma$ le sous-groupe engendré par $\mu$.
%On suppose que l'adhérence de Zariski de $\Gamma$ est simple.
%Pour tout $\eps_1>0$, il existe des constantes $C\geq 0$ et $\eps>0$ telles que l'énoncé suivant soit vérifié pour tout entier $q$ suffisamment grand.\\
%Soit $A\subset\Omega$ un ensemble symétrique tel que
%\[
%\mu^n(A) \geq q^{-\eps}
%\quad\mbox{pour}\quad
%n\geq C\log q.
%\]
%Étant donné $X_0$ dans $\g(\Z)$ notons $I$ l'ensemble des nombres premiers qui satisfont les conditions
%\begin{enumerate} 
%\item si $\eps_3 m_p\geq 1$, alors $p^{\lfloor\eps_3 m_p\rfloor} | X_0$ et $p^{\eps_2 m_p}\not | X_0$;
%\item si $\eps_3 m_p< 1$, alors $p|X_0$ et $p^2\not | X_0$;
%\end{enumerate}
%et
%\[
%q_I = \prod_{p\in I} p^{m_p-1}.^
%\]
%Alors
%\[
%N(\bracket{A,X_0}_s,q) \geq q^{-O(\eps_0)}q_I^{\dim\g}.
%\]
%\end{proposition}

\begin{proof}[Démonstration de la proposition~\ref{irrep}]
Fixons une $\Z$-base de $V(\Z)$ qui permet d'identifier $V(\Z)$ à $\Z^d$, et de définir un produit scalaire $\bracket{\,\cdot\,, \,\cdot\,}$ sur $V$.
Quitte à remplacer $v$ par $\frac{v}{\pgcd(q,v)}$, on peut supposer que $\pgcd(q,v)=1$.
%Ensuite, on écrit $X'=X/\pgcd(q,X)$, $q'=q/\pgcd(q,X)$, et
%\[
%N(\textstyle{\sum_CA\cdot X},q)
%= N(\textstyle{\sum_C A\cdot X',q')
%\geq \displaystyle (q')^{-C\eps} (q')^{\dim V}
%\geq \displaystyle q^{-C\eps}\left(\frac{q}{\pgcd(q,X)}\right)^{\dim V}.
%\]
D'après le corollaire~\ref{decay} appliqué à l'action de $\Gamma$ sur $V(\R)/V(\Z)$ et au point
\[
x_0=\frac{v}{q} \mod V(\Z) \in V(\R)/V(\Z),
\]
il existe deux constantes $C\geq 0$ et $\tau>0$ telles que pour tout $n\geq C\log q$,
\begin{equation}
\label{eq:decayonV}
\forall b\in V(\Z),\quad 
\abs{\widehat{\mu^{*n}*\delta_{x_0}}(b)} \leq \left(\frac{q}{\pgcd(q,b)}\right)^{-\tau}.
\end{equation}
Remarquons que la mesure $\mu^{*n}*\delta_{x_0}$ est supportée par $\frac{1}{q}V(\Z)/V(\Z)$ qui est en bijection avec $V(\Z)/qV(\Z)$.
Nous allons utiliser l'analyse de Fourier sur le groupe additif $V(\Z)/qV(\Z)$.
L'application $b \mapsto (x\mapsto e^{2i\pi\frac{\bracket{b,x}}{q}})$ identifie $V(\Z)/qV(\Z)$ avec son groupe dual. 
Notons $\nu$ l'image de $\mu^{*n}$ dans $V(\Z)/qV(\Z)$ par l'application
\[
\begin{array}{ccc}
\Gamma & \to & V(\Z)/qV(\Z)\\
g & \mapsto & g\cdot v \mod qV(\Z).
\end{array}
\]
L'inégalité \eqref{eq:decayonV} est alors équivalente à 
\[
\forall b\in V(\Z)/V(q\Z),\quad 
\abs{\hat{\nu}(b)} \leq \left(\frac{q}{\pgcd(q,b)}\right)^{-\tau}.
\]
Quitte à augmenter $C$, nous pouvons supposer $C\geq \frac{\dim V}{\tau}$, et si $\nu^{(C)}$ désigne la $C$-ième convolution additive de $\nu$, on a alors
\[
\forall b\in V(\Z)/qV(\Z),\quad 
\abs{\widehat{\nu^{(C)}}(b)} \leq \left(\frac{q}{\pgcd(q,b)}\right)^{-\dim V},
\]
et par la formule de Parseval dans le groupe $V(\Z)/qV(\Z)$,
%\[
%\norm{f}^2 = \frac{1}{\abs{\g_q}}\sum_{b\in\g_q^*} \abs{\hat{f}(b)}^2.
%\]
\begin{align*}
\norm{\nu^{(C)}}_2^2
& = \frac{1}{\abs{V(\Z)/qV(\Z)}}\sum_{b\in V(\Z/q\Z)} \abs{\widehat{\nu^{(C)}}(b)}^2\\
& \ll q^{-\dim V} \sum_{r|q} r^{\dim V} r^{-2\dim V}\\
& \ll q^{-\dim V}.
\end{align*}
Comme $\mu^{*n}(A)\geq q^{-\eps}$, on a aussi $\nu^{(C)}(\ssum_C A \cdot v) \geq q^{-O(\eps)}$, et avec l'inégalité de Schwarz,
\[
N(\ssum_CA\cdot v,q) \geq \displaystyle q^{-O(\eps)}q^{\dim V}.
\]
\end{proof}

\subsection{Représentation adjointe}

Nous appliquons maintenant les résultats du paragraphe précédent dans la représentation adjointe de $G$, pour en déduire un résultat d'expansion qui nous permettra, plus tard, de démontrer la proposition~\ref{global}.

\smallskip

Rappelons que $\mu$ est une probabilité à support fini sur $\SL_d(\Z)$, $\Gamma$ le sous-groupe engendré par le support de $\mu$, $G$ l'adhérence de Zariski de $\Gamma$, et $\Omega$ l'adhérence de $\Gamma$ dans $\SL_d(\ZZ)$.
On suppose de plus que $G$ est un groupe simple, connexe et simplement connexe.

Dorénavant, le nombre $q$ s'écrira $q=\prod_p p^{m_p}$.
%, et son radical $r=\prod_{p|q} p$.
Dans la suite, la variable utilisée dans les produits $\prod$ est toujours $p$, et elle parcourt l'ensemble des nombres premiers vérifiant la condition précisée en-desous du symbole. 
On omettra la condition \enquote{$p$ facteur premier de $q$}, qui sera implicite dans cette notation.
Pour tout ensemble $I$ de nombres premiers, nous noterons
\[
q_I = \prod_{p \in I} p^{m_p}% \quad \text{et} \quad r_I = \prod_{p\in I} p.
.\]
De plus, si $x\in\gl_d(\ZZ)$ et $p$ est un nombre premier, nous noterons $v_p(x)$ la valuation $p$-adique de $x$, i.e. le plus grand entier $n$ tel que $p^{-n}x\in\gl_d(\ZZ)$.
Enfin, pour $a,g \in \SL_d(\ZZ)$, on note $\iota(a)\cdot g = aga^{-1}$; ainsi donc, pour $A \subset \SL_d(\ZZ)$,
\[
\iota(A) \cdot g = \{\, aga^{-1} \mid a\in A\,\}.
\]

La proposition~\ref{xremplit} ci-dessous exprime que si $g\in\Omega$ et $A\subset\Omega$ est tel que $\mu^{*n}(A)\geq q^{-\eps}$, alors l'action par conjugaison de $A$ sur l'élément $g$ permet d'obtenir, en un nombre fini de produits, une part importante du plus petit sous-groupe de congruence contenant $g$.
Sa démonstration occupera le restant de ce paragraphe.

\begin{proposition}
\label{xremplit}
Étant donné $\delta > 0$, il existe $C\geq 0$ tel que l'assertion suivante soit vraie pour tout $\eps>0$ suffisamment petit.
Soit $g \in \Omega$ et $I$ un ensemble de nombres premiers tel que
\[
\forall p \in I,\quad v_p(g - 1) \geq \max\{1 + \delta_2(p), \floor{\delta m_p}\}.
\]
Si $A\subset\Omega$ vérifie $\mu^{*n}(A)\geq q^{-\eps}$ pour $n\geq C\log q$, alors
\[N(\Pi_C (\iota(A)\cdot g),q_I) \geq q^{-O(\eps)} \left(\frac{q_I}{\pgcd(q_I,g-1)}\right)^{\dim G}.\]
\end{proposition}
\begin{proof}
On a $g \in \Omega_{\dot{r}_I}$ où $\dot{r}_I = \prod_{p \in I} p^{1 + \delta_2(p)}$.
Soit $\g$ l'algèbre de Lie de $G$ et $\g(\Z)=\g\cap\gl_d(\Z)$.
Par les lemmes~\ref{isom} et \ref{lm:logg}, on peut écrire
\[
g = \exp(x) \mod q,
\]
avec $x\in\g(\Z)$ satisfaisant
\[
\forall p\in I,\quad v_p(x) = v_p(g - 1).
\]
En particulier $\dot{r}_I | x$ et $\pgcd(q_I,x) = \pgcd(q_I,g-1)$ et par conséquent, la proposition~\ref{irrep} appliquée dans la représentation adjointe $\Ad \colon \Gamma\to\GL(\g(\R))$ (noter que $\Gamma$ préserve le réseau $\g(\Z)$), montre que pour une certaine constante $C_0$,
\[
N(\ssum_{C_0} (\Ad A)\cdot x,q_I)
\geq q^{-C_0 \eps} \left(\frac{q_I}{\pgcd(q_I,g-1)}\right)^{\dim G},
\]
et la proposition~\ref{exp} ci-dessous permet de conclure.
\end{proof}

Il reste à démontrer le résultat que nous avons utilisé dans la démonstration ci-dessus, qui est en fait une conséquence de la formule de Campbell-Hausdorff.

\begin{proposition} 
\label{exp}
Étant donnés $C_0\in\N$ et $\delta > 0$, il existe une constante $C \geq 0$ telle que l'énoncé suivant soit vérifié pour tout entier $q \in \N^*$, tout vecteur $x \in \ssl_d(\Z)$, et toute partie finie $A \subset \SL_d(\Z)$.

Soit $I$ un ensemble de facteurs premiers de $q$ tel que pour tout $p \in I$, on a
\(v_p(x) \geq \max\{1 + \delta_2(p), \floor{\delta m_p}\}\),
alors
\[
N\bigl(\Pi_C \iota(A)\cdot \exp(x),q_I\bigr) \geq \frac{1}{C}N\bigl(\ssum_{C_0} \Ad(A) \cdot x,q_I\bigr).
\]
\end{proposition}

Dans ce qui suit, nous écrirons $e^x$ au lieu de $\exp(x)$.
Pour $s \in \N^*$, le groupe libre engendré par $s$ générateurs $a_1,\dotsc,a_s$ sera noté $F_s$.
On identifie chaque élément de $F_s$ avec l'application de mot qu'il induit sur $\GL_d(\ZZ)$.
Nous noterons aussi $\cF_s(\Z)$ la $\Z$-algèbre de Lie libre engendrée par $s$ générateurs.

\begin{lemma}
\label{mot}
Étant donnés deux entiers naturels non nuls $k$ et $s$, il existe un mot $w \in F_s$ et une constante $D \in \N^*$ tels que pour tout $R \in \N^*$ vérifiant $v_2(R) \neq 1$, pour tous $x_1,\dotsc,x_{s}$ dans $R\gl_d(\ZZ)$,
\[
e^{D(x_1+\dotsb+x_{s})}
\equiv w(e^{x_1},\dotsc,e^{x_{s}}) \mod R^k.
\]
\end{lemma}
\begin{proof}
Il s'agit de reprendre la démonstration de \cite[Lemma~3.5]{abrs1}, en s'assurant, quitte à ajuster la valeur des constantes pour contrôler les dénominateurs, que le reste est bien divisible par $R^k$.
%\noindent
%1. Pour tout $q\in\N^*$, si $r=\rad q$ et $X,Y\in\ssl_d(r\ZZ)$ vérifient $X\equiv Y\mod q$, alors $e^{2X}\equiv e^{2Y}\mod q$.
%Cela découle du calcul donné dans la démonstration du lemme~\ref{isom}.
%1. Étant donné $k$, il existe $C\in\N^*$ tel que si $R| X,Y$ et $q|R^k$ et $X\equiv Y\mod q$, alors $e^{CX}\equiv e^{CY}\mod q$.
%En effet, sous ces hypothèses, pour $C$ assez grand, $\sum_{n\geq k}\frac{(CX)^n}{n!}=0\mod q$.
%Cela se voit en vérifiant, pour chaque $p$ premier divisant $q$, que la valuation $p$-adique de chacun des termes de la somme est supérieure à $v_p(q)$:
%\[
%v_p(\frac{C^nX^n}{n!}) = nv_p(X)+nv_p(C) - v_p(n!)
%\geq v_p(q) + (n-k)v_p(X) + nv_p(C) - v_p(n!) 
%\geq v_p(q)
%\]
%si $C$ est bien choisi (indépendant de $p$). (Si $p$ est grand par rapport à $k$, c'est $(n-k)v_p(X)$ qui l'emporte, et sinon, on compense avec $v_p(C)$.)
%Si de plus on choisit $C$ divisible par $k!$, on trouve bien
%\[
%e^{CX} \equiv \sum_{n<k} \frac{(CX)^n}{n!} \equiv \sum_{n<k} \frac{(CY)^n}{n!} \equiv e^{CY} \mod q.
%\]

Par le lemme chinois, il suffit de traiter le cas où $R$ a un seul facteur premier.
On pourra donc travailler sur $\Z_p$.
%On s'assura que la constante $D$ et le mot $w$ qu'on obtient ne dépend pas de $p$.
Les lettres $C$ et $C'$ seront utilisées pour désigner des quantités dépendant seulement de $k$ et de $s$, et dont la valeur peut varier d'une ligne à l'autre.

La démonstration se fait en trois étapes.

\noindent
1. Soit $k\in\N^*$ et $w\in F_s$.
Il existe $C\in\N^*$ et une relation $\sr \in\cF_s(\Z)$ de degré strictement inférieur à $k$ telle que pour tous $x_1,\dotsc,x_s\in R \gl_d(\ZZ)$, 
\[
w(e^{C x_1},\dotsc,e^{C x_s}) \equiv e^{\sr(x_1,\dotsc,x_s)} \mod R^k.
\]
Cela se voit par récurrence sur la longueur du mot.
Le résultat est trivial pour le mot vide, de longueur nulle.
Supposons le résultat connu pour tous les mots de longueur au plus $\ell$.
Soit $w'$ un mot de longueur $\ell+1$.
Pour un certain $j$, $w'=a_jw$, avec $w$ de longueur $\ell$.
L'hypothèse de récurrence permet d'écrire, pour tous $x_1,\dots,x_s\in R \gl_d(\ZZ)$, $w(e^{C x_1},\dotsc,e^{C x_s}) \equiv e^{\sr(x_1,\dotsc,x_s)}\mod R^{k}$.
Alors, quitte à ajuster la valeur de $C$, avec la formule de Campbell-Hausdorff,
\begin{align*}
w'(e^{C^2X_1},\dots,e^{C^2X_s})
& \equiv e^{C^2X_j}w(e^{C^2X_1},\dots,e^{C^2X_s})\\
& \equiv e^{C^2X_j}e^{\sr(CX_1,\dots,CX_s)}\\
& \equiv e^{\sr'(X_1,\dots,X_s) + U} \mod R^k
\end{align*}
avec $U\equiv 0\mod R^k$. (La constante $C$ permet d'absorber les dénominateurs qui apparaissent dans le reste de la formule de Campbell-Hausdorff, cf. Serre \cite[page 29]{serre_lalg}).
%ou lemme~\ref{lm:BCHexp1} avec la remarque qui le suit
Grâce au lemme~\ref{lm:expZp}, on trouve donc bien
\[
w'(e^{C'x_1},\dotsc,e^{C'x_s}) \equiv e^{\sr'(x_1,\dotsc,x_s)} \mod R^k.
\]

\noindent
2. Remarquons que si $x_1,\dotsc,x_k$ sont dans $R\gl_d(\ZZ)$, alors
\[
e^{[Cx_1,[\dots,[Cx_{k-1},Cx_k]\dots]]} \equiv [e^{Cx_1},[\dots,[e^{Cx_{k-1}},e^{Cx_k}]\dots]]] \mod R^{k+1}.
\]
Cela se voit par récurrence, avec la formule de Campbell-Hausdorff. (Ici encore, il faut contrôler les dénominateurs qui apparaissent dans le reste, d'où la constante $C$.)
Par suite, utilisant encore la formule de Campbell-Hausdorff, si $\sr_k$ est un élément de $\cF_s(\Z)$ homogène de degré $k$ on peut écrire, pour $x_1,\dotsc,x_s$ dans $R\gl_d(\ZZ)$,
\begin{equation}\label{comm}
e^{\sr_k(Cx_1,\dotsc,Cx_s)} \equiv u(e^{x_1},\dotsc,e^{x_s}) \mod R^{k+1}.
\end{equation}
Le mot $u$ est obtenu comme le produit des commutateurs constituant $\sr_k$.

%Cela va nous permettre de \enquote{corriger} le mot $w_\ell$ et d'améliorer le degré de précision.
\noindent
3. On construit par récurrence sur $k$ un mot $w_k$ et une constante $C=C_k$ tels que pour tous $x_1,\dotsc,x_s\in R\gl_d(\ZZ)$,
\begin{equation}\label{hr}
e^{C(x_1+\dotsb+x_s)} \equiv w_k (e^{x_1},\dotsc,e^{x_s}) \mod R^k.
\end{equation}
Supposons construit le mot $w_k$ satisfaisant \eqref{hr}.
D'après le 1., nous pouvons écrire, pour $\sr\in\cF_s$ de degré inférieur à $k$,
%inférieur ou égal, bien entendu !
\[
w_k(e^{Cx_1},\dotsc,e^{Cx_s}) \equiv e^{\sr(x_1,\dots,x_s)} \mod R^{k+1}.
\]
Comme $w_k(e^{Cx_1},\dotsc,e^{Cx_s}) \equiv e^{C^2(x_1+\dotsb+x_s)}\mod R^k$, on a naturellement 
\[
\sr(x_1,\dots,x_s) = C^2(x_1+\dotsb+x_s) + \sr_k(x_1,\dotsc,x_s),
\]
 où $\sr_k \in \cF_s(\Z)$ est un élément homogène de degré $k$.
D'après \eqref{comm}, il existe un mot $u$ tel que $e^{-\sr_k(Cx_1,\dotsc,Cx_s)}=u(e^{x_1},\dotsc,e^{x_s})\mod R^{k+1}$ et par suite
\begin{align*}
e^{C'(x_1+\dotsb+x_s)}
& \equiv e^{\sr(Cx_1,\dotsc,Cx_s)-\sr_k(Cx_1,\dotsc,Cx_s)}\\
& \equiv e^{\sr(Cx_1,\dotsc,Cx_s)} e^{-\sr_k(Cx_1,\dotsc,Cx_s)}\\
& \equiv w_k(e^{C^2x_1},\dotsc,e^{C^2x_s})u(e^{x_1},\dotsc,e^{x_s}) \mod R^k,
\end{align*}
ce qu'il fallait démontrer.
\end{proof}

Nous pouvons maintenant démontrer la proposition~\ref{exp}.

\begin{proof}[Démonstration de la proposition~\ref{exp}]
Soit $k=\left\lceil\frac{2}{\delta}\right\rceil$ et
\[
R = \prod_{p\in I} p^{\max\{1 + \delta_2(p), \floor{\delta m_p}\}},
\]
de sorte que $R|X$ et $q_I|R^k$.
D'après le lemme~\ref{mot} il existe une constante $D\in\N^*$ et un mot $w \in F_{C_0}$ tels que pour tous $x_1,\dotsc,x_{C_0}$ dans $R\ssl_d(\ZZ)$, 
\[
e^{D(x_1+\dotsb+x_{C_0})} \equiv w(e^{x_1},\dotsc,e^{x_{C_0}}) \mod q_I.
\]
Cette égalité vaut en particulier si $x_i = \Ad(a_i)\cdot x$, avec $a_i \in A$.
Mais on a $e^{x_i}= \iota(a_i) \cdot e^x$ et par conséquent, si $C$ désigne la longueur du mot $w$,
\[
\exp(D \ssum_{C_0} \Ad(A) \cdot x) \subset \Pi_C (\iota(A)\cdot e^x) \mod q_I.
\]
Avec le lemme~\ref{isom}, on trouve bien
\begin{align*}
N(\ssum_{C_0}\Ad(A) \cdot x,q_I)
& \ll_D N(D\ssum_{C_0}A\cdot x,q_I)\\
& = N(\exp(D\ssum_{C_0}A\cdot x),q_I)\\
& \leq N(\Pi_C(\iota(A)\cdot e^x),q_I).
\end{align*}
Cela conclut la démonstration de la proposition puisque $D$ ne dépend que de $C_0$ et de $\delta$.% par le lemme~\ref{mot}.
\end{proof}

\section{Croissance des ensembles de grande $\mu^{*n}$-masse}
\label{sec:croissance}

%Afin d'obtenir la contradiction souhaitée pour la démonstration du théorème, nous montrons que les conditions \eqref{ag} ne peuvent pas être simultanément satisfaites.
Nous voulons maintenant conclure la démonstration de la proposition~\ref{global}. 
Grâce aux résultats des parties~\ref{sec:sqf} et \ref{sec:torus}, le problème se ramène à la construction d'un élément $g$ dans un ensemble produit $A^C$ qui satisfait certaines congruences.
Pour cela, l'instrument principal est la propriété presque diophantienne de $\mu^{*n}$, mais la mise en place des différents paramètres est subtile.
Nous reprenons en grande partie l'argument de Bourgain et Varjú \cite[\S5]{bv}, avec quelques changements notables.

\bigskip

Dans toute la suite, $\mu$ désigne une probabilité symétrique à support fini sur $\SL_d(\Z)$, $\Gamma$ le sous-groupe engendré par le support de $\mu$, $G$ l'adhérence de Zariski de $\Gamma$ dans $\SL_d(\Z)$, et $\Omega$ l'adhérence de $\Gamma$ dans le groupe profini $\SL_d(\ZZ)$.
Le groupe algébrique $G$ est supposé simple, connexe, et simplement connexe.
On note $q$ un entier arbitraire, dont la décomposition en facteurs premiers s'écrit
\[
q=\prod_p p^{m_p},
\]
et $r$ son radical:
\[
r=\prod_{p|q} p.
\]
Pour tout ensemble de nombres premiers $I$, nous noterons
\[
q_I = \prod_{p \in I} p^{m_p} \quad \text{et} \quad r_I = \prod_{p\in I} p.
\]

\subsection{Les conditions de congruences}
\label{ss:xcond}

%Pour pouvoir appliquer la proposition~\ref{xremplit}, il nous faudra partir d'un élément $g$ appartenant à $\Omega_r$, où $r=\rad q$.
%Par conséquent, cet argument nous permettra au mieux de construire un ensemble riche dans le sous-groupe de congruence $\Omega_r$, où $r=\rad q$.
%Ensuite, il nous faudra encore combiner cela avec le théorème de Salehi Golsefidy et Varjú \cite[Theorem~1]{sgv}, qui donne un ensemble riche dans le quotient $\Omega/\Omega_r$.
%Nous présentons maintenant le détail de l'argument, partant d'un vecteur $X$ qui satisfait certaines conditions, et dont nous montrerons l'existence au paragraphe suivant.

Les conditions de divisibilité convenables pour l'élément $g$ recherché sont celles de la proposition ci-dessous, qui sera démontrée dans les paragraphes suivants.

\begin{proposition}
\label{constructionx}
Étant donné $\tau > 0$, il existe des constantes $C > 1$ et $\delta > 0$ telles que l'assertion suivante soit vraie pour tout entier naturel $q$ suffisamment grand.
Si $n\geq C\log q$, et $A \subset \SL_d(\Z)$ est une partie symétrique vérifiant $\mu^{*n}(A) \geq q^{-\delta}$, alors il existe $g \in \Pi_C A$ vérifiant
\begin{equation}\label{xgrand}
\pgcd(q,g-1) \leq q^{O(\tau)} r
\end{equation}
et 
\begin{equation}\label{xpetit}
q_I \geq q^{1-O(\tau)}
\end{equation}
avec 
\[I = \bigl\{\, p \mid v_p(g-1) \geq \max\{1, \floor{\delta m_p}\}\,\bigr\}.\]
\end{proposition}

Admettant ce résultat pour le moment, nous pouvons aisément démontrer la proposition~\ref{global}.

\begin{proof}[Démonstration de la proposition~\ref{global}]
Soit $g \in \Pi_C A$ et $I$ l'ensemble de premiers donnés par la proposition~\ref{constructionx}.
Si $2 \in I$ et $\delta m_2 < 2$, on peut remplacer $I$ par $I\setminus\{2\}$ pour avoir 
\[\forall p \in I,\quad v_p(g  - 1) \geq \max\{1 + \delta_2(p), \floor{\delta m_p}\},\]
tout en gardant les conditions \eqref{xgrand} et \eqref{xpetit}, si $q$ est suffisamment grand.
La proposition~\ref{xremplit}, le lemme~\ref{isom} et le lemme~\ref{lm:logg} nous donnent alors 
\[
N(\Pi_C (\iota(A)\cdot g),q_I) \geq q^{-O(\tau)} \left(\frac{q_I}{r_I}\right)^{\dim G} \geq  q^{-O(\tau)} N(\Omega_{r_I},q_I).
\]
Or, $\Pi_C (\iota(A)\cdot g) \subset (\Pi_{C'} A) \cap\Omega_{r_I}$, car $g \in (\Pi_C A) \cap \Omega_{r_I}$.
Quitte à ajuster la valeur de $C$, on trouve donc
\[
N((\Pi_C A) \cap\Omega_{r_I}, q_I)
\geq q^{-O(\tau)}  N(\Omega_{r_I},q_I).
\]
Mais d'autre part, comme $r_I$ est sans facteur carré, d'après le lemme~\ref{lm:pirA},
\[
N( \Pi_C A,r_I) \geq q^{-O(\eps)} N(\Omega,r_I).
\]
Mis bout à bout, cela montre
\begin{align*}
N(\Pi_{2C}A, q_I) &\geq N(\Pi_C A,r_I) N( (\Pi_C A) \cap\Omega_{r_I},q_I)\\
&\geq q^{-O(\tau)} N(\Omega,r_I) N(\Omega_{r_I},q_I)\\
&\geq q^{-O(\tau)} N(\Omega,q_I),
\end{align*}
Comme $q_I \geq q^{1-O(\tau)}$, et donc $N(\Omega,q_I) \geq q^{-O(\tau)} N(\Omega,q)$, cela démontre la proposition~\ref{global}.
\end{proof}

\subsection{Propriété presque diophantienne.}

Pour démontrer la proposition~\ref{constructionx}, nous utiliserons une propriété importante de non concentration pour la loi $\mu^{*n}$ de la marche aléatoire au temps $n$.

\begin{proposition}[Propriété presque diophantienne]
\label{presquedioph}
Soit $\mu$ une probabilité symétrique sur $\SL_d(\Z)$ dont le support est fini et engendre un sous-groupe $\Gamma$ non moyennable.
Il existe des constantes $C > 0$ et $c > 0$ telles que pour tout entier $q \in \N^*$,
\[
\forall n \geq C \log q,\quad \mu^{*n}(\Omega_q) \leq C q^{-c}.
\]
\end{proposition}
\begin{proof}
Comme $\Gamma$ est non moyennable, il existe $c' > 0$ tel que
\[\forall n \geq 1,\quad \max_{g \in \Gamma} \mu^{*n}(g) \ll e^{-c'n}.\]
Posons $M = \max_{g \in \Supp(\mu)} \norm{g}$ et $m = \lfloor \frac{\log q}{2\log M}\rfloor$.
Nous avons, pour tout $g \in \Supp(\mu^{*2m})$, $\norm{g} < q$.
Donc $\Supp(\mu^{*2m}) \cap \Omega_q = \{1\}$ et 
par conséquent, pour tout $g \in \Gamma$, $\Supp(\mu^{*m}) \cap g\Omega_q$ contient au plus un élément. 
Par suite,
\[\sup_{g\in \Gamma} \mu^{*m}(g\Omega_q) \leq  \max_{g \in \Gamma} \mu^{*m}(g) \ll e^{-c'm} \ll q^{-c}\]
avec $c = \frac{c'}{2\log M}$.
Enfin, pour tout entier $n \geq m$,
\[\mu^{*n}(\Omega_q) = \sum_{g\in \Gamma} \mu^{*(n-m)}(g^{-1}) \mu^{*m}(g\Omega_q) \leq \sup_{g \in \Gamma} \mu^{*m}(g\Omega_q) \ll q^{-c}.\]
\end{proof}

Pour $\delta\in]0,1[$, nous noterons 
\[
q_\delta = \prod p^{\floor{\delta m_p}}.
\]
À l'aide de la propriété presque diophantienne, il est facile de trouver un élément $g \in \Pi_C A$ vérifiant \eqref{xgrand} et $q_\delta | g - 1$ à la place de \eqref{xpetit}.

\begin{lemma}
\label{constructionx0}
Étant donné $\tau > 0$, il existe des constantes $C, \delta > 0$ telles que l'assertion suivante soit vraie.
Si $n\geq C\log q$ et $A \subset \SL_d(\Z)$ est une partie vérifiant $\mu^{*n}(A) \geq q^{-\delta}$, alors il existe un élément $g_0 \in AA$ tel que
\begin{equation*}
\pgcd(q, g_0 - 1) \leq q^{\tau}
\quad \text{et} \quad 
q_\delta |g_0 -1.
\end{equation*}
\end{lemma} 

\begin{proof}
D'une part, par l'inégalité de Schwarz, et avec la symétrie de $\mu$ et de $A$,
\begin{align*}
\mu^{*2n}(AA\cap\Omega_{q_\delta}) & \geq \sum_{g \in \Omega/\Omega_{q_\delta}} \mu^{*n}(A \cap g \Omega_{q_\delta})^2\\
& \geq {[\Omega : \Omega_{q_\delta}]}^{-1} \Bigl(\sum_{g \in \Omega/\Omega_{q_\delta}} \mu^{*n}(A \cap g \Omega_{q_\delta}) \Bigr)^2\\
& \geq {[\Omega : \Omega_{q_\delta}]}^{-1} \mu^{*n}(A)^2.
\end{align*}
On majore l'indice ${[\Omega : \Omega_{q_\delta}]}$ simplement par $q_\delta^{d^2} \leq q^{d^2\delta}$, d'où l'on tire
\[
\mu^{*2n}(AA\cap\Omega_{q_\delta}) \geq q^{-(d^2+2)\delta}.
\]
D'autre part, par le lemme~\ref{presquedioph}, il existe $c > 0$ dépendant seulement de $\mu$ tel que 
\[
\sum_{s|q \text{ et } s \geq q^\tau} \mu^{*2n}(\Omega_s) \leq \sum_{s|q \text{ et } s \geq q^\tau} s^{-c} %\\
\leq \sum_{s|q} q^{-c\tau} %\\
\leq q^{-\frac{c\tau}{2}}.
\]
Si $\delta > 0$ est choisi tel que  $(d^2+2)\delta < \frac{c\tau}{2}$, l'ensemble  $(AA\cap\Omega_{q_\delta}) \setminus \bigcup_{s|q \text{ et } s \geq q^\tau} \Omega_s$ est nécessairement non vide.
\end{proof}

\subsection{Construction de l'élément $g$}

Pour aider à la compréhension de la démonstration un peu technique qui va suivre, nous commençons par en donner une interprétation plus imagée.
Pour cela, nous représentons l'entier $q=\prod p^{m_p}$ sous la forme du graphe de la fonction $p\mapsto m_p$.
Le support de cette fonction est l'ensemble $\cP$ des diviseurs premiers de $q$, que l'on munit de la mesure $\nu$ définie par $\nu(p)=\log p$.
Ensuite, à un élément $g\in\Omega$ nous associons la fonction $f$ définie sur $\cP$ par $f(p)=v_p(g-1)$.
En termes de la fonction $f$, les conditions de la proposition~\ref{constructionx} deviennent grosso modo:
\begin{enumerate}
\item \label{un} $\frac{2}{\delta}f(p)\geq m_p$ pour tout $p$ (hors d'un ensemble exceptionnel $^cI$ tel que $\int_{^cI}m_p\dd\nu(p)\leq\tau\log q$);
\comm{Cette condition permet d'appliquer la formule de Campbell-Hausdorff à un odre borné par $\frac{2}{\delta}$, en arrivant à la précision donnée par $q$.}
\item $\int f(p)\dd\nu(p) \leq \log r + \tau\log q = \int 1\dd\nu(p)+\tau\int m_p\dd\nu(p)$.
\comm{Cela nous permet de récupérer à partir de $g$, via la proposition~\ref{xremplit}, la quasi-totalité de $q$, excepté bien sûr le radical $r$ inaccessible par cette méthode.}
\end{enumerate}

La construction de l'élément $g$ se fait en trois étapes, illustrées dans la figure~\ref{vp}, et décrites grossièrement comme suit.
\begin{enumerate}
\item [(a)] La proposition~\ref{presquedioph} donne l'existence d'un élément $g_0$ tel que la fonction $f_0$ associée satisfasse $f_0(p)\geq\lfloor\delta m_p\rfloor$ et $\int f_0(p)\dd\nu(p) \leq \log r + \tau\log q$.
\item [(b)] Pour avoir la condition~\ref{un}, on doit corriger $g_0$ aux places $p$ où $\lfloor\delta m_p\rfloor=0$. Cela se fait à l'aide de la proposition~\ref{sqf}. On obtient un élément $g_1$ dont la fonction associée $f_1$ vérifie $f_1(p)=f_0(p)$ si $\delta m_p\geq 1$, et pour tout $p\in\cP,\ f_1(p)\geq 1$.
\item [(c)] Malheureusement, en passant de $g_0$ à $g_1$ on perd la propriété $\int f_1(p)\dd\nu(p)\leq\log r+\tau\log q$. Il faut donc réduire la valuation $p$-adique de $g_1-1$ aux places $p$ telles que $\delta m_p<1$. Cela se fait à l'aide du lemme~\ref{nosec2}, et permet d'obtenir l'élément $g$ désiré.
\end{enumerate}

\begin{figure}[H]
\begin{center}
\begin{tikzpicture}
\draw[->] (-.5,0) -- (10,0) node[anchor=north]{\tiny{$p$}};
\draw[->] (0,-.5) -- (0,5) node[anchor=east]{\tiny{$m$}};

\draw[gray,-] (-.2,.5) node[anchor=east]{\tiny{1}} -- (9.6,.5);
\draw[black, thick] (0,2) -- (2,2) -- (2,1.5) -- (4,1.5) -- (4,3) -- (5,3) -- (5,5) -- (7,5) -- (7,4) -- (8,4) -- (8,1.6) -- (9.9,1.6) node[anchor=west]{\tiny{$q:m=m_p$}};
\draw[red, thick] (0,.46) -- (2,.46) -- (2,0) -- (4,0) -- (4,.75) -- (5,.75) -- (5,1.2) -- (7,1.2) -- (7,1) -- (8,1) -- (8,0) -- (9.9,0) node[anchor=west]{\tiny{$g_0:m=\lfloor\delta m_p\rfloor$}};
\draw[green!90!red, thick] (0,.5) -- (2.04,.5) -- (2.04,.5) -- (3.96,.5) -- (3.96,.79) -- (4.96,.79) -- (4.96,1.24) -- (7.04,1.24) -- (7.04,1.04) -- (8.04,1.04) -- (8.04,.5) -- (9.9,.5) node[anchor=west]{\tiny{$g:m=\max(1,\lfloor\delta m_p\rfloor)$}};
\draw[blue, thick] (0,.54) -- (2.3,.54) -- (2.3, 2) -- (3,2) -- (3, .83) -- (3.92,.83) -- (4.92,.83) -- (4.92,1.28) -- (7.08,1.28) -- (7.08,1.08) -- (8.08,1.08) -- (8.08,.54) -- (8.4,.54) -- (8.4,1.5) -- (9,1.5) -- (9,1) -- (9.9,1) node[anchor=west]{\tiny{$g_1:m\geq \max(1,\lfloor\delta m_p\rfloor)$}};
%\foreach \x in {0,...,4}
%{
%\draw (\x,-0.1) -- (\x,0.1) node[anchor=south] {\tiny{\x}};
%}
%%\draw[red,thin] (0,0) -- (1,-0.3) -- (2,0) -- (3,-1.95) -- (4,0.05);
%\filldraw[red] (0,0) circle (1pt);
%\filldraw[red] (1,-0.3) circle (1pt);
%\filldraw[red] (2,0) circle (1pt);
%\filldraw[red] (3,-2) circle (1pt);
%\filldraw[red] (4,0) circle (1pt);
%\draw[red] (0,0) node[cross=2pt,thick]{};
%\draw[red] (1,-0.3) node[cross=2pt,thick]{};
%\draw[red] (2,0) node[cross=2pt,thick]{};
%\draw[red] (3,-2) node[cross=2pt,thick]{};
%\draw[red] (4,0) node[cross=2pt,thick]{};
%\draw[red] (2.7,-0.5) node {$Y_0$};
%\draw[blue,thin] (0,0) -- (3,-2) -- (4,0);
%\filldraw[blue] (0,0) circle (1pt);
%\filldraw[blue] (3,-2) circle (1pt);
%\filldraw[blue] (4,0) circle (1pt);
%\draw[blue] (0,0) node[cross=2pt]{};
%\draw[blue] (3,-2) node[cross=2pt]{};
%\draw[blue] (4,0) node[cross=2pt]{};
%\draw[blue] (3.8,-1) node {$Y$};
\end{tikzpicture}
\caption{Valuations des éléments $g_0$, $g_1$, $g$.}
\label{vp}
\end{center}
\end{figure}
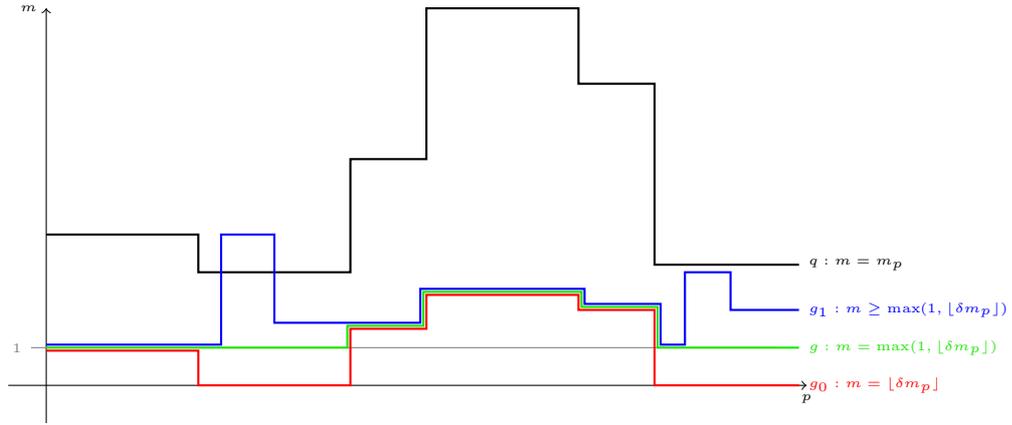

Nous donnons enfin la démonstration rigoureuse de la proposition~\ref{constructionx}. 
Soit $\tau>0$ fixé.
La construction de $g$ dépend de certaines quantités $\alpha > \beta > \gamma> \delta>0$.
Pour le choix de ces quantités, nous utilisons un paramètre auxiliaire $L \in \N$, et procédons en plusieurs étapes.
Au départ, $L=1$, et ensuite, à chaque tentative infructueuse, nous ajustons la valeur de $L$.
Il est important de noter que si la valeur finale de $L$ dépend de $q$, elle est toutefois bornée indépendamment de $q$.

Dans cette démonstration, $C$ désigne une constante dépendant de $\mu$ et de $\tau$ et $A \subset \SL_d(\Z)$ est une partie symétrique satisfaisant $\mu^{*n}(A)\geq q^{-\delta}$ pour certain $n \geq C \log q$.

\begin{enumerate}
\item \label{it:1}
 Choisissons $\alpha>0$ tel que $L \alpha < \tau$.
\item \label{it:2}
 Choisissons $\beta > 0$ tel que, pour $n \geq C \log q$, si $A' \subset \SL_d(\Z)$ est une partie symétrique et $\mu^{*n}(A')\geq q^{-\beta}$, alors il existe $r'|r$ avec $r' \geq q^{- \alpha}r$ tel que 
 \[\pi_{r'}(\Pi_C A') = \Omega/\Omega_{r'}.\]
Cela est possible, d'après la proposition~\ref{sqf}.
\item \label{it:3}
 Choisissons $\gamma>0$ tel que $\mu^{*2n}(AA\cap \Omega_{q_\gamma}) \geq q^{-\beta}$.
%Cela est possible car on peut recouvrir $A$ par $q^{d^2\eps_3}$ translatés de $\Omega_{q_{\eps_3}}$, et alors, il suffit de choisir le translaté de $\mu_n$-mesure maximale, puis de le recentrer en l'identité.
Il suffit de prendre $\gamma = \frac{\beta}{d^2+2}$. La preuve est identique à la première moitié de la démonstration du lemme~\ref{constructionx0}.
%\item \label{it:4}
% Choisissons $\eps_4>0$ satisfaisant la propriété suivante. Soit $X$ tel que, pour tout $q_4|q$ vérifiant $q_4\geq q^{\eps_4}$, on ait $q_4\nmid X$ ; si $J$ désigne l'ensemble des $p$ tels que $p^{\lceil\eps_3 m_p\rceil}|X$, alors $q_J\leq q^{\eps_0}$. (Comme $q_J^{\eps_3}|X$, on doit avoir $q_J^{\eps_3}\leq q^{\eps_4}$, et il suffit de prendre $\eps_4=\eps_0\eps_3$.)
\item \label{it:4}
Choisissons $\delta > 0$ tel qu'il existe $g_0$ dans $AA$ satisfaisant $q_{\delta} | g_0 - 1$ tandis que $\pgcd(q,g_0 - 1) \leq q^{\alpha \gamma}$.
Cela est possible, d'après le lemme~\ref{constructionx0} appliqué avec $\tau=\alpha \gamma$.
%Cela est possible car $\mu_n$ satisfait une condition presque diophantienne $\mu_n(\Omega_{q_3})\leq q_3^{-\tau}$, ce qui permet de majorer $\mu_n(\bigcup_{q^{\eps_3^2}\leq q_3|q}\Omega_{q_3})\leq (\log q)q^{-\tau\eps_3^2}$.
%Ensuite, on recouvre $A$ par des translatés de $\Omega_{q_{\eps_4}}$, et on choisit un translaté de $\mu_n$-mesure maximale.
%Il suffit donc de prendre $\eps_4$ tel que $\eps+d^2\eps_4< \frac{\tau\eps_3^2}{2}$.
\item Si $\prod_{L<p\leq 1/\delta} p^{m_p} \leq q^{\tau}$, les quantités $\alpha$, $\beta$, etc. sont fixées pour le restant de la démonstration. Sinon, nous reprenons la construction ci-dessus en partant de $L=\floor{1/\delta}$.
\end{enumerate}

Vue la condition $\prod_{L<p\leq 1/\delta} p^{m_p} \geq q^{\tau}$ obtenue en cas d'échec, le nombre de tentatives nécessaires pour conclure cette procédure est majoré, indépendamment de $q$, par $1/\tau$.
Cela montre que la quantité finale $\delta$ est minorée par une constante indépendante de $q$.

%Partant de $x_0$, nous voulons construire un élément $x$ dans $\Pi_C A$ qui satisfasse aux conditions \eqref{xgrand} et \eqref{xpetit} données en \S\ref{ss:xcond}.
Partons de l'élément $g_0$ obtenu au point~\ref{it:4}, qui satisfait
\begin{equation*}
\forall p |q, \quad v_p(g_0 - 1) \geq \floor{\delta m_p}
\end{equation*}
et 
\begin{equation}
\label{x0grand}
\pgcd(q,g_0 - 1) \leq q^{\alpha \gamma}.
\end{equation}

Nous devons en premier lieu faire en sorte que $g-1$ soit divisible par presque tous les facteurs premiers de $q$, afin qu'il puisse vérifier la propriété \eqref{xpetit}.
Avec le point \ref{it:3}, le point \ref{it:2} appliqué à $A' = AA\cap \Omega_{q_\gamma}$ montre que l'on peut trouver $r = r' r_0$ avec $r_0 \leq q^{\alpha}$ tel que 
\begin{equation}
\label{r'plein}
\pi_{r'}( (\Pi_C A)\cap\Omega_{q_\gamma})= \Omega/\Omega_{r'}.
\end{equation}
En particulier il existe $g_0' \in (\Pi_C A) \cap \Omega_{q_\gamma}$ tel que 
\[g_0'\equiv g_0\mod r'.\]
Posons $g_1=g_0'g_0^{-1}$, et montrons que $g_1$ satisfait
\begin{equation}
\label{x1petit}
q_{I_1} \geq q^{1- O(\tau)},
\end{equation}
où
\[
I_1 = \bigl\{\, p \mid v_p(g_1-1) \geq \max\{1, \floor{\delta m_p}\} \,\bigr\}.
\]
Tout d'abord, $g_1 \in \Omega_{q_\delta} \cap \Omega_{r'}$, et donc, pour tout $p |q$, $p \notin I_1$ implique $p| r_0$ et $\delta m_p < 1$, d'où
\[
\frac{q}{q_{I_1}} = \prod_{p \notin I_1} p^{m_p} \leq \prod_{p| r_0 \text{ et } \delta m_p<1} p^{m_p}.
\]
En séparant les premiers $p$ selon $m_p \leq L$ ou non, on trouve
\[
\frac{q}{q_{I_1}} \leq r_0^L  \prod_{L < m_p \leq 1/\delta} p^{m_p} \leq q^{L\alpha + \tau} \leq q^{O(\tau)},
\]
ce qui démontre \eqref{x1petit}.
Il ne nous resterait donc qu'à montrer l'inégalité $\pgcd(q,g_1-1)\leq q^{O(\tau)}r$.
Malheureusement, l'élément $g_1$ n'a pas de raison de vérifier cette propriété, à cause des nombres premiers $p$ tels que $\gamma m_p<1$ et $v_p(g_1-1)\geq 2$.
La suite de la démonstration a pour but de corriger $g_1$ en ces places ; cela va nous contraindre à un petit détour.

Soit 
\[ J = \{\, p \mid \gamma m_p \geq 1 \text{ ou } v_p(g_1 - 1) \leq 1 \,\}.\]
Montrons que
\begin{equation}
\label{x1grand}
\pgcd(q_J, g_1-1)  \leq q^{O(\alpha)} r_J.
\end{equation}
On peut écrire $J = J_1 \cup J_2 \cup J_3$ avec 
\[J_1 = \{\, p \mid \gamma m_p \geq 1 \text{ et } v_p(g_0 - 1) \geq \floor{\gamma m_p} \,\},\]
\[J_2 = \{\, p \mid \gamma m_p \geq 1 \text{ et } v_p(g_0 - 1) < \floor{\gamma m_p}\} \,\]
et 
\[J_3 = \{\, p \mid v_p(g_1 - 1) \leq 1 \,\}.\]
Remarquons que si $\gamma m_p \geq 1$ alors $\floor{\gamma m_p} \geq \frac{\gamma}{2} m_p$. En particulier, pour tout $p \in J_1$, $m_p \leq \frac{2}{\gamma}v_p(g_0 - 1)$. Donc, avec \eqref{x0grand},
\[q_{J_1} \leq \prod_{p \in J_1} p^{\frac{2}{\gamma}v_p(g_0 - 1)} \leq \pgcd(q,g_0 - 1)^{\frac{2}{\gamma}} \leq q^{2\alpha}.\]
Pour tout $p \in J_2$, $v_p(g_0-1) < v_p(g_0'-1)$ donc $v_p(g_1 - 1) = v_p(g_0 - 1) < \floor{\gamma m_p}$. Donc
\[\pgcd(q_{J_2},g_1-1) \leq q_{J_2}^\gamma \leq q^\alpha.\] 
Par la définition de $J_3$, 
\[\pgcd(q_{J_3},g_1-1) \leq r_{J_3} \leq r_J.\]
Enfin,  on trouve \eqref{x1grand} en combinant les trois inégalités ci-dessus avec la suivante
\[
\pgcd(q_J,g_1-1) \leq q_{J_1} \pgcd(q_{J_2},g_1-1) \pgcd(q_{J_3},g_1-1).
\]
La proposition~\ref{xremplit} avec \eqref{x1grand} montre donc que
\begin{equation}\label{jplein}
N \bigl( \Pi_C (\iota(A) \cdot g_1), q_J \bigr) 
\geq q^{-O(\alpha)} \left(\frac{q_J}{r_J}\right)^{\dim G}.
\end{equation}
Par ailleurs, l'égalité \eqref{r'plein} et le lemme~\ref{nosec2} appliqué à $r'$ montrent qu'on peut écrire $r'= r_1 r''$ avec $r_1 \leq q^{\alpha}$ tel qu'il existe $g_2$ dans $(\Pi_C A) \cap \Omega_{q_\gamma}$ vérifiant
\[g_2 \in \Omega_{r'} \quad \text{ et } \quad \forall p|r'',\ v_p(g_2-1) = 1.\]
L'inégalité \eqref{jplein} permet de choisir $g_3$ dans $\Pi_C (\iota(A)\cdot g_1)$ tel que $g=g_2g_3$ vérifie $\pgcd(q_J,g - 1)\leq q^{O(\alpha)}r_J$.
En effet, 
\[
N(g_2 \Pi_C(\iota(A)\cdot {g_1}),q_J) \geq q^{-O(\alpha)} \left(\frac{q_J}{r_J}\right)^{\dim G},
\]
tandis que pour tout entier $1 \leq s \leq q_J$,
\[
N(\{\, g \in \Omega \mid \pgcd(g - 1,q_J)\geq s \,\},q_J) \ll (\log q) \left(\frac{q_J}{s}\right)^{\dim G}.
\]

Par ailleurs, $p\notin J$ signifie $\gamma m_p < 1$ et $g_1 \in \Omega_{p^2}$ donc $g_3 \in \Omega_{p^2}$.
Si de plus $p| r''$, alors $v_p(g_2-1) = 1$ donc $v_p(g - 1) = 1$.
Ainsi,
\begin{align*}
\pgcd(q, g - 1)
& \leq \pgcd(q_J, g - 1) \biggl(\prod_{p\not\in J \text{ et } p |r''} p\biggr) \biggl( \prod_{p |r_0r_1 \text{ et } \gamma m_p < 1} p^{m_p}\biggr)\\
& \leq q^{O(\alpha)}r_J \biggl(\prod_{p\not\in J \text{ et } p |r} p\biggr) (r_0r_1)^{L} \biggl( \prod_{L < m_p \leq 1/\gamma} p^{m_p} \biggr)\\
& \leq q^{O(\alpha)} r q^{2L\alpha} q^\tau = q^{O(\tau)}r.
\end{align*}
Cela montre que $x$ vérifie la condition~\eqref{xgrand}.

D'autre part, on a $g_1 \in \Omega_{q_\delta} \cap \Omega_r$ et $g_2 \in \Omega_{q_\gamma} \cap \Omega_{r'}$.
Donc $g_3 \in \Omega_{q_\delta} \cap \Omega_r$ puis $g \in \Omega_{q_\delta} \cap \Omega_{r'}$.
Avec le même argument que celui pour \eqref{x1petit}, on montre que $g$ vérifie la condition~\eqref{xpetit}.

%\section*{Conclusion}

\appendix

\section{Approximation dans les groupes semi-simples}
\label{sec:prelim}

Nous résumons dans cet appendice les résultats obtenus par Matthews, Vaserstein et Weisfeiler \cite{MVW} et Nori \cite{Nori} sur l'approximation dans les groupes algébriques simples, ainsi que quelques autres propriétés que nous avons utilisées dans le corps de l'article.
Étant donné un sous-groupe $\Gamma$ dans $\SL_d(\Z)$, on s'intéresse à son adhérence $\Omega$ dans $\SL_d(\ZZ)$.
Pour parler sans ambiguïté des points sur $\Z/q\Z$ de l'adhérence de Zariski de $\Gamma$, nous commençons par introduire quelques éléments de langage de la théorie des schémas en groupe.

\subsection{Le schéma en groupes $G$}

Dans l'algèbre $\Z[X_{11},\dotsc,X_{dd}]$ des polynômes à coefficients entiers sur les matrices $d\times d$, on considère l'idéal $\cI$ défini par 
\[
\cI = \{\,f \in \Z[X_{11},\dotsc,X_{dd}] \mid \forall g \in \Gamma,\, f(g) = 0 \,\}
\]
et le foncteur
\[
G = \Hom(\Z[X_{ij}] / \cI, \,\mathbf{-}\,)
\]
de la catégorie des anneaux commutatifs unifères dans la catégorie des ensembles.
Si $R$ est un anneau commutatif unifère, alors $G(R)$ est l'ensemble des morphismes d'anneau de $\Z[X_{ij}] / \cI$ dans $R$, et si $\phi \colon R \to R'$ est un morphisme, alors $G(\phi)$ est la composition par $\phi$.
De manière équivalente, 
\[
G(R) = \{\, (x_{ij}) \in R^{d^2} \mid \forall f \in \cI,\, f(x_{ij}) = 0\,\}.
\]
Évidemment, $\det(X_{ij}) - 1 \in \cI$, et $G(R)$ peut donc être identifié à une partie de $\SL_d(R)$.
Cette identification sera implicite dans la suite. 
De l'hypothèse que $\Gamma$ est un sous-groupe de $\SL_d(\Z)$, on peut déduire que pour tout $R$, $G(R)$ est un sous-groupe de $\SL_d(R)$.
On peut alors voir $G$ comme un foncteur de la catégorie des anneaux commutatifs unifères dans la catégorie des groupes.
Comme $G$ est de plus représentable -- représenté par $\Z[X_{ij}] / \cI$ -- c'est un schéma en groupes affine sur $\Z$.
\wknote{C'est en fait un sous-schéma en groupes fermé de $\SL_{d,\Z}$}.

L'extension de base $\Z \to \Q$ permet d'obtenir à partir de $G$ un schéma en groupes affine sur $\Q$,
\[
G_\Q = \Hom((\Z[X_{ij}] / \cI)\otimes_\Z \Q, \,-\,), 
\]
\wknote{$G_\Q$ est aussi appelé la fibre générique de $G$.}
qu'on appelle la \emph{fibre générique de $G$}.
Le schéma $G_\Q$ est en fait une variété, qui coïncide avec la clôture de Zariski de $\Gamma$ dans $\SL_d$.
Suivant la terminologie de Borel \cite{Borel}, nous dirons que $G_\Q$ est un $\Q$-groupe.
La dimension de $G_\Q$ est simplement égale à la dimension de la variété $G_\Q$.

Si l'anneau $R$ est muni d'une topologie, nous munirons $G(R)$ de la topologie induite par la topologie produit sur l'espace de matrices $M_d(R)$. 
Pour un nombre premier $p$, $G(\Q_p)$ est alors un groupe fermé du groupe analytique $\SL_d(\Q_p)$, donc un sous-analytique par le théorème de Cartan~\cite[Part II, Chap. V, \S 9]{serre_lalg}.
Notons aussi que $G(\Z_p)$ est un sous-groupe ouvert dans $G(\Q_p)$.

\subsection{L'algèbre de Lie}

Au schéma en groupes $G$ est associée une algèbre de Lie sur $\Z$, notée $\g(\Z)$.
Nous rappelons les grandes lignes de cette construction, et renvoyons par example au livre \cite[Chapter 12]{Waterhouse} de Waterhouse pour plus de détails sur le sujet. 
Concrètement, notant $I_d = (\delta_{ij}) \in M_d(\Z)$,
\[
\g(\Z) = \Bigl\{\, (x_{ij}) \in M_d(\Z) \mid \forall f \in \cI,\  \sum_{i,j} \partial_{ij}f(I_d) x_{ij} = 0\,\Bigr\}.
\]
C'est un sous-module de $\ssl_d(\Z)$ sur $\Z$ et, muni du crochet usuel sur $\ssl_d(\Z)$, une sous-algèbre de Lie sur $\Z$.
Les équations qui définissent $\g(\Z)$ sont à coefficients dans $\Z$, on peut donc définir $\g(R)$ pour tout anneau $R$ commutatif et unifère:
\[
\g(R) = \Bigl\{\, (x_{ij}) \in M_d(R) \mid \forall f \in \cI,\  \sum_{i,j} \partial_{ij}f(I_d)x_{ij} = 0\,\Bigr\}.
\]
On identifie naturellement $\g(R)$ à une sous-algèbre de Lie de $\ssl_d(R)$.
\wknote{En général, $\g(R)$ n'est pas isomorphe à $\g(\Z) \otimes_\Z R$ il me semble.}
Si $R$ est sans torsion, alors $\g(R) = \g(\Z) \otimes_\Z R$. \wknote{Parce qu'un $\Z$-module sans torsion est plat.}
De même, si $p$ est un nombre premier suffisamment grand, alors $\g(\Z/p\Z) =  \g(\Z) \otimes_\Z \Z/p\Z$. \wknote{$p$ doit être plus grand que les coefficients d'un système d'equation engendrant celui qui définit $\g$.}

On vérifie aisément que cette notion d'algèbre de Lie coïncide avec d'autres définitions:
\begin{enumerate}
\item $\g(\C)$ est l'algèbre de Lie du groupe algébrique linéaire $G_\Q$. L'hypothèse que $G_\Q$ est simple implique que $\g(\C)$ est une algèbre de Lie simple sur $\C$.
\item Pour tout premier $p$, $\g(\Q_p)$ est égale à l'algèbre de Lie du groupe analytique $p$-adique $G(\Q_p)$.
%\item $\g(\R)$ est égale à égale à l'algèbre de Lie du groupe de Lie réel $G(\R)$.
\end{enumerate}
Nous aurons aussi besoin du lemme suivant, qui relie l'algèbre de Lie et le groupe des points sur l'anneau $\Z/q\Z$.

\begin{lemma}
\label{lm:12congru}
Pour tout nombre premier $p$ et tout entier $k \geq 1$. 
Le noyau de $G(\Z/p^{k+1}\Z) \to G(\Z/p^k\Z)$ est isomorphe au groupe additif $\g(\Z/p\Z)$.
\end{lemma}
\begin{proof}
\wknote{$G(\Z/p^{k+1}\Z) \to G(\Z/p^k\Z)$ est-t-il surjectif pour petit $p$ ?}
Un antécédent de $I_d \in M_d(\Z/p^k\Z)$ dans $M_d(\Z/p^{k+1}\Z)$ s'écrit de manière unique sous la forme $I_d + p^k x$, avec $x \in \gl_d(\Z/p\Z)$.
Pour tout $f \in \cI$,
\[f(I_d + p^k x) = p^k \sum_{i,j}\partial_{ij}f(I_d)x_{ij},\]
et par conséquent, $I_d + p^k x \in G(\Z/p^{k+1}\Z)$ si et seulement si $x \in \g(\Z/p\Z)$.
Cela donne l'isomorphisme souhaité.
\end{proof}

\subsection{Approximation forte.}

Le groupe $G(\ZZ)$ des points de $G$ sur l'anneau profini $\ZZ=\varprojlim\Z/q\Z$ s'identifie à un sous-groupe fermé de $\SL_d(\ZZ)$ pour la topologie profinie.
Si $\Omega$ désigne l'adhérence de $\Gamma$ dans $\SL_d(\ZZ)$, nous avons donc $\Omega \subset G(\ZZ)$. 
Le théorème d'approximation forte ci-dessous est une forme de réciproque à cette inclusion.
Il est dû à Matthews, Vaserstein et Weisfeiler~\cite{MVW} lorsque $G_\Q$ est simple, et à Nori~\cite{Nori} dans le cas général. 

\wknote{Je ne cite pas Theorem 5.2 de Nori parce qu'il me semble faux.}
\nsnote{Pourquoi?}
\begin{theorem}[Approximation forte]
\label{thm:Nori}
Soit $\Gamma$ un sous-groupe de $\SL_d(\Z)$, et $G$ le schéma en groupes associé.
Si $G_\Q$ est connexe, semi-simple et simplement connexe, alors
\begin{enumerate}
\item pour tout nombre premier $p$ assez grand, $\Omega/\Omega_p = G(\Z/p\Z)$;
\item le groupe $\Omega$ est ouvert dans $G(\ZZ)$ pour la topologie profinie. En particulier, $\Omega$ est d'indice fini dans $G(\ZZ)$.
\end{enumerate}
\end{theorem}

\comm{
\begin{remark}
Cela est faux si le groupe n'est pas simplement connexe, comme le montre l'exemple de l'image de $\Gamma=\SL_2(\Z)$ dans $\PGL_2$.
Cette image est Zariski dense.
Cependant, pour chaque $p$, l'image de $\Gamma$ dans $\PGL_2(\Z/p\Z)$ est d'indice au moins 2, car on ne peut pas obtenir une matrice de la forme $\diag(1,a)$, où $a$ n'est pas un carré modulo $p$.
\end{remark}
}

Notons deux corollaires immédiats du théorème ci-dessus.

\begin{corollary}
\label{chinois}
Il existe un entier $q_0 \in \N$ tel que  pour tout entier $q \in \N^*$ premier avec $q_0$, si $q = q_1 \dotsm q_n$ avec $q_1, \dotsc, q_n \in \N^*$ deux-à-deux premiers entre eux, alors nous avons le lemme chinois : 
\[(\pi_{q_1},\dotsc,\pi_{q_n}) \colon \Omega/\Omega_q \to \Omega/\Omega_{q_1} \times \dotsm \times \Omega/\Omega_{q_n}\]
est un isomorphisme de groupes.
\end{corollary}

\begin{corollary}
\label{chinois2}
Il existe une constante $C \geq 1$ tel que pour tout entier $q \in \N^*$, si $q = q_1 \dotsm q_n$ avec $q_1, \dotsc, q_n \in \N^*$ deux-à-deux premiers entre eux, 
\[ {[\Omega:\Omega_q]} \geq \frac{1}{C}  {[\Omega:\Omega_{q_1}]}	 \dotsm  {[\Omega:\Omega_{q_n}]}.\]
\end{corollary}

Outre ces deux corollaires, nous aurons encore besoin de quelques lemmes sur la structure du groupe $G(\Z/p\Z)$.
Une matrice $g$ est dite unipotent si $g - 1$ est nilpotent.
Si $p$ est un premier avec $p > d$ alors $g \in \SL_d(\Z/p\Z)$ est unipotent si et seulement si $g^p = 1$.

\begin{lemma}
\label{unipotent}
Soit $\Gamma$ un sous-groupe de $\SL_d(\Z)$ et $G$ le schéma en groupes associé.
Si $G_\Q$ est connexe, semi-simple et simplement connexe, alors, pour tout nombre premier $p$ suffisamment grand, $\Omega/\Omega_p$ est engendré par ses éléments unipotents.
\end{lemma}
\begin{proof}
\wknote{C'est à qui ce résultat ?}
Cela découle du théorème~\ref{thm:Nori} et d'un résultat de Steinberg \cite[Theorem~12.4]{steinberg_endomorphismsoflinearalgebraicgroups} selon lequel $G(\Z/p\Z)$ est engendré par ses éléments unipotents.
\comm{
En fait, l'approche de Nori pour démontrer le théorème~\ref{thm:Nori} commence justement par démontrer que $\Omega/\Omega_p$ contient des éléments unipotents (en utilisant un théorème de Jordan) et par semi-simplicité, les logarithmes des conjugués de ces éléments unipotents engendrent linéairement l'algèbre de Lie 
(cf. le survey de Breuillard~\cite[page 10]{breuillard_survey}).
}
\end{proof}

\begin{lemma}
\label{conjAd}
Soit $\Gamma$ un sous-groupe de $\SL_d(\Z)$ et $G$ le schéma en groupes associé.
On suppose que $G_\Q$ est connexe, semi-simple et simplement connexe.
Pour tout nombre premier $p$ suffisamment grand, le groupe quotient $\Omega_p/\Omega_{p^2}$ est isomorphe au groupe abélien $\g(\Z/p\Z)$.
L'action de $\Omega$ sur $\Omega_p/\Omega_{p^2}$ par conjugaison se factorise par $\Omega/\Omega_p$ et s'identifie à l'action adjointe de $G(\Z/p\Z)$ sur $\g(\Z/p\Z)$.
Le seul point fixe de cette action est $1 \in \Omega_p/\Omega_{p^2}$\comm{(autrement dit $0 \in \g(\Z/p\Z)$)}.
\end{lemma}

\begin{proof}
D'après le théorème~\ref{thm:Nori}, pour $p$ assez grand, les projections $\Omega \to G(\Z/p\Z)$ et $\Omega \to G(\Z/p^2\Z)$ sont surjectives.
Donc 
\[
\Omega_p/\Omega_{p^2} \simeq \ker\bigl(G(\Z/p^2\Z) \to G(\Z/p\Z)\bigr) \simeq \g(\Z/p\Z)
\]
par le lemme~\ref{lm:12congru}.
Cet isomorphisme est réalisé par l'application exponentielle.
Par le lemme~\ref{conjexpAd}, on voit donc que l'action de $\Omega$ sur $\Omega_p/\Omega_{p^2}$ par conjugaison  s'identifie à l'action adjointe de $G(\Z/p\Z)$ sur $\g(\Z/p\Z)$.
Pour $p$ grand, cette action se décompose en somme directe $\oplus_i\g_i(\Z/p\Z)$, où les $\g_i$ sont les algèbres de Lie des facteurs directs de $G$, et la dernière assertion découle donc du lemme suivant.
\end{proof}

\begin{lemma}
\label{lm:GFpsimple}
Soit $\Gamma$ un sous-groupe de $\SL_d(\Z)$ et $G$ le schéma en groupes associé.
On suppose que $G_\Q$ est simple. 
Alors pour $p$ assez grand, l'action adjointe de $G(\Z/p\Z)$ sur $\g(\Z/p\Z)$ est irréductible.
\end{lemma}
\begin{proof}
La propriété \enquote{l'action adjointe de $G(R)$ sur $\g(R)$ est irréductible} peut être exprimée comme une formule logique du premier ordre sur le langage des anneaux. La théorie des corps algébriquement clos admet l'élimination des quantificateurs. Le lemme découle donc de \cite[Corollary 9.2.2]{FriedJarden}.
\end{proof}

\section{L'application exponentielle}
\label{sec:exp}

Nous donnons ici la construction de l'application expontielle à valeurs dans $\GL_d(\Z/q\Z)$, lorsque $q\in\N^*$ est un entier arbitraire, et rappelons quelques-unes de ses propriétés élémentaires.
L'algèbre des matrices carrées de taille $d$ à coefficients dans un anneau $R$ est notée $\gl_d(R)$.

\subsection{L'application exponentielle $p$-adique}

Rappelons d'abord les propriétés de l'application exponentielle sur $\gl_d(\Z_p)$.
Notons $\alpha_p = 1 + \delta_2(p)$, i.e. $\alpha_2=2$ et $\alpha_p=1$ si $p\neq 2$.

\begin{lemma}
\label{lm:expZp}
Pour tout nombre premier $p$, la série entière \[\exp(x) = \sum_{n = 0}^{+\infty} \frac{x^n}{n!}\] est convergente sur $p^{\alpha_p} \gl_d(\Z_p)$ et définit une isométrie entre $p^{\alpha_p} \gl_d(\Z_p)$ et le sous-groupe de congruence $\ker(\GL_d(\Z_p) \to \GL_d(\Z/p^{\alpha_p} \Z))$.
Sa fonction réciproque, notée $\log$ s'écrit
\[\log(1 + x) = \sum_{n = 1}^{+\infty} \frac{(-1)^{(n-1)}}{n}x^n.\]
\end{lemma}
\begin{proof}
Pour $n \in \N$, nous avons
\begin{equation}
\label{eq:valfactoriel}
v_p(n!) = \Bigl\lfloor \frac{n}{p}\Bigr\rfloor + \Bigl\lfloor \frac{n}{p^2}\Bigr\rfloor + \dotsb \leq \frac{n}{p} + \frac{n}{p^2} + \dotsb \leq \frac{n}{p-1}.
\end{equation}
Pour tout $x, y \in p^{\alpha_p} \gl_d(\Z_p)$,
\[\exp(y) - \exp(x) = (y -x) \biggl( 1 + \frac{y  + x}{2} + \sum_{n = 3}^{+\infty} \frac{y^{n-1} + \dotsb + x^{n-1}}{n!}  \biggr).\]
On vérifie aisément $v_p(\frac{y  + x}{2}) > 0$ et pour $n \geq 3$,
\[v_p\bigl(\frac{y^{n-1} + \dotsb + x^{n-1}}{n!}\bigr) \geq  (n-1) \alpha_p - \frac{n}{p-1} > 0.\] 
\end{proof}

Le lemme suivant nous sera utile dans l'appendice~\ref{sec:qa}, où nous étudierons la propriété quasi-aléatoire du groupe $G(\ZZ)$, lorsque $G$ est un sous-schéma en groupes sur $\Z$ de $\SL_d$.

\begin{lemma}
\label{lm:BCHexp1}
Posons 
\[\beta_p = \begin{cases} 3 \quad \text{si $p=2$,} \\
2 \quad \text{si $p=3$,} \\
 1 \quad \text{sinon.}\end{cases}\]
Pour tout $x, y \in p^{\beta_p} \gl_d(\Z_p)$, on a 
\begin{equation}
\label{eq:BCHexp1}
\log(\exp(x)\exp(y)) \equiv x + y \mod p^{v_p(x) + v_p(y) - 2\delta_2(p)}.
\end{equation}
\end{lemma}
\begin{proof}
Formellement,
\begin{align}
\log(\exp(x)\exp(y)) &= \sum_{m=1}^{+\infty}\frac{(-1)^{m+1}}{m}(\exp(x)\exp(y) - 1)^m \notag\\
& = \sum_{m=1}^{+\infty}\frac{(-1)^{m+1}}{m} \Bigl(\sum_{k, \ell\geq 0}\frac{x^k y^\ell}{k! \ell!} - 1 \Bigr)^m \notag\\
& = \sum_{m=1}^{+\infty}\frac{(-1)^{m+1}}{m} \sum_{\substack{k_1 + \ell_1 \geq 1\\ \cdots \\ k_m + \ell_m \geq 1}}\frac{x^{k_1}y^{\ell_1} \dotsm x^{k_m}y^{\ell_m}}{k_1! \dotsm k_m! \ell_1! \dotsm \ell_m!}. \label{eq:BCHexp}.
\end{align}
En utilisant \eqref{eq:valfactoriel}, on a
\begin{align}
v_p\Bigl(\frac{(-1)^{m+1}}{m} \frac{x^{k_1}y^{\ell_1} \dotsm x^{k_m}y^{\ell_m}}{k_1! \dotsm k_m! \ell_1! \dotsm \ell_m!}\Bigr) &\geq k v_p(x) + \ell v_p(y) - \frac{k + \ell + m}{p-1} \notag\\
&\geq k v_p(x) + \ell v_p(y) - \frac{2}{p-1}(k + \ell), \label{eq:BCHterme4}
\end{align}
où $k = k_1 + \dotsb +k_m$ and $ \ell = \ell_1+  \dotsb + \ell_m$.
Ainsi, la série dans \eqref{eq:BCHexp} est convergente pour $x, y \in p^{\beta_p} \gl_d(\Z_p)$.
%Pour démontrer \eqref{eq:BCHexp1}, calculons les termes de degré au plus $3$,
%\[\log(\exp(x)\exp(y)) = x + y + \frac{1}{2}[x,y] + \frac{1}{12}[x,[x,y]] + \frac{1}{12}[y,[y,x]] + \dotsb \]
À l'aide de~\eqref{eq:BCHterme4}, et en faisant attention aux cas $p = 2$ et $p = 3$, on vérifie que tous les termes de degré homogène supérieur à $2$ qui apparaissent dans \eqref{eq:BCHexp} sont de valuation $p$-adique au moins $v_p(x)+v_p(y)-2\delta_2(p)$.
\end{proof}

D'après \cite[Part I, Chap. IV, Theorem 7.4]{serre_lalg}, chaque terme homogène dans la série dans \eqref{eq:BCHexp} est en fait un élément dans l'algèbre de Lie libre à deux indéterminées sur $\Q$, c'est-à-dire une combinaison $\Q$\dash{}linéaire de crochets itérés en $x$ et $y$.

\subsection{L'application exponentielle modulo $q$}
Étant donné $q = \prod p^{m_p}$, avec $m_2\neq 1$, posons
\[
\dot{r} = \prod_{p | q} p^{1 + \delta_2(p)}.
\]
En utilisant les isomorphismes
\[
\gl_d(\prod_{p|q} \Z_p) \simeq \prod_{p|q} \gl_d(\Z_p)
\quad\mbox{et}\quad
\GL_d(\prod_{p|q} \Z_p) \simeq \prod_{p|q} \GL_d(\Z_p),
\]
et en combinant les applications exponentielles sur chaque facteur, on obtient une bijection définie sur $\dot{r} \gl_d(\prod_{p|q} \Z_p)$ et à valeurs dans $\ker(\GL_d(\prod_{p|q} \Z_p) \to \GL_d(\Z/\dot{r}\Z))$.
Ainsi, pour $x \in \dot{r}\gl_d(\ZZ)$, $\exp(x)$ est bien défini dans $\GL_d(\prod_{p|q} \Z_p)$. 
Plus généralement, si $A \subset \dot{r}\gl_d(\ZZ)$, l'image $\exp(A)$ est bien définie, et l'on note
\[
N(\exp(A),q) = \card \pi_q(\exp(A))
\]
le cardinal de sa projection modulo $q$.

\begin{lemma}
\label{isom}
\begin{enumerate}
\item Pour tout $x \in \dot{r}\gl_d(\ZZ)$, pour tout facteur premier $p$ de $q$,
\[v_p(\exp(x) - 1) = v_p(x).\]
\item Pour tout $B \subset \dot{r}\gl_d(\ZZ)$, 
\[N(\exp(B),q) = N(B,q).\]
Inversement, si $A \subset \ker(\GL_d(\hat\Z) \to \GL_d(\Z)/\GL_d(\Z/\dot{r}\Z))$,
\[N(\log(A),q) = N(A,q).\]
\end{enumerate}
\end{lemma}

\begin{proof}
Ce sont des conséquences immédiates du lemme~\ref{lm:expZp}.
\end{proof}

Naturellement, si $\Gamma$ est un sous-groupe de $\SL_d(\Z)$, $G$ le schéma en groupes associé, et $\Omega$ l'adhérence de $\Gamma$ dans $\SL_d(\ZZ)$, on peut restreindre les applications exponentielles et logarithme à $\g(\ZZ)$ et $G(\ZZ)$.
En notant, pour $q\in\N^*$, $\Omega_q=\Omega\cap\ker\pi_q$, cela donne le lemme suivant.

\begin{lemma}
\label{lm:logg}
Si $g \in \Omega_{\dot{r}}$ alors $\log g\in \dot{r}\g(\prod_{p|q}\Z_p)$.
En particulier $\log g\in \dot{r}\g(\Z) \mod q\g(\Z)$.
\comm{Notons que comme $\g(\Z)$ est dense dans $\g(\ZZ)$, on a $\dot{r}\g(\Z)/q\g(\Z)\simeq\dot{r}\g(\ZZ)/q\g(\ZZ)\simeq\dot{r}\g(\prod_{p|q}\Z_p)/q\g(\prod_{p|q}\Z_p)$.}
\end{lemma}
\begin{proof}
D'après les propriétés de l'exponentielle modulo $q$ sur $\gl_d$, l'élément $\log g$ est bien défini dans $\gl_d(\dot{r}\prod_{p|q}\Z_p)$.
Par ailleurs, pour tout nombre premier $p$, l'algèbre de Lie du groupe analytique $G(\Z_p)$ est $\g(\Q_p)$, donc 
\[
\exp(p^{\alpha_p}\g(\Z_p)) \subset G(\Q_p) \cap \GL_d(\Z_p) \subset G(\Z_p),
\]
et
\[
\log g \in \prod_{p|q}p^{\alpha_p}\g(\Z_p) = \dot{r} \prod_{p|q}\g(\Z_p) \simeq \dot{r}\g(\prod_{p|q}\Z_p).
\]
%Comme l'anneau $\prod_{p|q}\Z_p$ est sans-torsion, $\g(\prod_{p|q}\Z_p) = \g(\Z) \otimes_\Z \prod_{p|q}\Z_p$, et cela montre le lemme.
\end{proof}

\comm{
\begin{remark}
Sans introduire de schéma en groupes, on peut définir $\g(\Z)=\g(\Q)\cap\gl_d(\Z)$, et de même pour $\g(\Z_p)$ et $\g(\prod_{p|q}\Z_p)$.
On définit aussi $G(\Z)=G(\Q)\cap\SL_d(\Z)$, etc.
Cela permet d'interpréter la première partie du lemme de façon élémentaire, et de voir que $\log g\in\dot{r}\g(\Z)\mod q$.
Ensuite, si on définit $\g(\Z/q\Z)=\g(\Z)/q\g(\Z)$, et $G(\Z/q\Z)=\pi_q(G(\Z))$ la seconde partie du lemme est claire aussi.
Ces définitions ne sont pas compatibles avec celles obtenues via les schémas en groupes, car on ne peut pas toujours relever une solution des équations de $G$ dans $\Z/p\Z$ en une solution dans $\Z$.
\end{remark}
\begin{remark}
Pour $p$ grand, grâce au lemme de Hensel -- la condition du jacobien non nul modulo $p$ sera vérifiée -- on montre que $G(\Z_p)=G(\Q_p)\cap\ssl_d(\Z_p)$ se surjecte sur $G(\Z/p\Z)$.
Si l'on pouvait démontrer que $G(\Z)$ est dense dans $G(\Z_p)$ (approximation faible), on en déduirait que $G(\Z)\to G(\Z/p\Z)$ est surjective.
\end{remark}
\begin{remark}
En tout état de cause, le théorème d'approximation forte montre que pour $p$ grand, $G(\Z)\to G(\Z/p\Z)$ est surjective, et idem pour $\g$. Ce qui permet en quelque sorte de parler de $G(\Z/p\Z)$ (et même $G(\Z/q\Z)$) sans avoir à introduire de schémas en groupes.
\end{remark}
}

\subsection{Exponentielle et représentation adjointe}

Si $R$ est un anneau unifère quelconque, l'action adjointe de $\GL_d(R)$ sur l'algèbre de Lie $\gl_d(R)$ est définie par
\[
\forall a \in \GL_d(R),\ \forall x \in \gl_d(R),\quad
\Ad(a)\cdot x = a x a^{-1}.
\] 
Nous noterons aussi, pour $a,g \in \GL_d(R)$,
\[
\iota(a)\cdot g = aga^{-1}.
\]
En particulier, pour $A \subset \GL_d(R)$,
\[
\iota(A) \cdot g = \{\, aga^{-1} \mid a\in A\,\}.
\]

\begin{lemma}
\label{conjexpAd}
Soit $q=\prod p^{m_p}$, avec $m_p\neq 1$, et $\dot{r}=\prod_{p|q} p^{1+\delta_2(p)}$.
\begin{enumerate}
\item Pour tout $a \in \SL_d(\Z)$ et $x \in \dot{r} \ssl_d(\Z)$
\[\iota(a) \cdot \exp(x) \equiv \exp(\Ad(a)\cdot x) \mod q.\]
\item Pour tout $a \in G(\Z)$ la restriction de $\Ad(a)$ à $\g(\Z)$ est un automorphisme de $\Z$-module. 
\end{enumerate}
\end{lemma}
\begin{proof}
Le premier point est une conséquence du développement en série entière de l'application exponentielle sur $\ssl_d$.
Le second découle de ce que $\Ad a$ préserve $\g(\Z)=\g(\Q)\cap\ssl_d(\Z)$ et admet pour inverse $\Ad a^{-1}$.
\end{proof}

\section{Propriété quasi-aléatoire}
\label{sec:qa}

Comme ci-dessus, $\Gamma$ désigne un sous-groupe de $\SL_d(\Z)$, $G$ est le schéma en groupes associé, et $\Omega$ l'adhérence de $\Gamma$ dans $\SL_d(\ZZ)$.
Dans ce dernier appendice, nous donnons une démonstration de la propriété quasi-aléatoire du groupe $G(\ZZ)$, déjà énoncée comme proposition~\ref{qr} dans le corps de l'article, et que nous rappelons ici pour plus de lisibilité.

\begin{proposition}[Propriété quasi-aléatoire]
\label{qra}
On suppose que $G_\Q$ est semi-simple, connexe, et simplement connexe.
Alors, il existe $\kappa>0$ tel que pour toute représentation irréductible $(\rho,V_\rho)$ de $\Omega$, il existe $q \in \N^*$ tel que
\[
\Omega_q \subset \ker \rho
\quad\mbox{et}\quad
\dim V_\rho \geq \kappa {[\Omega : \Omega_q]}^\kappa.
\]
\end{proposition}

%
%Soit $\pi$ une représentation complexe irréductible de $\Omega$, et $q\in\N^*$ minimal tel que
%\[
%\pi \equiv 1 \quad\mbox{sur}\quad \Omega_{q}.
%\]
%Écrivons
%\[
%q = \prod_p p^{m_p}.
%\]
%
%\begin{theorem}[Propriété quasi-aléatoire]
%\label{qr}
%Soit $\Gamma$ un sous-groupe de $\SL_d(\Z)$ et $\Omega$ l'adhérence de $\Gamma$ dans $\SL_d(\ZZ)$.
%On suppose que l'adhérence de Zariski $G$ de $\Gamma$ dans $\SL_d$ est simple, connexe, et simplement connexe.
%Il existe $\tau>0$ tel que pour toute représentation non triviale irréductible $\pi$ de $\Omega$, si $q\in\N^*$ est minimal tel que $\pi\equiv 1\quad\mbox{sur}\quad\Omega_q$, alors
%\[
%\dim V_\pi \geq \tau q^\tau.
%\]
%\end{theorem}

\comm{
\begin{remark}
Il est indispensable de supposer que $V_\rho$ est irréductible, comme le montre l'exemple de $G=\SL_2$, et d'une représentation somme d'une représentation de $\SL_2(\Z_p)$, et d'une représentation de $\SL_2(\Z_q)$, où $p$ et $q$ sont deux nombres premiers distincts.
La dimension d'une telle représentation est de l'ordre de $p^n+q^m$ (où $n$ et $m$ sont donnés par les noyaux des représentations), qui peut être considérablement plus petit que $p^nq^m$.
On aurait aussi pu considérer une somme directe de représentations standard de $\SL_2(\Z/p_i\Z)$, où les $p_i$ sont des premiers distincts.
La dimension est de l'ordre de $\sum_i p_i^2$, qui est négligeable devant $\prod_i p_i$.
\end{remark}
\begin{remark}
Il est aussi indispensable de supposer que $G_\Q$ est simplement connexe, comme on le voit avec l'exemple de $\PGL_2$.
En effet, $\PGL_2(\Z/p\Z)$ admet un sous-groupe d'indice 2, noté $\PSL_2(\Z/p\Z)$, égal au noyau de l'application $\det\colon\PGL_2(\Z/p\Z)\to(\Z/p\Z)^*/(\Z/p\Z)^{*2}$. Et donc $\PGL_2(\Z/p\Z)$ admet toujours une représentation non triviale de degré 2, égale à $L^2(\PGL_2(\Z/p\Z)/\PSL_2(\Z/p\Z)$.
Noter que $\PSL_2(\Z/p\Z)$ est justement le sous-groupe de $\PGL_2(\Z/p\Z)$ engendré par les unipotents.
\end{remark}
}

%De plus, pour tout $p$ assez grand, $U_p=G(\Z_p)$.
%Soit $\pi'$ une sous-représentation irréductible de la restriction de $\pi$ à $\Omega'$.
%Comme $\pi$ est irréductible, la réciprocité de Frobenius montre que $\pi$ est contenue dans l'induite de $\pi'$ à $\Omega$, et par conséquent:
%\begin{itemize}
%\item leurs degrés sont comparables;
%\begin{tiny}
%Plus précisément, $\dim V_{\pi'}\leq\dim V_\pi\leq [\Omega:\Omega']\dim V_{\pi'}$.
%\end{tiny}
%\item les entiers $q$ et $q'$ sont comparables.
%\begin{tiny}
%On peut trouver $C$ tel que $\Omega_C\subset\Omega'$. Alors, si $\pi'\equiv 1$ sur $\Omega_{q'}$, on a bien sûr $\pi'\equiv 1$ sur $\Omega_{Cq'}\subset\Omega'$, et comme $\pi$ est incluse dans l'induite par $\pi'$, nécessairement, $\pi\equiv 1$ sur $\Omega_{Cq'}$.
%Donc $q\leq Cq'$. L'inégalité $q\geq q'$ est évidente.
%\end{tiny}
%\end{itemize}

\subsection{Cas des groupes \texorpdfstring{$p$-adiques}{p-adiques}}

Pour un nombre premier $p$ et $m \in \N$, notons $H_{p,m}$ le sous-groupe de congruence de $G(\Z_p)$,
\[
H_{p,m} = \{g\in G(\Z_p)\ |\ g\equiv 1\mod p^m\}.%\ker\bigl(G(\Z_p) \to G(\Z/p^m\Z)\bigr).
\]
Nous dirons que le schéma en groupes $G$ est parfait si sa fibre générique $G_\Q$ l'est, ce qui revient à dire que son algèbre de Lie $\g_\Q$ coïncide avec son algèbre dérivée $[\g_\Q,\g_\Q]$.

\begin{proposition}
\label{qr,padic}
On suppose que $G$ est parfait.
Alors, il existe une famille d'entiers naturels $(o_p)_p$ telle que $o_p = 0$ sauf pour un nombre fini de nombres premiers $p$ et que l'assertion suivante soit vraie.

Pour tout $p$ premier, et tous entiers $k \geq 2$ et $m \geq 6$, si $(\rho,V)$ est une représentation unitaire de $H_{p,k}$ triviale sur $H_{p,m}$ et non triviale sur $H_{p,m-1}$, alors 
\[
\dim V \geq p^{\floor{m/2} - k - o_p}.
\]
\end{proposition}

\comm{
\begin{remark}
En particulier, avec le lemme~\ref{lm:qrsousgroupe}, cela implique la propriété quasi-aléatoire d'un un sous-groupe ouvert de $G(\Z_p)$ par rapport à la famille de ses sous-groupes de congruence.
\end{remark}
}

\begin{proof}%[Démonstration de la proposition~\ref{qr,padic}]
Dans cette démonstration, nous écrirons
\[
H_\ell = H_{p,\ell}\quad \mbox{pour}\ \ell\in\N,
\quad\mbox{et}\quad
H = H_k.
\]
Posons $\h = \log(H_k)$ et, pour tout $\ell \geq 2$, $\h_\ell = \log H_\ell$. 
D'après la discussion dans les appendices~\ref{sec:prelim} et \ref{sec:exp}, $\h_\ell = p^\ell \g(\Z_p)$ et $\g(\Z_p) = \g(\Z) \otimes_\Z \Z_p$.
En particulier $\h$ et $\h_\ell$ sont des sous-algèbres de Lie de $\g(\Z_p)$ sur $\Z_p$.

Maintenant posons $\ell = \lceil m/2 \rceil + \delta_2(p)$ de sorte que l'application exponentielle induit un isomorphisme de groupe entre $\h_\ell/\h_m$ et $H_\ell/H_m$, par lemme~\ref{lm:BCHexp1}.
On peut supposer $\ell \geq k$, sans quoi il n'y a rien à démontrer.
La représentation $\rho$ se décompose en une somme directe suivant les caractères de $\h_{\ell}$,
\[V = \bigoplus_{\psi \in \Hom(\h_\ell,S^1)} V_\psi,\]
où $\psi$ parcourt l'ensemble des caractères unitaires de $\h_\ell$ se factorisant par $\h_\ell/\h_m$ et pour un tel $\psi$,
\[
V_\psi = \{\, v \in V \mid \forall x \in \h_\ell,\ \rho(\exp(x))v = \psi(x)v\,\}.
\]

Comme $H_\ell$ est distingué dans $H$, l'action de $H$ sur $V$ permute les sous-espaces caractéristiques correspondant à l'action co-adjointe. 
Plus précisément, %rappelons
%\[\forall g \in H,\ \forall x \in \h_\ell,\quad \Ad(g)x = \log(g \exp(x) g^{-1}),\]
%et posons
notant
\[
\forall g \in H,\ \forall \psi \in \Hom(\h_\ell,S^1),\quad \Ad^*(g)\psi = \psi\circ\Ad(g^{-1}),
\]
nous avons 
\[
\forall g \in H,\ \forall \psi \in \Hom(\h_\ell,S^1),\quad \rho(g) V_\psi = V_{\Ad^*(g)\psi}.
\]
Il s'ensuit que pour des caractères $\psi \in \Hom(\h_\ell,S^1)$ avec $V_\psi \neq \{0\}$, on a 
\[\dim V \geq \abs{\Ad^*(H)\psi} = {[H : \Stab_{H}(\psi)]}\]
où $\Stab_{H}(\psi) = \{\,g \in H \mid \Ad^*(g)\psi = \psi\,\}$ désigne le stabilisateur de $\psi$ sous l'action co-adjointe.
La fin de la démonstration consiste à minorer l'indice $[H:\Stab_H(\psi)]$, ce pour quoi nous aurons besoin de deux lemmes.
Le premier est une variante d'un énoncé de Howe \cite[Lemma 1.1]{Howe}.

\begin{lemma}
Posons $\stab_{\h}(\psi) = \log \Stab_{H}(\psi)$, alors 
\[\stab_{\h}(\psi) = \{\, x \in \h \mid \forall y \in \h_\ell,\ \psi([x,y]) = 1\}.\]
\end{lemma}
\begin{proof}
Pour $x\in\h$, la somme $\tau(x) = \sum_{n = 0}^{+ \infty} \frac{\ad(x)^n}{(n+1)!}$ définit une isométrie de $\h\to\h$ vérifiant
\[
\ad(x) \circ \tau(x) = e^{\ad(x)} - \Id = \Ad(\exp(x)) - \Id.
\]
\comm{
En effet, pour $x \in p^\alpha \gl_d(\Z_p)$ et $y \in p^\alpha \gl_d(\Z_p)$,
\[\tau(x)y - y = \sum_{n = 1}^\infty \frac{\ad(x)^n y}{(n+1)!}\]
a une valuation
\[v_p(\tau(x)y - y) \geq v_p(y) + \min_{n\geq 1}\bigl(n v_p(x) - v_p((n+1)!)\bigr) > v_p(y).\]
Si $\h \subset p^\alpha \gl_d(\Z_p)$ est une sous\dash{}$\Z_p$\dash{}algèbre de Lie, alors pour tout $x \in \h$, $\tau(x)\h \subset \h$. Or $\tau(x)$ est isométrique donc $\tau(x)$ préverve les $\h_m$, donc induit des bijections de $\h/\h_m$. On en déduit que la restriction de $\tau(x)$ à $\h$ a une image dense dans $\h$, donc est surjectif, donc un automorphisme de $\h$.  
}
Comme $\tau(x)$ est une isométrie, $\tau(x)\h_\ell = \h_\ell$ et donc, pour tout $x \in \h$,
\begin{align*}
\exp(x) \in \Stab_H(\psi) 
& \Longleftrightarrow
\forall y \in \h_\ell,\quad \psi(\Ad(\exp(x))y - y) = 1\\
& \Longleftrightarrow
\forall y \in \h_\ell,\quad \psi([x,\tau(x)y]) = 1\\
& \Longleftrightarrow
\forall y \in \h_\ell,\quad \psi([x,y]) = 1.
\end{align*}
\end{proof}

Le second lemme est tiré de la démonstration de \cite[Lemma~32]{sg1}.
\begin{lemma}
Pour tout caractère $\psi \in \Hom(\h_\ell,S^1)$, si $\stab_\h(\psi^p) = \stab_\h(\psi)$, alors $\h = \stab_\h(\psi)$.
\end{lemma}
\begin{proof}
Notons que $\stab_\h(\psi)$ est ouvert dans $\h$ et que par conséquent, le quotient $\h / \stab_\h(\psi)$ est fini.
Or, $\h$ est pro\dash{}$p$, donc le quotient $\h / \stab_\h(\psi)$ est un $p$\dash{}groupe.
Pour montrer qu'il est trivial il suffit de montrer qu'il n'a pas de $p$\dash{}torsion.

Supposons que $\stab_\h(\psi^p) = \stab_\h(\psi)$, et soit $x \in \h$ avec $px \in \stab_\h(\psi)$.
Alors, pour tout $y \in \h_\ell$, $\psi^p([x,y]) = \psi([px,y]) = 1$ d'après le lemme précédent. Par le lemme précédent de nouveau, $x \in \stab_\h(\psi^p)$ et donc $x \in \stab_\h(\psi)$ par hypothèse.
Cela montre que $\h / \stab_\h(\psi)$ n'a pas de $p$\dash{}torsion.
\end{proof}

En itérant ce lemme, on obtient, pour tout caractère unitaire $\psi$ de $\h_\ell$, un entier $n$ tel que
\[
\stab_\h(\psi) \subsetneq \stab_\h(\psi^p) \subsetneq \dots \subsetneq \stab_\h(\psi^{p^n}) = \h.
\]
%est une suite strictement croissante de sous-groupes ouverts de $\h$.
Comme chaque terme de cette suite est d'indice au moins $p$ dans le terme suivant,
\[
{[H: \Stab_H(\psi)]}= {[\h : \stab_\h(\psi)]} \geq p^n.
\]
Reste à minorer l'entier $n$.
%L'égalité $\stab_\h(\psi^{p^n}) = \h$ implique que 
%\[
%p^n [\h,\h_\ell] \subset \ker \psi.
%\]
Par hypothèse, $[\g(\Q),\g(\Q)] = \g(\Q)$,
%Or, $\g(\Q) = \g(\Z) \otimes_\Z \Q$.
et il existe donc un entier $D = \prod p^{o_p}$ tel que $D \g(\Z) \subset  [\g(\Z),\g(\Z)]$.
Avec l'égalité $\g(\Z_p) = \g(\Z) \otimes_\Z \Z_p$, cela donne
\[p^{o_p} \g(\Z_p) \subset  [\g(\Z_p),\g(\Z_p)].\]
Par ailleurs, $p^k \g(\Z_p) = \h$ et $p^\ell \g(\Z_p) = \h_\ell$, et donc
\[
\h_{k +\ell + o_p} = p^{k+\ell+o_p} \g(\Z_p) \subset [\h,\h_\ell].
\]
Mais l'égalité $\stab_\h(\psi^{p^n})=\h$ implique $p^n [\h,\h_\ell] \subset \ker \psi$, puis
\[
\h_{n + k + \ell + o_p} \subset p^n [\h,\h_\ell] \subset \ker \psi.
\]
Par hypothèse, il existe $\psi \in \Hom(\h_\ell,S^1)$ tel que $\dim V_\psi > 0$ et $\h_{m-1} \not\subset \ker \psi$, d'où
\[
n + k + \ell + o_p \geq m
\]
et enfin
\[\dim V \geq {[H: \Stab_H(\psi)]} \geq p^n \geq p^{\floor{m/2} - k - o_p - \delta_2(p)}.\]
\end{proof}

\comm{
%\begin{tiny}
\begin{remark}
Dans le cas $p\neq 2$,
%quitte à réduire un peu $U_p$,
on aurait pu appliquer la théorie de Kirillov pour les groupes $p$-adiques compacts: d'après \cite[Theorem~1.1]{Howe} la représentation $(\rho,V)$ provient d'un caractère $\chi$ de $\h$ trivial sur $\h_m$ et
\[
\dim V \geq \abs{\Ad^*(H)\chi}^{1/2} = {[H:\Stab_{H}(\chi)]}^{1/2},
\]
où $\Ad^*(H)$ désigne ici l'action co-adjointe $H$ de sur $\Hom(\h,S^1)$. On pourra alors conclure comme ci-dessus.
Mais le théorème~1.1 de Howe \cite{Howe} n'est pas valable pour $p=2$, contrairement à l'argument présenté ci-dessus.
\end{remark}
}

\comm{
\begin{remark}
Argument alternatif pour minorer l'indice du stabilisateur.
Le dual de $\Z_p=\varprojlim\Z/p^n\Z$ est $\varinjlim\Z/p^n\Z$, et on peut donc écrire en identifiant $\h$ à $\Z_p^d$,
\[
\chi(x) = e^{2i\pi\frac{\bracket{a,x}}{p^{m}}},
\]
avec $a\in(\Z/p^n\Z)^{d}$ tel que $p \nmid a$, et $\bracket{a,x}=\sum_i a_i\pi_{p^n}(x)_i$.
En fait, si l'on note, pour $a,x\in\Z_p^d$, $\bracket{a,x}=\sum_i\pi_{p^n}(a_ix_i)$ la formule ci-dessus ne dépend pas du choix d'un relevé arbitraire de $n$ dans $\Z_p^d$, toujours noté $a$.
Avec ces notations, $g\cdot\chi = \chi$ si et seulement si, pour la distance $p$-adique
\[
d_p(ga,a)\leq p^{-m}
\]
ce qui implique, avec la bonne version de l'inégalité de \L ojasiewicz,
\[
d_p(g,\Stab_H(a))\leq p^{-\tau m},\quad\mbox{où}\quad \Stab_H(a) =\{ g\in H\ |\ g\cdot a = a\}.
\]
\wknote{J'ai un doute. Il me semble que le stabisateur est un sous-groupe ouvert.}
\nsnote{Attention! Je considère ici le stabilisateur du vecteur $a$ qui appartient à une $\Q_p$-représentation, et non à une $\C$-représentation. Le stabilisateur n'est pas ouvert.}
Comme l'action co-adjointe de $H$ sur $\Hom(\h,\Q_p) \simeq \Q_p^k$ est irréductible, on doit avoir $\dim_{\Q_p} \Stab_H(a) < \dim_{\Q_p} H$.
Pour conclure, il suffit de remarquer que l'exponentielle envoie $Lie(\Stab_H(a))(\Q_p)$ sur $\Stab_H(a)(\Q_p)$ au voisinage de l'identité:
cela découle du développement en série entière du logarithme au voisinage de $1$, cf. \cite[démonstration du lemme~2.7]{hesaxce_sumproductinrepresentations}.
\end{remark}
}

\subsection{Cas des groupes linéaires sur $\Z/p\Z$}

Comme nous n'avons pas pu en trouver une démonstration en un seul tenant dans la littérature, nous montrons ici une borne inférieure sur le degré d'une représentation irréductible non triviale de $G(\Z/p\Z)$.
Ce résultat remonte à Frobenius \cite{frobenius} lorsque $G=\SL_2$, apparaît dans Landazuri et Seitz \cite{ls} lorsque $G$ est un groupe de Chevalley, et dans le cas général dans l'article d'Emmanuel Breuillard \cite[Proposition~6.1]{breuillard_survey} sur le sujet.

\begin{theorem}
\label{qr,p}
Soit $G$ un sous-schéma en groupes fermé de $\SL_{d,\Z}$ dont la fibre générique $G_\Q$ est un groupe algébrique connexe semi-simple simplement connexe.
Pour tout nombre premier $p$ suffisamment grand, le degré de toute représentation linéaire irréductible non triviale de $G(\Z/p\Z)$ est minoré par $\frac{p-1}{2}$.
\end{theorem}

\begin{proof}
Supposons dans un premier temps que $G_\Q$ soit absolument simple.
L'idée de la démonstration est de reprendre l'argument élémentaire valable pour $\SL_2(\Z/p\Z)$ en utilisant un $\ssl_2$-triplet bien choisi dans $\g(\Z/p\Z)$.
Soit $(\rho,V)$ une représentation linéaire non triviale de $G(\Z/p\Z)$.

Comme $G_\Q$ est simplement connexe, pour tout $p$ suffisamment grand, le groupe $G(\Z/p\Z)$ est engendré par ses éléments unipotents \cite[Theorem~12.4]{steinberg_endomorphismsoflinearalgebraicgroups}.
Soit $u$ un élément unipotent quelconque de $G(\Z/p\Z)$. 
Le théorème de Jacobson-Morozov pour les algèbres de Lie semi-simples est habituellement cité en caractéristique nulle, et c'est le cas dans Bourbaki \cite[Chapitre VIII, \S11, Proposition~2]{bourbaki_gal_7_8}.
\comm{Voir aussi les notes d'Yves Benoist \cite[Théorème~2.15, page 20]{benoist_reseauxdesgroupesdelie}.}
Cependant, on peut vérifier que la démonstration reste valable dans notre cadre dès que $p$ est suffisamment grand.
\comm{Essentiellement, tout ce dont on a besoin, c'est que la forme de Killing soit non dégénérée sur $\Z/p\Z$,
-- c'est bien le cas, car c'est la réduction modulo $p$ de la forme de Killing sur $\Q$, qui est non-dégénérée --
 et que $p>r+2$, où $r$ est tel que $(\ad x)^r=0$.
Comme l'ordre de nilpotence de $\ad x$ est borné par $(\dim G_\Q)^2$, indépendamment de $p$, on peut aussi s'assurer de cette deuxième condition.
}
Par conséquent, l'élément nilpotent $x=\log u$ dans $\g(\Z/p\Z)$ fait partie d'un certain $\ssl_2$-triplet $(x,y,h)$.
Les éléments $u=\exp x$ et $\exp y$ sont dans $G(\Z/p\Z)$.
\comm{-- si $G$ est un sous-groupe algébrique de $\SL_d$ défini sur $k$, et si $x\in\g(k)$ est nilpotent, alors $\exp x\in G(k)$ (exercice) --}
Notons $H$ le sous-groupe qu'ils engendrent.
L'action adjointe de $H$ sur l'algèbre de Lie $\bracket{x,y,h}\simeq\ssl_2(\Z/p\Z)$ est isomorphe à $\Aut\ssl_2(\Z/p\Z)\simeq \PSL_2(\Z/p\Z)$.
De plus, le noyau de $H \to \PSL_2(\Z/p\Z)$ est contenu dans le centre de $H$.
%Tout sous-groupe qui se projette surjectivement sur $\PSL_2(\Z/p\Z)$ est dense dans $\PSL_2(\Z_p)$, exercice.
%On a identifié ci-dessus $\PSL_2$ avec le groupe des automorphismes de $\bracket{X,H,Y}\simeq\ssl_2$.

Soit alors $t$ un générateur de $(\Z/p\Z)^\times$ et $a\in G(\Z/p\Z)$ un élément de $H$ tel que les images de $a$ et $u$ dans $\PSL_2(\Z/p\Z)$ s'identifient respectivement à $\begin{pmatrix}t & 0\\ 0 & t^{-1}\end{pmatrix}$ et $\begin{pmatrix}1 & 1\\0 & 1\end{pmatrix}$.
Il existe un élément $z$ dans le centre de $H$ tel que $aua^{-1} = zu^{t^2}$. Comme $u$ est d'ordre $p$, on a $z^p = 1$. 
Soit $k$ l'inverse de $1 - t^2$ dans $(\Z/p\Z)^\times$.
Quitte à remplacer $u$ par $z^ku$, on peut supposer 
\begin{equation}
\label{eq:aetu}
aua^{-1} = u^{t^2}.
\end{equation}

D'après le lemme~\ref{lm:GFpsimple}, les logarithmes des conjugués de $u$ engendrent linéairement $\g(\Z/p\Z)$, pour $p$ assez grand.
D'après \cite[Theorem B]{Nori}, $u$ et ses conjugués engendrent $G(\Z/p\Z)$.
\comm{Par Nori, le groupe engendré par $u$ et ses conjugués s'écrit $H(F_p)^+$, avec $H$ algébrique, et comme son cardinal est $\gg p^{dim G}$ (cf. argument dans le survey d'Emmanuel \cite[Proof of Theorem 2.2]{breuillard_survey}), on doit avoir d'après l'inégalité de Lang-Weil, $\dim H=\dim G$, et par connexité de $G$, $H=G$.
Mais $G$ est simplement connexe, et donc $H(F_p)^+=G(F_p)^+=G(F_p)$, d'après Steinberg \cite[Theorem~12.4]{steinberg_endomorphismsoflinearalgebraicgroups}.
}
%Ici on utilise de nouveau le fait que le groupe $G(\Z/p\Z)$ est engendré par ses éléments unipotents.
%Ensuite, d'après un résultat de Tits \cite[Theorem 7.1]{PlatonovRapinchuk}, tout sous-groupe propre distingué de $G(\Z/p\Z)$ est central.
%En plus, comme on va voir tout de suite, un élément unipotent n'est dans le centre.
%On en déduit qu'un élément unipotent n'est contenu dans aucun sous-groupe propre distingué de $G(\Z/p\Z)$.
Si l'on décompose $V$ suivant les caractères du groupe $U\simeq\Z/p\Z$ engendré par $u$,
\[
V = \bigoplus_{\chi \in \Hom(U,S^1)} V_\chi,
\]
où
\[V_\chi = \{\, v\in V \mid \forall g \in U,\, \rho(g)v = \chi(g)v \,\},\]
on voit apparaître un caractère non trivial.
Or, $a$ normalise le sous-groupe $U$.
Il agit donc sur le groupe de ses caractères.
À l'aide de \eqref{eq:aetu}, on voit aisément que l'orbite d'un charatère non-trivial $\chi$ sous l'action du sous-groupe engendré par $a$ est de cardinal $\frac{p-1}{2}$.
Comme tous les éléments de cette orbite doivent apparaître dans la décomposition de $V$, on trouve bien
\[
\dim V \geq \frac{p-1}{2}.
\]
Ceci termine la démonstration dans le cas où $G_\Q$ est absolument simple ou, plus précisément, le cas où l'action adjointe de $G(\Z/p\Z)$ sur $\g(\Z/p\Z)$ est irréductible.

Dans le cas général, on se ramène au cas où l'action de $G(\Z/p\Z)$ est irréductible grâce à une décomposition de la réduction modulo $p$ de $G$ en facteurs simples sur $\Z/p\Z$.
Comme nous n'utilisons pas le théorème dans cette généralité, les détails de la démonstration sont laissés au lecteur.
\comm{
Pour le cas général, l'idée est de décomposer la réduction modulo $p$ de $G$ en facteur simples sur $\Z/p\Z$.
On peut donc écrire (pour $p$ assez grand)
\[G(\Z/p\Z) = \prod_{i} G_{p,i}(\Z/p\Z)\]
où chaque $G_{p,i}$ est un groupe algébrique défini et simple sur $\Z/p\Z$.
Le seul problème est que cette suite de groupe $G_{p,i}$ dépend de $p$.
Il suffit alors de remarquer qu'il y a une colletion fini de schémas en groupe $G_j$ sur $\Z$ dont la fibre générique est semisimple telle que pour tout $p$ assez grand, chaque $G_{p,i}$ est la réduction modulo $p$ d'un certain $G_j$. 
}
\end{proof}

\begin{corollary}
\label{cr:qr,p}
Soit $G$ un sous-schéma en groupes fermé de $\SL_{d,\Z}$ dont la fibre générique $G_\Q$ est un groupe algébrique connexe semi-simple et simplement connexe.
Alors pour tout nombre premier $p$ suffisamment grand, l'indice d'un sous-groupe propre de $G(\Z/p\Z)$ est minoré par $\frac{p-1}{2}$.
\end{corollary}
\begin{proof}
Si $H$ est un sous-groupe propre de $G(\Z/p\Z)$ alors la représentation quasi-régulière $\ell^2(G(\Z/p\Z)/H)$ de $G(\Z/p\Z)$ est non triviale et son degré est égal à l'indice de $H$ dans $G(\Z/p\Z)$.
\end{proof}

Le théorème~\ref{qr,p} permet aussi de minorer le degré d'une représentation non triviale du groupe profini $G(\Z_p)$, lorsque $G$ est simple.

\begin{proposition}
\label{qr,pm}
Soit $G$ un sous-schéma en groupes fermé de $\SL_{d,\Z}$ dont la fibre générique $G_\Q$ est un groupe algébrique connexe semi-simple et simplement connexe.
Pour tout nombre premier $p$ suffisamment grand, le degré de toute représentation unitaire non triviale de $G(\Z_p)$ est au moins $\frac{p-1}{2}$.
\end{proposition}
\wknote{L'hypothèse simple est utilisée pour démontrer que l'action coadjointe de $G(\Z/p\Z)$ sur $\g(\Z/p\Z)^*$ n'a qu'un seul point fixe.}
\nsnote{Cela est encore vrai si $G_\Q$ est seulement semi-simple.}

\begin{proof}
D'après le lemme de Hensel, si $p$ est suffisamment grand, la projection $G(\Z_p) \to G(\Z/p^m\Z)$ est surjective pour tout entier $m \geq 1$.
\comm{Alternativement, cela se montre aussi à l'aide du théorème~\ref{thm:Nori}, d'après lequel, si $p$ est suffisamment grand, la projection $G(\Z_p) \to G(\Z/p\Z)$ est surjective.
En fait, cela découle simplement du lemme de Hensel, il n'y a pas vraiment besoin du résultat de Nori ici. Le résultat de Nori est beaucoup plus fort, et implique que $G(\Z)\to G(\Z/p\Z)$ est sujective, et on peut même remplacer $G(\Z)$ par un sous-groupe dense au sens de Zariski.
Cela implique d'ailleurs que pour tout $m \geq 1$, la projection $G(\Z_p) \to G(\Z/p^m\Z)$ est surjective.
Justification: on procède par récurrence sur $m$. Par exemple, pour $m=2$, comme $G(\Z_p)\to G(\Z/p\Z)$ n'admet pas de section, il existe un élément $x\in G(\Z_p)$ tel que $x\equiv 1\mod p$ et $x\not\equiv 1\mod p^2$. On fait ensuite agir $G(\Z/p\Z)$ par conjugaison sur $x$; comme cette action s'identifie à l'action adjointe sur $\g(\Z/p\Z)$, et que cette action est irréductible pour $p$ grand, cela permet d'obtenir par sommes et produits tout le sous-groupe $\{g\in G(\Z/p^2\Z)\ |\ g\equiv 1\mod p\}$.
%Alternativement, on peut aussi appliquer l'argument basé sur le lemme de Hensel, qui fonctionne directement pour $G(\Z_p)\to G(\Z/p^m\Z)$.
}
Supposons en outre $p$ assez grand pour que les conclusions du lemme~\ref{lm:GFpsimple}, du théorème~\ref{qr,p} et du corollaire~\ref{cr:qr,p} soient vérifiées.

Soit $(\rho,V)$ une représentation unitaire non triviale de $G(\Z_p)$ et $m\in\N$ minimal tel que $\rho$ se factorise par $G(\Z/p^m\Z)$.
Alors, $(\rho,V)$ s'identifie à une représentation linéaire de $G(\Z/p^m\Z)$.

Si $m = 1$, alors $\dim V \geq \frac{p - 1}{2}$ par le théorème~\ref{qr,p}.
Sinon, notons $H$ le noyau de la projection $G(\Z/p^m\Z) \to G(\Z/p^{m-1}\Z)$.
Par le lemme~\ref{lm:12congru}, $H$ est abélien, isomorphe à $\g(\Z/p\Z)$.
On peut décomposer $V$ en sous-espaces caractéristiques
\[
V = \bigoplus_{\chi \in \Hom(H,S^1)} V_\chi,
\]
où
\[V_\chi = \{\, v\in V \mid \forall g \in H,\, \rho(g)v = \chi(g)v \,\}.\]
Par minimalité de $m$, il existe un caractère $\chi$ non trivial tel que $V_\chi \neq \{0\}$.
Comme dans la démonstration de la proposition~\ref{qr,padic}, l'action co-adjointe de $G(\Z/p^m\Z)$ sur les caractères de $H$ permet de montrer que
\[
\dim V \geq {[G(\Z/p^m\Z): \Stab_{G(\Z/p^m\Z)}(\chi)]}.
\]
Or, cette action de $G(\Z/p^m\Z)$ se factorise par $G(\Z/p\Z)$, et sous l'isomorphisme $H \simeq \g(\Z/p\Z)$, s'identifie à l'action co-adjointe de $G(\Z/p\Z)$ sur le dual de $\g(\Z/p\Z)$.
Par le lemme~\ref{lm:GFpsimple}, la dernière action n'a pas de point fixe sauf l'élément $0$ dans le dual de $\g(\Z/p\Z)$.
Donc ${[G(\Z/p^m\Z): \Stab_{G(\Z/p^m\Z)}(\chi)]}$ est égal à l'indice d'un sous-groupe propre de $G(\Z/p\Z)$, qui est minoré par $\frac{p-1}{2}$, d'après le corollaire~\ref{cr:qr,p}.
\end{proof}

\subsection{La propriété quasi-aléatoire de $\Omega$.}

Nous dirons qu'un groupe pro-fini $\Omega$ est \emph{quasi-aléatoire} par rapport à une famille de sous-groupes distingués $(\Omega_q)_{q \in \N^*}$ si pour toute représentation irréductible unitaire $(\rho,V_\rho)$ de $\Omega$, il existe $q \in \N^*$ tel que
\[
\Omega_q \subset \ker \rho
\quad\mbox{et}\quad
\dim V_\rho \geq \kappa {[\Omega : \Omega_q]}^\kappa.
\]
Pour conclure la démonstration de la proposition~\ref{qra}, nous utiliserons le lemme suivant.

\begin{lemma}
\label{lm:qrsousgroupe}
Soit $\Omega$ un groupe profini et $(\Omega_q)_{q \in \N^*}$ une famille de sous-groupes ouverts distingués vérifiant $\Omega_{q} \cap \Omega_{q'} = \Omega_{\pgcd(q,q')}$ pour tous $q,q' \in \N^*$, et telle que $\bigcap_{q\in\N^*}\Omega_q=\{1\}$.
Soit $\Omega'$ un sous-groupe fermé d'indice fini de $\Omega$.
Si $\Omega'$ est quasi-aléatoire par rapport à la famille $\Omega'_q = \Omega_q \cap \Omega',\, q\in \N^*$, alors $\Omega$ est quasi-aléatoire par rapport à la famille $(\Omega_q)_{q\in \N^*}$.
\end{lemma}

\begin{proof}
Comme $\Omega'$ est fermé et d'indice fini, il est ouvert dans $\Omega$, et contient donc $\Omega_{q'}$ pour certain $q' \in \N^*$.

Soit $(\rho,V_\rho)$ une représentation irréductible unitaire de $\Omega$ et $(\rho',V_{\rho'}) \in \hat\Omega'$ une sous-représentation irréductible de la restriction de $\rho$ à $\Omega'$.
D'après la propriété quasi-aléatoire de $\Omega'$, il existe $\kappa> 0$ indépendant de $\rho$ tel qu'il existe $q \in \N^*$ vérifiant $\Omega'_q \subset \ker \rho'$ et $\dim V_{\rho'} \geq \kappa {[\Omega' : \Omega'_q]}^\kappa$. 
En posant $s = \pgcd(q,q')$, on a
\[\Omega_{s} = \Omega_q \cap \Omega_{q'} \subset \Omega_q \cap \Omega' = \Omega'_q \subset \ker \rho.\] 
Par conséquent,
\begin{align*}
{[\Omega: \Omega_s]} & = {[\Omega: \Omega'_q]} {[\Omega'_q : \Omega_s]}\\
& \leq {[\Omega: \Omega'_q]} {[\Omega'\cap\Omega_q : \Omega_{q'}\cap\Omega_q]}
\leq {[\Omega: \Omega_{q'}]}  {[\Omega': \Omega'_q]}
\end{align*}
et enfin,
\[\dim V_\rho \geq \dim V_{\rho'} \geq \kappa {[\Omega' : \Omega'_q]}^\kappa \geq \kappa\frac{{[\Omega: \Omega_s]}^\kappa}{{[\Omega: \Omega_{q'}]}^\kappa} .\]
Cela montre que $\Omega$ est quasi-aléatoire par rapport à la famille $(\Omega_q)_{q \in \N^*}$.
\end{proof}

\begin{proof}[Démonstration de la proposition~\ref{qr}]
D'après le théorème~\ref{thm:Nori} il existe un entier $s = \prod p^{k_p}$ tel que $\Omega$ contienne le sous-groupe de congruence $\ker \pi_s\subset G(\ZZ)$.
Par le lemme~\ref{lm:qrsousgroupe}, on peut supposer $\Omega$ égal à ce sous-groupe de congruence.
Alors, $\Omega$ s'écrit comme un produit direct
\[\Omega = \prod H_{p,k_p},\]
où $H_{p,k_p} = \ker\pi_p\subset G(\Z_p)$.

Soit $(\rho,V) \in \hat{\Omega}$ une représentation unitaire irréductible, et $q$ le multiple minimal de $s$ tel que $\Omega_q \subset \ker \rho$.
Écrivons $q = \prod p^{m_p}$.
Comme $\Omega$ est un produit direct, $(\rho,V)$ s'écrit comme un produit tensoriel
\[
(\rho, V) = \bigotimes_{p|q} (\rho_p,V_p),
\]
où pour tout facteur premier $p$ de $q$, $(\rho_p,V_p)$ est une représentation unitaire irréductible de $H_{p,k_p}$. 
De plus, pour tout $p$, $m_p$ est l'entier minimal tel que $\rho_p$ soit trivial sur $H_{p,m_p}$.

Soit $M\geq 0$ tel que pour tout $p \geq M$, la conclusion de la proposition~\ref{qr,pm} soit valable, et qu'en outre $k_p = 0$ et $\frac{p-1}{2} \geq p^{1/2}$.
On partitionne les nombres premiers en trois parties en posant
\begin{align*}
I_1 &= \{\,p \mid p \leq M \text{ et } m_p \leq 5\,\},\\
I_2 &= \{\,p \mid p \geq M \text{ et } m_p \leq 5\,\},\\
I_3 &= \{\,p \mid m_p \geq 6\,\}.
\end{align*}
Évidemment,
\[\prod_{p \in I_1} \dim V_p \geq 1.\]
De plus, si $p \in I_2$, alors $\rho_p$ est une représentation unitaire de $G(\Z_p)$, et d'après la proposition~\ref{qr,pm},
\[
\prod_{p \in I_2} \dim V_p \geq \prod_{p \in I_2} \frac{p-1}{2} \geq \prod_{p \in I_2} p^{m_p/10}.
\]
Enfin, si $p \in I_3$, alors $\floor{m_p/2} \geq m_p/3$, et la proposition~\ref{qr,padic} permet de minorer
\[
\prod_{p \in I_3} \dim V_p \geq \prod_{p \in I_3} p^{m_p/3 - k_p - o_p}.
\]
Mis bout à bout, cela donne
\[
\dim V = \prod_{p|q} \dim V_p \geq \frac{q^{1/10}}{C},
\]
avec $C = (\prod_{p \in I_1} p^{1/2})(\prod_{p\in I_3} p^{k_p + o_p})$.
%Avec la borne triviale $[\Omega:\Omega_q] \leq q^{d^2}$, cela démontre la propriété quasi-aléatoire de $\Omega$ par rapport à la famille de ses sous-groupe de congruences.
\end{proof}

\noindent\textbf{Remerciements.}
Nous remercions Emmanuel Breuillard et Péter Varjú pour leur encouragement à rédiger en détail les démonstrations présentées ici.
%, et Yves Benoist pour plusieurs explications éclairantes au sujet des groupes algébriques sur les corps finis.

\smallskip

W.H. est supporté par ERC 2020 grant HomDyn (grant no.~833423) à Hebrew University of Jerusalem et par KIAS Individual Grant (no.~MG080401) à Korea Institute for Advanced Study.

\bibliographystyle{plain}
\bibliography{bib_sg}

\bigskip

{\small
\noindent Korea Institute for Advanced Study, Seoul 02455, Republic of Korea.\\
\noindent CNRS -- Université Paris 13, LAGA, 93430 Villetaneuse, France.
}

\end{document}